\documentclass[12pt]{amsart}

\usepackage{amssymb,amsmath,amsthm}  
\usepackage{graphicx}
\usepackage{bm}
\usepackage{enumerate}
\usepackage{enumitem}
\usepackage{braket}
\usepackage{amsaddr}
\usepackage{amscd}
\usepackage[dvipsnames]{xcolor}
\usepackage[mathscr]{eucal}
\usepackage{epstopdf}
\usepackage[hidelinks]{hyperref}
\usepackage{subcaption}
\usepackage{cleveref}
\usepackage{relsize}
\hypersetup{
  colorlinks   = true, 
  urlcolor     = PineGreen, 
  linkcolor    = PineGreen, 
  citecolor   = PineGreen 
}

 \allowdisplaybreaks
   
\numberwithin{equation}{section}
\newtheorem{theorem}{Theorem}[section] 

\newtheorem{lemma}[theorem]{Lemma}

\newtheorem{remark}[theorem]{Remark}
\newtheorem{proposition}[theorem]{Proposition}
\newcommand\norm[1]{\left\lVert#1\right\rVert}
\newcommand{\tu}{\textup}

\newcommand{\bfa}[1]{\boldsymbol{#1}} 			%

\newcommand{\e}{\epsilon}

\newcommand{\Div}{\text{Div}}     				%
\newcommand{\ddiv}{\text{div}}     				%
\usepackage[numbers]{natbib}
\usepackage[latin9]{inputenc}
\usepackage[letterpaper]{geometry}
\usepackage{pifont}
\usepackage{float}
\usepackage{color}
\usepackage{adjustbox}
\usepackage{diagbox}
\usepackage{multirow}
\usepackage{makecell}
\usepackage[makeroom]{cancel}
\usepackage{array}
\usepackage{mathrsfs}
\usepackage{tikz}
\usepackage{comment}
\usepackage{csquotes}

\graphicspath{{./Figures/}}

\usepackage{amsfonts}

\usepackage{color}
\usepackage{adjustbox}
\usepackage{diagbox}
\definecolor{black}{rgb}{0,0,0}

\definecolor{red}{rgb}{1,0,0}

\definecolor{blue}{rgb}{0,0,1}

\newcommand{\di}[1]{{\color{blue}{#1}}}
\newcommand{\tm}[1]{{\color{red}{#1}}}

\numberwithin{equation}{section}

\renewcommand{\div}{\mathop{\rm div}\nolimits}

\newcommand{\dx}{ \mathrm{d}x}

\newcommand{\dt}{ \mathrm{d}t}
\newcommand{\da}{ \mathrm{d}a}

\newcommand{\beq}{\begin{equation}}
\newcommand{\eeq}{\end{equation}}
\newcommand{\beqq}{\begin{equation*}}
\newcommand{\eeqq}{\end{equation*}}
\newcommand{\beqas}{\begin{eqnarray*}}
\newcommand{\eeqas}{\end{eqnarray*}}
\newcommand{\bsp}{\begin{split}}
\newcommand{\eesp}{\end{split}}

\makeatother

\usepackage[english]{babel}

\date{\today}

\title[G\MakeLowercase{eneralized multiscale finite element method for a nonlinear elastic strain-limiting} C\MakeLowercase{osserat model}]
{\textsf{\LARGE G\MakeLowercase{eneralized multiscale finite element method for a nonlinear elastic strain-limiting} C\MakeLowercase{osserat model}}}

\author[D\MakeLowercase{mitry} A\MakeLowercase{mmosov}, T\MakeLowercase{ina} M\MakeLowercase{ai}, J\MakeLowercase{uan} G\MakeLowercase{alvis}]{D\MakeLowercase{mitry} A\MakeLowercase{mmosov$^{a}$}, T\MakeLowercase{ina} M\MakeLowercase{ai$^{b,c,*}$}, J\MakeLowercase{uan} G\MakeLowercase{alvis$^{d}$}}

\begin{document}

%


\maketitle

\begin{abstract}
    For nonlinear Cosserat elasticity, we consider multiscale methods in this paper.  In particular, we explore the generalized multiscale finite element method (GMsFEM) to solve an isotropic Cosserat problem with strain-limiting property (ensuring bounded linearized strains even under high stresses).  Such strain-limiting Cosserat model can find potential applications in solids and biological fibers.   
    However, Cosserat media with naturally rotational degrees of freedom, nonlinear constitutive relations, high contrast, and heterogeneities may produce challenging multiscale characteristics in the solution, and upscaling by multiscale methods is necessary.  
    Therefore, we utilize the offline and residual-based online (adaptive or uniform) GMsFEM in this context while handling the nonlinearity by Picard iteration.
    Through various two-dimensional experiments (for perforated, composite, and stochastically heterogeneous media with small and big strain-limiting parameters), our numerical results show the approaches' convergence, efficiency, and robustness.  In addition, these results demonstrate that such approaches provide good accuracy, the online GMsFEM gives more accurate solutions than the offline one, and the online adaptive strategy has similar
accuracy to the uniform one but with fewer degrees of freedom.  
\end{abstract}




\vfill
\noindent Contact:



\noindent 
\textit{Dmitry Ammosov}; $^a$Laboratory of Computational Technologies for Modeling Multiphysical and Multiscale Permafrost Processes, North-Eastern Federal University, Yakutsk, Republic of Sakha (Yakutia), Russia, 677980; \texttt{dmitryammosov@gmail.com}\\

\noindent
$^*$Corresponding author: \textit{Tina Mai}; $^b$Institute of Research and Development, Duy Tan University, Da Nang, 550000, Vietnam; $^c$Faculty of Natural Sciences, Duy Tan University, Da Nang, 550000, Vietnam; \texttt{maitina@duytan.edu.vn}\\

\noindent
\textit{Juan Galvis}; $^d$Departamento de Matem\'aticas, Universidad Nacional de Colombia, Carrera 45 No. 26-85, Edificio Uriel Guti\'errez, Bogot\'a D.C., Colombia;\\ \noindent \texttt{jcgalvisa@unal.edu.co}\\



\newpage


\section{Introduction}

The present paper considers a nonlinear strain-limiting Cosserat elasticity model in heterogeneous media. This problem is inspired by the striking progress of investigating nonlinear responses of materials based on the recently developed strain-limiting theory \cite{KRR-MMS2011a}. For special Cosserat rods, \cite{raj23cr1d} is the first work that examines a particular set of strain-limiting constitutive relations.  Our paper extends this modeling of Cosserat rods to two-dimensional (2D) Cosserat media in the strain-limiting setting.  The Cosserat elasticity enables us to describe materials with rotating internal degrees of freedom; and in the strain-limiting theory, the linearized strains of these materials stay in a certain range, even for very high stresses.
Some potential applications that come to mind are for modeling solids and biological fibers using uniform, hyperelastic, nonlinear strain-limiting Cosserat materials \cite{raj23cr1d}. 

Regarding the  strain-limiting phenomenon, it is a special subclass of the implicit constitutive theory, which provides a useful foundation for developing nonlinear, infinitesimal strain theories for elastic-like (non-dissipative) material behavior, as stated by Rajagopal in \cite{Raji03,KRR-ZAMP2007,KRR-MMS2011a,KRR-MMS2011b}. It is significant to note that this situation differs from classical Cauchy and Green elasticity modeling methodologies, which provide conventional linear models when faced with infinitesimal strains. Furthermore, the implicit constitutive theory offers a robust theoretical framework for simulating fluid and solid mechanics in a variety of applications, including engineering, chemistry, and physics.  Here, we focus in particular on a unique subclass of the implicit constitutive theory in solids called the strain-limiting theory \cite{KRR-MMS2011a}, which guarantees that the linearized strain will remain confined even in the presence of extremely high stress (while in the conventional linear model, the stress increases when the strain increases and vice versa).  Because of this, the strain-limiting theory is useful for explaining the behavior of fracture, brittle materials close to crack tips or notches (such as crystals), or intense loads on the material's boundary or within its body. No matter how small the gradient of the displacement (and hence, infinitesimal strain) is, both scenarios result in a concentration of stress. These materials are capable of withstanding infinite stresses and are resistant to breaking due to the boundedness of strains, according to \cite{KRR-MMS2011a}.

Even though it is well-known that several processes have been described by classical elasticity, more intricate models are required in many situations. For materials with rotational internal degrees of freedom, their mechanical behaviors have frequently been described in terms of Cosserat media (outside of the strain-limiting framework). Granular media,  liquid crystals, and cellular solids are a few examples of such materials \cite{Lee1973,besdo1985inelastic, ONCK2002717,li2016mixed}. They all contribute significantly to various branches of research and engineering \cite{LAGERWALL20121387, de1999granular, gibson_2003}.
However, nonlinear Cosserat media often contain heterogeneities, as a source of theoretical and computational challenges, which affect the performance of discretizations and solvers for numerical solutions. Additionally, in our context, heterogeneities may present in the strain-limiting parameters. 

To avoid costly, in-depth numerical simulations for Cosserat media, upscaling them has been looked into by many studies \cite{ONCK2002717, kouznetsova2004multi, CHANG20053773, LI2010291, LI2011362}, including earlier works on homogenization with regard to scale separation such as spatial periodicity\cite{forest1999estimating, yer2001direct}.
In these earlier homogenization works, the homogenization procedure depends on the hierarchy of characteristic lengths such as the heterogeneity (or unit cell) size $l_h\,,$ the internal Cosserat length $l_{i}\,,$ and the structure (global domain) size $L_d\,.$ One can obtain different homogenized media based on the ratios of these characteristic values. That is, one gets an elastic Cauchy medium in the case $l_h$
is comparable to $l_{i}$, and they are both much more minor than $L_d$. Whereas, if $l_i$ is comparable to $L_d\,,$ 
one obtains a Cosserat elastic medium. In the literature, all this analysis was carried out for the linear case of Cosserat elasticity.

Now, since our paper deals with no scale separation and avoids formal homogenization procedure, we develop efficient numerical approaches for solving the heterogeneous nonlinear Cosserat problem based on the generalized multiscale finite element method (GMsFEM) \cite{G1}.
The GMsFEM is currently a notable approach for numerically upscaling partial differential equations with high-contrast multiscale physical parameters, especially for heterogeneous media \cite{yeb2}.
This methodology allows us to overcome the limitations of analytical homogenization, such as the need for scale separation (for instance, spatial periodicity) and consideration of characteristic length relations (for example, Cosserat elasticity). The GMsFEM has been successfully used in various applications, including linear heat transfer, fluid flow in porous media, linear and nonlinear classical elasticity, linear Cosserat elasticity, parabolic and hyperbolic equations, and variational inequalities, among others \cite{vasilyeva2019multiscale, gle, gne, cosserat2022, mcl,rtt21, Spiridonov2019, contreras2023exponential, gao2015generalized}. 
The primary goal of the GMsFEM is to create coarse-scale multiscale basis functions by constructing local snapshot spaces (consisting of some local solutions) then carrying out local spectral decomposition (via analysis) within these snapshot spaces.  Therefore, through the coarse-grid multiscale basis functions, the generated eigenfunctions can convey the local features to the global ones.  After that, solution approximations are achieved by solving a coarse-scale matrix system in the multiscale space (formed by the multiscale basis functions), using the Galerkin method or some other projection operators.

Applying the GMsFEM, the nonlinear strain-limiting elasticity problem was handled in \cite{gne}, then the linear Cosserat problem was studied in \cite{cosserat2022}. Our paper will extend \cite{cosserat2022} to nonlinear isotropic heterogeneous Cosserat media \cite{isohe}, in the strain-limiting setting \cite{KRR-MMS2011a, A-Mai-Walton, B-Mai-Walton, gne, cemnlporo}, through the concepts of Picard iteration \cite{gne} combining with GMsFEM. In this way, the GMsFEM can tackle the difficulties from multiple scales, heterogeneities, and high contrast presented in the media.  In particular, during every Picard iteration, either the standard offline GMsFEM or the residual-based online adaptive GMsFEM will be used.  In the former approach (offline GMsFEM), we build multiscale basis functions 
and then use them
to solve the problem on a coarse grid. 
Regarding terminology, we generally refer to online spaces when the multiscale basis functions rely on online data, like the parameter or right-hand side \cite{yeb2}.
The latter method (online adaptive GMsFEM) is based on the residual-driven online GMsFEM \cite{chungres}, where we examine the construction of residual-based online multiscale basis functions in connection with adaptivity \cite{chungres,gne}, 
 meaning the insertion of online basis functions into a few chosen areas.  Here, ``adaptive'' implies either ``adaptive'' or ``uniform'' (see Section \ref{online} for more details).
 Being essential to lower the cost of online multiscale basis calculations, adaptivity is a critical step in obtaining an effective local multiscale model reduction.  That is, an adaptive algorithm permits adding more basis functions in neighborhoods having increased complexity without the need for a priori information. 
 
 Our plan is to update online basis functions after a certain number of Picard iterations if the relative change of the coefficient is greater than a specified fixed tolerance.  When the desired error is obtained, the updating process stops, and the new basis functions remain unchanged for the next Picard iterations until they must be computed again.  

We investigate 2D models in various heterogeneous media (perforated, composite, and stochastically heterogeneous) to test the proposed multiscale approaches with small and large strain-limiting parameters. To compute finite element reference solutions, standard FEM piecewise linear basis functions are employed on a computational fine grid. We consider different numbers of (offline and online) multiscale basis functions per coarse node together with different update frequencies of online basis functions. Numerical results show that the proposed multiscale approaches can significantly reduce the number of degrees of freedom (basis functions) while maintaining high accuracy. Provided enough number of initial basis functions in the offline space, the use of online (adaptive or uniform) basis functions improves the accuracy of the multiscale solution.  Moreover, adaptively adding online basis functions substantially lowers the GMsFEM's computing cost, speeds up convergence, and minimizes error. Also, the online adaptive strategy achieves comparable accuracy to the uniform strategy with fewer degrees of freedom, thus lowering the computing
cost of the GMsFEM.

This paper has the following structure. Section \ref{pre} introduces preliminaries and notations. In Section \ref{formulate}, we develop a nonlinear strain-limiting Cosserat elasticity model. Section \ref{ffem} outlines the finite element approximation of the problem on a fine grid and Picard iterative process for linearization. Section \ref{approx} describes our offline multiscale approach following the GMsFEM. In Section \ref{online}, we present the residual-based online adaptive (meaning adaptive or uniform) GMsFEM. Section \ref{num_results} shows the numerical results. In Section \ref{conclusion}, the work is summarized. 


\section{Preliminaries}\label{pre}
Consider a nonlinear Cosserat elastic medium $\Omega\,,$ which is also an open, convex, bounded, Lipschitz, simply connected computational domain of $\mathbb{R}^d\,.$  While we concentrate on the case $d=2$ in this study, the approach is easily extended to $d=3\,.$  For the preliminary work, we refer the readers to \cite{C-G-K, gne}. The symbol $\nabla$ represents a spatial gradient. Latin indices vary in the set $\{1,2,3\}$.  Italic capitals (e.g., $f$) are used to represent functions; bold letters (e.g., $\bfa{v}$ and $\bfa{T}$) are employed to denote vector fields, matrix fields, third- and fourth-order tensors over $\Omega\,.$ The space of functions, vector fields, and matrix fields
 over $\Omega$ are represented, respectively, by italic capitals (e.g., $L^2(\Omega)$), 
boldface Roman capitals (e.g., $\bfa{V}$), 
and special Roman capitals (e.g., $\mathbb{S}$).  To denote a space of symmetric matrix fields, we utilize subscript $s$ added to a specific Roman capital.  

The points $\bfa{x}$ have coordinates indicated by $(x_1,x_2) \in \Omega \subset \mathbb{R}^2\,.$  The symbol for the Kronecker is $\delta_{ij}\,.$ In a Cartesian system, an orthonormal basis is denoted by the vectors $(\bfa{e}_i)_{i=1,2,3}$.  In the tensor index notation for partial derivatives, the subscript $,j$ stands for differentiation with regard to $x_j\,.$
A second-order mixed partial derivative is defined as
\begin{equation}\label{mixed_der}
\frac{\partial^2 f}{\partial y \partial x} = \frac{\partial}{\partial y} \left( \frac{\partial f}{\partial x} \right) = (f'_x)'_y = f''_{xy} = \partial_{yx} f = \partial_y \partial_x f\,.
\end{equation}
Any repeated index adheres to the convention of Einstein summation.

The gradient ($\nabla$) of a scalar-valued function and a vector field are respectively defined as follows \cite{div1}:
\begin{equation}\label{graddiv}
\nabla a = a_{,i} \bfa{e}_i\,, \quad \nabla \bfa{a} = a_{i,j} \bfa{e}_i \otimes \bfa{e}_j\,.
\end{equation}
The divergence ($\div$ or $\nabla \cdot$) of vector field and second-order tensor-valued function are
\begin{equation}\label{graddiv2}
 \div(\bfa{a}) = \nabla \cdot \bfa{a} = a_{i,i}\,, \quad \nabla \cdot \bfa{A} = a_{ji,j} \bfa{e}_i\,,
\end{equation}
respectively. 
%

We make the following remark on the divergence of the second-order tensor.  In a Cartesian coordinate system, using notation from \eqref{graddiv}, the following relations for a second-order tensor field $\bfa{S}$ hold \cite{div1}:
\begin{equation}\label{divdef}
	\nabla \cdot \bfa{S} = \frac{\partial S_{ji}}{\partial x_j}  \bfa{e}_i = S_{ji,j}  \bfa{e}_i\,.  
\end{equation}
The result is a contravariant (column) vector, 
and note that
\[\nabla \cdot \bfa{S} \neq  \nabla \cdot \bfa{S}^{\tu{T}} = \div(\bfa{S})\,,\]
where
\begin{equation}\label{div2row}
	\nabla \cdot \bfa{S}^{\tu{T}} = \div(\bfa{S})=\frac{\partial S_{ij}}{\partial x_j}  \bfa{e}_i = S_{ij,j}  \bfa{e}_i\,.
\end{equation}
Only for a symmetric second-order tensor $S_{ij} = S_{ji}\,,$ it occurs that $\div(\bfa{S}) = \nabla \cdot \bfa{S}$ as the divergence \cite{div2}. 
Therefore, in mechanical equations where tensor symmetry is assumed, the two definitions \eqref{divdef}--\eqref{div2row} (and corresponding symbols $\div$ and $\nabla \cdot$) are both valid.  Nevertheless, when tensor symmetry is not assumed, the definition \eqref{div2row} is not consistent with curvilinear expression.

 Let $\bfa{\e} = (\e_{ijk})$
be the third-order antisymmetric Levi-Civita (permutation) tensor.  Here, $\e_{ijk}$ is the sign of the permutation $(i,j,k)\,,$ that is,
when the sequence of indices $(i,j,k)$ 
is an even permutation of the sequence $(1,2,3)$ then $\e_{ijk}=1$; whereas for an odd permutation, $\e_{ijk}=-1$; if any two indices coincide then $\e_{ijk} = 0\,.$  For example, 
\[\e_{123}=\e_{231}=\e_{312}=-\e_{213}= -\e_{132}=-\e_{321}=1\,, \quad \e_{232}=-\e_{232}=0\,.\]



We consider the classical Lebesgue and Sobolev spaces. In particular, the spaces $V: = H^1(\Omega)$ and $\bfa{V}: = \bfa{H}^1(\Omega) = \bfa{W}^{1,2}(\Omega)$ are used. Also, we denote $V_0: = H^1_0(\Omega)$ and $\bfa{V}_0: = \bfa{H}_0^1(\Omega) = \bfa{W}_0^{1,2}(\Omega)\,.$  The dual norm to $\| \cdot \|_{\bfa{H}_0^1(\Omega)}$ is $\| \cdot \|_{\bfa{H}^{-1}(\Omega)}$.  The component representation of vector $\bfa{v} \in \bfa{V}$ is $(v_1, v_2, v_3)$ or $[v_i]$ for $i=1,2,3\,.$  Let $| \bfa{v}|$ denote the Euclidean norm of the 3-component vector-valued function 
$ \bfa{v}$; and $| \nabla \bfa{v}|$ represents the Frobenius norm of the $3 \times 3$ matrix $\nabla \bfa{v}\,,$ where $|\bfa{A}| = \sqrt{\bfa{A} : \bfa{A}}\,,$ with the Frobenius inner product defined as $\langle \bfa{A},\bfa{B}\rangle_{\tu{F}} = \bfa{A} : \bfa{B} = \tu{Tr}(\bfa{A}^{\tu{T}}\bfa{B}) = \sum_{i,j} a_{ij} b_{ij}\,.$

\section{A nonlinear strain-limiting Cosserat model}\label{formulate}

The Cosserat theory of elasticity \cite{polish1, polish2} is also known as micropolar elasticity, or the micropolar theory of elasticity, or micropolar continuum mechanics. 
Its main equations hold for all points of the body $\Omega$ and are equilibrium equations, kinematic relations, and constitutive equations.   An important feature of this paper is that we consider Cosserat's theory of elasticity 
within the strain-limiting nonlinear elasticity setting \cite{KRR-MMS2011a, A-Mai-Walton, B-Mai-Walton, gne, cemnlporo}. 
In this section, we first review the general Cosserat model and then introduce the strain-limiting Cosserat problem in dimension two. The strain-limiting elasticity setting can be found in \cite{gne}.

\subsection{General elastic Cosserat model}

Let $\bfa{u}=(u_1, u_2, u_3)$ be the displacement vector and $\bfa{\phi} = (\phi_1, \phi_2, \phi_3)$ be the rotation vector.
We use the symbols $\bfa{\tau} = (\tau_{ij})$ for the force stress tensor and $\bfa{m}= (m_{ij})$ for the couple stress (or moment stress) tensor. 
The fact that these stress tensors are asymmetric (not symmetric)
follows from the equilibrium equations for the body element.  Let $\bfa{\gamma} = (\gamma_{ij})$ be the
 strain tensor and $\bfa{\chi} = (\chi_{ij})$ be the 
 torsion (or wryness or curvature or curvature-twist or bend-twist) tensor, which are the asymmetric deformation measures.   The given body force vector is represented by $\bfa{f}\,,$  
and the body couple moment vector is denoted by $\bfa{g}\,.$  
We now write the fundamental equations for this paper. The Cosserat body is taken to have no heat sources or heat flow from its exterior, according to \cite{polish1, polish2}. 
Furthermore, because of heat conduction, it is assumed that heat exchanges slowly during deformation.  As a result, the entire Cosserat body may be thought of as thermally insulated, and the thermodynamic procedure can be seen as adiabatic (heat neither enters nor exits the system). According to the Cosserat theory of elasticity, there are micro-moments at every point of the continuum.

The strain-displacement and torsion-rotation (or curvature-rotation) relations are of the following forms \cite{ji1966sandru, cosserat2022}: 
%
\begin{align}
	\gamma_{ij} &= u_{j,i} + \e_{jik} \phi_k  \,,\label{strain-displace}\\
    \chi_{ij} &=\phi_{j,i}\,. \label{torsion-rotation}
\end{align}
Thanks to \cite{ji1999deform}, we can assume that the rotation satisfies $\phi_{i,j} = \phi_{j,i}\,.$


In the case of static Cosserat body, equilibrium (or balance) equations 
with body forces and body moments respectively express the balance (or conservation) of forces and moments as the following direct tensor representation \cite{ji2017directE, yer2001direct,div1}: 
\begin{align}
	\nabla \cdot \bfa{\tau} + \bfa{f} &= 0 \,, \label{linearmoment2}\\
	\bfa{\e} : \bfa{\tau} + \nabla \cdot \bfa{m} + \bfa{g} &=0\,, \label{angularmoment2}
\end{align}
where the double inner product or double dot product has the form
\[\bfa{\e} : \bfa{\tau} = \e_{irs} \tau_{rs} \bfa{e}_i\,,\]
and recall that $\bfa{\e}$ is the Levi-Civita tensor (permutation tensor).  In a rectangular (Cartesian) coordinate system, Eqs.\ \eqref{linearmoment2}--\eqref{angularmoment2} have the following forms \cite{ji1967deformEringen,cosserat2022}: 
\begin{align}
	\tau_{ij,i} + f_j &= 0 \,, \label{linearmoment}\\
	\e_{jrs}\tau_{rs} + m_{ij,i} + g_j &=0\,, \label{angularmoment}
\end{align}
where $\e_{jrs}$ is the sign of the permutation $(j,r,s)\,.$  It should be noted that the final equation (involving couple stress) was derived laboriously from the integral equation of the theory 
of micropolar elasticity \cite{div1}.

\subsection{Elastic Cosserat  planar model}\label{plane}
We investigate a two-dimensional Cosserat problem motivated by \cite{cosserat2022, energymodes}. In this case, a three-dimensional configuration is restricted to the $x,y$-plane, having rotations only about the $z$-axis \cite{neff17plane}. That is, 
\begin{align}\label{planar}
\begin{split}
\bfa{u} &= (u_1(x_1,x_2), u_2(x_1,x_2),0) \,,\\
\bfa{\phi} &= (0,0, \Phi(x_1,x_2))\,,
\end{split}
\end{align}
where the vector $\bfa{\phi}$ identifies the axis of rotation. Then, 
\eqref{strain-displace} and \eqref{torsion-rotation} become
\begin{align}\label{kinematic}
\begin{split}
\gamma_{ij} &= \frac{\partial u_j}{\partial{x_i}} + \e_{ji} \Phi \,, \\
\chi_{i3} &= \frac{\partial \Phi}{\partial x_i}\,.
\end{split}
\end{align}
Herein, $i,j = 1,2\,.$ Also, $\bfa{\e} = (\e_{ji})$ is the second-order Levi-Civita tensor (permutation tensor).  Furthermore, $\bfa{\gamma}$ and $\bfa{\tau}$ are second-order tensors while $\bfa{\chi},\bfa{m}$ are first-order tensors in the forms
\begin{align}\label{kiforms}
\bfa{\gamma} =
\begin{bmatrix}
\gamma_{11} & \gamma_{12}\\
\gamma_{21} & \gamma_{22}
\end{bmatrix}\,,
\quad 
\bfa{\tau} =
\begin{bmatrix}
\tau_{11} & \tau_{12}\\
\tau_{21} & \tau_{22}
\end{bmatrix}\,,
\quad
\bfa{\chi} =
\begin{bmatrix}
\chi_{13}\\
\chi_{23}
\end{bmatrix}\,,
\quad
\bfa{m}=
\begin{bmatrix}
m_{13}\\
m_{23}
\end{bmatrix}\,.
\end{align}
They are related to each other by the constitutive equations \eqref{rel1a}--\eqref{rel1b} and \eqref{rel2a}--\eqref{rel2b} below.
\subsection{Nonlinear planar strain-limiting Cosserat elasticity energy}\label{energy}
We refer to \cite{gne} 
for a review of the strain-limiting framework.  To describe our nonlinear strain-limiting Cosserat model, let $\alpha, \beta, \xi$ be positive parameters.  Using the notation from \eqref{kiforms}, we consider the quadratic forms 
\begin{align}
Q(\bfa{\chi},\bfa{\gamma}) &= \beta(\bfa{x})^2 \Big(\alpha^2(\chi^2_{13} + \chi^2_{23}) + \xi^2 (\gamma_{11}^2 +  \gamma_{21}^2 + \gamma_{12}^2 + \gamma_{22}^2)\Big) \,, \label{q}\\
Q^*(\bfa{m},\bfa{\tau}) &=  \beta(\bfa{x})^2 \left(\frac{1}{\alpha^2} (m_{13}^2 + m_{23}^2) + \frac{1}{\xi^2} (\tau_{11}^2 + \tau_{21}^2 + \tau_{12}^2  + \tau_{22}^2)\right) \,, \label{qs}
\end{align}
which are positive semi-definite.  For $i,j = 1,2\,,$ utilizing the notation from \eqref{kinematic}, we assume that the constitutive relations between the strains and stresses are given by
\begin{align}
\chi_{i3} &= \frac{1}{1+ \sqrt{Q^*(\bfa{m},\bfa{\tau})}} \frac{1}{\alpha^2} m_{i3} \,, \label{rel1a}\\
\gamma_{ij} &= \frac{1}{1+ \sqrt{Q^*(\bfa{m},\bfa{\tau})}} \frac{1}{\xi^2} \tau_{ij}\,. \label{rel1b}
\end{align}
It is observable that the model is strain-limiting in the sense that $\bfa{\chi} \in \mathbb{R}^2\,,\bfa{\gamma} \in \mathbb{R}^{2\times2}$ 
from \eqref{rel1a}--\eqref{rel1b} meet the condition
\begin{equation}\label{less1}
Q(\bfa{\chi},\bfa{\gamma}) < 1\,,
\end{equation}
where $Q(\bfa{\chi},\bfa{\gamma})$ is defined in \eqref{q}.
It is possible to derive the constitutive relations \eqref{rel1a} and \eqref{rel1b} from a complementary strain-energy density $W^*(\bfa{m}, \bfa{\tau})$   by taking its gradients \cite{gurtin1, energy4a}. That is, 
\begin{equation}\label{straine}
\bfa{\chi} = 
\begin{bmatrix}
 \chi_{13} \\
\chi_{23}
\end{bmatrix}
= \nabla_{\bfa{m}} W^* (\bfa{m}, \bfa{\tau})\,, \quad 
\bfa{\gamma} =
\begin{bmatrix}
 \gamma_{11} & \gamma_{12} \\
\gamma_{21} & \gamma_{22}
\end{bmatrix}
= \nabla_{\bfa{\tau}} W^* (\bfa{m}, \bfa{\tau})\,.
\end{equation}
Equivalently, their components are partial derivatives
\begin{equation}\label{strainec}
\chi_{i3} = \frac{\partial W^*(\bfa{m}, \bfa{\tau})}{\partial m_{i3}}\,, \quad 
\gamma_{ij} = \frac{\partial W^*(\bfa{m}, \bfa{\tau})}{\partial \tau_{ij}}\,.
\end{equation}
Here,
\begin{equation}\label{ws}
W^*(\bfa{m}, \bfa{\tau}) =  \frac{1}{2} \frac{1}{\beta(\bfa{x})^2} \int_0^{Q^*(\bfa{m}, \bfa{\tau})} \frac{1}{1+ \sqrt{t}} \dt =  \frac{1}{\beta(\bfa{x})^2} \left(\sqrt{Q^*} - \ln \left(1+\sqrt{Q^*}\right) \right) \,.
\end{equation}

\bigskip

The relations \eqref{rel1a} and \eqref{rel1b} can be inverted to give the stresses as functions of the strains (for $i,j=1,2$):
\begin{align}
m_{i3} &= \frac{1}{1- \sqrt{Q(\bfa{\chi}, \bfa{\gamma})}} \alpha^2 \chi_{i3} \,, \label{rel2a}\\
\tau_{ij} &= \frac{1}{1- \sqrt{Q(\bfa{\chi}, \bfa{\gamma})}} \xi^2 \gamma_{ij} \,, \label{rel2b}
\end{align}
for $\bfa{\chi} \in \mathbb{R}^2\,,\bfa{\gamma} \in \mathbb{R}^{2\times2}$ satisfying $Q(\bfa{\chi}, \bfa{\gamma}) <1$ \eqref{less1}.  Because the strains increase to their limiting levels asymptotically in \eqref{rel1a} and \eqref{rel1b}, the constitutive relations can be inverted as a result of this type of behavior (refer to Section \ref{mono} for monotonicity).  We can now see that our model is hyperelastic since the relationships \eqref{rel1a} and \eqref{rel1b} may be derived from a strain-energy density $W(\bfa{\chi}, \bfa{\gamma})$ by using its gradients \cite{gurtin1, energy4a, raj23cr1d} as follows:
\begin{equation}\label{stresse}
\bfa{m} = 
\begin{bmatrix}
 m_{13} \\
m_{23}
\end{bmatrix}
= \nabla_{\bfa{\chi}} W (\bfa{\chi}, \bfa{\gamma})\,, \quad 
\bfa{\tau} =
\begin{bmatrix}
 \tau_{11} & \tau_{12} \\
\tau_{21} & \tau_{22}
\end{bmatrix}
= \nabla_{\bfa{\gamma}} W (\bfa{\chi}, \bfa{\gamma})\,.
\end{equation}
Equivalently (by partial derivatives),
\begin{equation}\label{stressec}
m_{i3} = \frac{\partial W(\bfa{\chi}, \bfa{\gamma})}{\partial \chi_{i3}}\,, \quad 
\tau_{ij} = \frac{\partial W(\bfa{\chi}, \bfa{\gamma})}{\partial \gamma_{ij}}\,.
\end{equation}
Herein,
\begin{equation}\label{we}
W(\bfa{\chi}, \bfa{\gamma}) =  \frac{1}{2} \frac{1}{\beta(\bfa{x})^2} \int_0^{Q(\bfa{\chi}, \bfa{\gamma})} \frac{1}{1 - \sqrt{t}} \dt =  - \frac{1}{\beta(\bfa{x})^2} \left(\sqrt{Q} + \ln \left(1-\sqrt{Q}\right) \right) \,.
\end{equation}

\subsection{Monotonicity}\label{mono}
Note that $\bfa{m}, \bfa{\chi} \in \mathbb{R}^2$ are vectors in our 2D problem, and the matrices $\bfa{\gamma}, \bfa{\tau} \in \mathbb{R}^{2\times2}$ can be viewed as vectors by the isomorphism between $\mathbb{R}^{r \times s}$ and $\mathbb{R}^{rs}$ via the vectorization (here $r=s=2$).  
It is well-known that $\bfa{m}$ and $\bfa{\tau}$ are monotone if and only if for each $\bfa{\chi},\bfa{\gamma}\,,$ the Jacobian $\nabla_{\bfa{\chi}} \bfa{m}(\bfa{\chi},\bfa{\gamma}), \nabla_{\bfa{\gamma}} \bfa{\tau} (\bfa{\chi},\bfa{\gamma})$ are positive semi-definite (see the book \cite[ Proposition 12.3]{rockafellar2009variational} for this statement and proof).
Using this remark and the positive semi-definite property, we now establish a monotonicity feature of our model that is mathematically appealing and useful.  
The following proposition implies monotonicity, that is, entry-wisely, increasing the couple stress $\bfa{m}$ and force stress $\bfa{\tau}$ corresponds to increasing the torsion $\bfa{\chi}$ and strain $\bfa{\gamma}\,,$ respectively.  

\begin{proposition}\label{prop}
For any $\bfa{\chi} \in \mathbb{R}^2$ and $\bfa{\gamma} \in \mathbb{R}^{2 \times 2}$ in the form of vectorization $\bfa{\gamma} \in \mathbb{R}^{4}$ such that $Q(\bfa{\chi}, \bfa{\gamma}) <1\,,$ it holds that the Hessian of the strain-energy density
\begin{align}\label{psd}
\begin{split}
\mathbf{H} (W(\bfa{\chi}, \bfa{\gamma})) &= \mathbf{D}^2 (W(\bfa{\chi}, \bfa{\gamma}))^{\tu{T}} = (\mathbf{J} (\nabla W(\bfa{\chi}, \bfa{\gamma})))^{\tu{T}}= 
\left(\begin{bmatrix}
\dfrac{\partial \bfa{m}(\bfa{\chi},\bfa{\gamma})}{\partial \bfa{\chi}} & \dfrac{\partial \bfa{m}(\bfa{\chi},\bfa{\gamma})}{\partial \bfa{\gamma}}\\\\
\dfrac{\partial \bfa{\tau}(\bfa{\chi},\bfa{\gamma})}{\partial \bfa{\chi}} & \dfrac{\partial \bfa{\tau}(\bfa{\chi},\bfa{\gamma})}{\partial \bfa{\gamma}}
\end{bmatrix}
\right)^{\tu{T}}\nonumber \\
&=
\begin{bmatrix}
 \dfrac{\partial^2 W(\bfa{\chi},\bfa{\gamma})}{\partial \bfa{\chi}^2} & \dfrac{\partial^2 W(\bfa{\chi},\bfa{\gamma})}{ \partial  \bfa{\chi} \partial 
 \bfa{\gamma} }\\\\
\dfrac{\partial^2 W(\bfa{\chi},\bfa{\gamma})}{ \partial \bfa{\gamma} \partial \bfa{\chi}} & \dfrac{\partial^2 W(\bfa{\chi},\bfa{\gamma})}{\partial \bfa{\gamma}^2}
\end{bmatrix}
\end{split}
\end{align}
is positive semi-definite.
\end{proposition}
\begin{proof}
Denote $\beta:=\beta(\bfa{x})\,, Q:= Q(\bfa{\chi},\bfa{\gamma})$ from \eqref{q}, and
\[P:=\frac{1}{\sqrt{Q}(1-\sqrt{Q})^2}\,, \qquad R:= \frac{1}{1-\sqrt{Q}}\,.\]
We first observe that
\begin{align*}
\dfrac{\partial^2 W(\bfa{\chi},\bfa{\gamma})}{\partial \bfa{\chi}^2} 
&= \begin{bmatrix}
   \dfrac{\partial^2 W}{\partial \chi_{13} \partial \chi_{13}} & \dfrac{\partial^2 W}{\partial \chi_{23} \partial \chi_{13}}\\\\
   \dfrac{\partial^2 W}{\partial \chi_{13} \partial \chi_{23}} & \dfrac{\partial^2 W}{\partial \chi_{23} \partial \chi_{23}}
\end{bmatrix}
&= 
\begin{bmatrix}
      \alpha^4 \beta^2 \chi^2_{13} P + \alpha^2 R \quad &   \alpha^4 \beta^2 \chi_{13} \chi_{23} P\\\\
     \alpha^4 \beta^2 \chi_{23} \chi_{13}  P 
    \quad &  \alpha^4 \beta^2 \chi^2_{23} P + \alpha^2 R
\end{bmatrix}\,,
\end{align*}
\begin{align*}
\dfrac{\partial^2 W(\bfa{\chi},\bfa{\gamma})}{ \partial \bfa{\chi} \partial \bfa{\gamma}}
&= \begin{bmatrix}
   \dfrac{\partial^2 W}{\partial \chi_{13} \partial \gamma_{11} } & \dfrac{\partial^2 W}{ \partial \chi_{23} \partial \gamma_{11}}\\\\
   \dfrac{\partial^2 W}{\partial \chi_{13} \partial \gamma_{12} } & \dfrac{\partial^2 W}{\partial \chi_{23} \partial \gamma_{12} } \\\\
   \dfrac{\partial^2 W}{\partial \chi_{13} \partial \gamma_{21}} & \dfrac{\partial^2 W}{\partial \chi_{23} \partial \gamma_{21}}\\\\
   \dfrac{\partial^2 W}{\partial\chi_{13} \partial \gamma_{22}} & \dfrac{\partial^2 W}{\partial \chi_{23} \partial \gamma_{22} }
\end{bmatrix}
&= 
\begin{bmatrix}
     \alpha^2 \beta^2  \xi^2 \gamma_{11} \chi_{13} P  \quad 
     & \alpha^2 \beta^2  \xi^2 \gamma_{11} \chi_{23} P\\ \\
    \alpha^2 \beta^2  \xi^2  \gamma_{12} \chi_{13} P 
    \quad & \alpha^2 \beta^2 \xi^2  \gamma_{12} \chi_{23} P\\ \\
    \alpha^2 \beta^2  \xi^2  \gamma_{21} \chi_{13} P 
    \quad & \alpha^2 \beta^2  \xi^2  \gamma_{21} \chi_{23} P\\ \\
    \alpha^2 \beta^2  \xi^2  \gamma_{22} \chi_{13} P 
    \quad & \alpha^2 \beta^2 \xi^2  \gamma_{22} \chi_{23} P
\end{bmatrix}\,,
\end{align*}
\begin{align*}
\dfrac{\partial^2 W(\bfa{\chi},\bfa{\gamma})}{ \partial \bfa{\gamma} \partial \bfa{\chi}} 
&= \begin{bmatrix}
   \dfrac{\partial^2 W}{\partial \gamma_{11} \partial \chi_{13} } 
   & \dfrac{\partial^2 W}{\partial \gamma_{12} \partial \chi_{13} }
   & \dfrac{\partial^2 W}{\partial \gamma_{21} \partial \chi_{13} } 
   & \dfrac{\partial^2 W}{\partial \gamma_{22} \partial \chi_{13} } \\\\
   \dfrac{\partial^2 W}{\partial \gamma_{11} \partial \chi_{23} } 
   & \dfrac{\partial^2 W}{\partial \gamma_{12} \partial \chi_{23} }
   & \dfrac{\partial^2 W}{\partial \gamma_{21} \partial \chi_{23} } 
   & \dfrac{\partial^2 W}{\partial \gamma_{22} \partial \chi_{23} }
\end{bmatrix}\\
&= 
\begin{bmatrix}
     \alpha^2 \beta^2  \xi^2 \chi_{13} \gamma_{11} P  \quad &  \alpha^2 \beta^2  \xi^2 \chi_{13} \gamma_{12} P \quad &  \alpha^2 \beta^2  \xi^2 \chi_{13} \gamma_{21} P \quad &  \alpha^2 \beta^2  \xi^2 \chi_{13} \gamma_{22} P\\\\
    \alpha^2 \beta^2  \xi^2 \chi_{23} \gamma_{11} P 
    \quad & \alpha^2 \beta^2  \xi^2 \chi_{23} \gamma_{12} P & \alpha^2 \beta^2  \xi^2 \chi_{23} \gamma_{21} P & \alpha^2 \beta^2  \xi^2 \chi_{23} \gamma_{22} P
\end{bmatrix}\,,
\end{align*}
and
\begin{align*}
\dfrac{\partial^2 W(\bfa{\chi},\bfa{\gamma})}{\partial \bfa{\gamma}^2}
&= \begin{bmatrix}
   \dfrac{\partial^2 W}{\partial \gamma_{11} \partial \gamma_{11} } 
   & \dfrac{\partial^2 W}{ \partial \gamma_{12} \partial \gamma_{11}}
   & \dfrac{\partial^2 W}{ \partial \gamma_{21} \partial \gamma_{11}} 
   & \dfrac{\partial^2 W}{\partial \gamma_{22} \partial \gamma_{11} } \\\\
   \dfrac{\partial^2 W}{\partial \gamma_{11} \partial \gamma_{12} } 
   & \dfrac{\partial^2 W}{ \partial \gamma_{12} \partial \chi_{12}}
   & \dfrac{\partial^2 W}{\partial \gamma_{21} \partial \gamma_{12} } 
   & \dfrac{\partial^2 W}{\partial \gamma_{22} \partial \gamma_{12} }\\\\
   \dfrac{\partial^2 W}{\partial \gamma_{11} \partial \gamma_{21} } 
   & \dfrac{\partial^2 W}{\partial \gamma_{12} \partial \gamma_{21} }
   & \dfrac{\partial^2 W}{\partial \gamma_{21} \partial \gamma_{21} } 
   & \dfrac{\partial^2 W}{\partial \gamma_{22} \partial \gamma_{21} }\\\\
   \dfrac{\partial^2 W}{\partial \gamma_{11} \partial \gamma_{22} } 
   & \dfrac{\partial^2 W}{\partial \gamma_{12} \partial \gamma_{22} }
   & \dfrac{\partial^2 W}{\partial \gamma_{21} \partial \gamma_{22} } 
   & \dfrac{\partial^2 W}{\partial \gamma_{22} \partial \gamma_{22} }
\end{bmatrix}\\
&= 
\begin{bmatrix}
     \beta^2 \xi^4 \gamma^2_{11} P + \xi^2 R \quad &  \beta^2 \xi^4 \gamma_{11} \gamma_{12} P &  \beta^2 \xi^4 \gamma_{11} \gamma_{21} P &  \beta^2 \xi^4 \gamma_{11} \gamma_{22} P\\\\
    \beta^2 \xi^4 \gamma_{12} \gamma_{11} P 
    \quad & \beta^2 \xi^4 \gamma^2_{12} P + \xi^2 R \quad & \beta^2 \xi^4 \gamma_{12} \gamma_{21} P \quad & \beta^2 \xi^4 \gamma_{12} \gamma_{22} P\\\\
    \beta^2 \xi^4  \gamma_{21} \gamma_{11} P 
    \quad & \beta^2 \xi^4 \gamma_{21}\gamma_{12}  P \quad & \beta^2 \xi^4 \gamma^2_{21} P + \xi^2 R \quad & \beta^2 \xi^4 \gamma_{21}\gamma_{22}  P\\\\
    \beta^2 \xi^4 \gamma_{22} \gamma_{11}  P 
    \quad & \beta^2 \xi^4  \gamma_{22} \gamma_{12} P \quad & \beta^2 \xi^4 \gamma_{22} \gamma_{21}  P \quad & \beta^2 \xi^4 \gamma^2_{22} P + \xi^2 R
\end{bmatrix}\,.
\end{align*}

For all vectors 
\[\bfa{a}=(a_1, a_2)^{\tu{T}} \in \mathbb{R}^2\,, \quad \bfa{b} = (b_{11}, b_{12}, b_{21}, b_{22})^{\tu{T}} = \mathbf{vec}\left( 
\begin{bmatrix}
b_{11} & b_{21}\\
b_{12} & b_{22}
\end{bmatrix}
\right)\in \mathbb{R}^4\,,\]
we obtain
\begin{align*}
\begin{split}
& \begin{bmatrix}
\bfa{a} & \bfa{b}
\end{bmatrix}
 \mathbf{J} (\nabla W(\bfa{\chi}, \bfa{\gamma}))^{\tu{T}} 
\begin{bmatrix}
\bfa{a}\\
\bfa{b}
\end{bmatrix}=\frac{1}{\beta(\bfa{x})^2}\frac{1}{1 - \sqrt{Q}} Q(\bfa{a}, \bfa{b}) \\
& \quad +  \frac{\beta(\bfa{x})^2}{\sqrt{Q} \left(1 - \sqrt{Q}\right)^2} \left(\alpha^2 a_1 \chi_{13} + \alpha^2 a_2 \chi_{23} + \xi^2 b_{11} \gamma_{11} 
+ \xi^2 b_{12} \gamma_{12} + \xi^2 b_{21} \gamma_{21} + \xi^2 b_{22} \gamma_{22}\right)^2\,.
\end{split}
\end{align*}
The proof is completed because of the positive semi-definiteness of $Q(\cdot,\cdot)$ defined in \eqref{q}.
\end{proof}

\subsection{Field equations for plane nonlinear strain-limiting Cosserat elasticity}\label{feqs}
In our 2D static case, with the variables $\Phi$ and $\bfa{u} = (u_1,u_2)$ (abbreviation of $\bfa{u} = (u_1,u_2,0)$) in \eqref{planar}, and $\bfa{f} = (f_1,f_2)\,,$ for $i,j = 1,2\,,$ using \eqref{kinematic}, the quadratic form \eqref{q} becomes
\begin{align}\label{qu}
Q(\bfa{\chi},\bfa{\gamma}) = Q(\bfa{u},\Phi) &= \beta(\bfa{x}) ^2\left(\alpha^2 \Big((\Phi_{,1})^2 + (\Phi_{,2})^2\Big) + \xi^2  \sum_{i,j=1,2} (u_{j,i} + \bfa{\e}_{ji}\Phi)^2\right)\,,
\end{align}
where the subscript $,j$ stands for differentiation with respect to  $x_j\,,$
\[\bfa{\e}_{11}= \bfa{\e}_{22}=0\,, \quad \bfa{\e}_{12} =1\,, \quad \bfa{\e}_{21} = -1   \,.\]
In the plane, the balance of forces \eqref{linearmoment} is given by
\begin{align}\label{bof}
\frac{\partial}{\partial x_i}\left(\frac{\xi^2 (u_{j,i}+ \bfa{\e}_{ji}\Phi)}{1 - \sqrt{Q(\bfa{u},\Phi)}} \right) + f_j = 0\,. 
\end{align}
The balance of moments \eqref{angularmoment} in the plane has the form
\begin{align}\label{bom}
\frac{\bfa{\e}_{ij} \xi^2 ( u_{j,i} + \bfa{\e}_{ji} \Phi)}{1 - \sqrt{Q(\bfa{u},\Phi)}} + \frac{\partial}{\partial x_i}\left(\frac{\alpha^2 \Phi_{,i}}{1 - \sqrt{Q(\bfa{u},\Phi)}} \right) + g_3=0\,.
\end{align}
Eqs.~\eqref{bof}--\eqref{bom} are 2D field equations through index notation for plane nonlinear strain-limiting Cosserat elasticity, and these equations are explicitly equivalent to the following system of equations:
\begin{align}
\frac{\partial}{\partial x_1}\left(\frac{\xi^2 u_{1,1}}{1 - \sqrt{Q(\bfa{u},\Phi)}} \right) 
+ \frac{\partial}{\partial x_2}\left(\frac{\xi^2 (u_{1,2}+ \Phi)}{1 - \sqrt{Q(\bfa{u},\Phi)}} \right) + f_1 = 0\,, \label{bof1}\\
\frac{\partial}{\partial x_1}\left(\frac{\xi^2 (u_{2,1} - \Phi)}{1 - \sqrt{Q(\bfa{u},\Phi)}} \right) 
+ \frac{\partial}{\partial x_2}\left(\frac{\xi^2 u_{2,2}}{1 - \sqrt{Q(\bfa{u},\Phi)}} \right) + f_2 = 0\,,\label{bof2}\\
\frac{\xi^2 (u_{2,1} - u_{1,2} - 2\Phi)}{1 - \sqrt{Q(\bfa{u},\Phi)}} + \frac{\partial}{\partial x_i}\left(\frac{\alpha^2 \Phi_{,i}}{1 - \sqrt{Q(\bfa{u},\Phi)}} \right) + g_3=0\,.\label{bom12}
\end{align}

\noindent \textbf{Boundary conditions.} 
We denote by $\partial \Omega = \partial \Omega_D \cup \partial \Omega_N$ a partition of the boundary of $\Omega\,.$  Conventionally, $\partial \Omega_N = \Gamma_N$ is a relatively open set, and $\partial \Omega_D=\Gamma_D$ is a closed
set in $\partial \Omega\,.$ 
Here, $\bfa{n}$ stands for the piecewise smooth field on $\partial \Omega$ of outward pointing unit
normal vectors. 
On  $\Gamma_D\,,$ 
 the displacement $\bfa{\hat{u}}$ and the rotation $\bfa{\hat{\phi}}$ are prescribed.
 On $\Gamma_N\,,$ the tractions $\bfa{\hat{\tau}}$ and $\bfa{\hat{m}}$ are prescribed, respectively. Then, we obtain mixed boundary conditions for the equilibrium equations \eqref{bof}--\eqref{bom} of the kinematic relations \eqref{strain-displace}--\eqref{torsion-rotation} and of the constitutive equations \eqref{rel2a}--\eqref{rel2b}.
 More precisely \cite{polish1, polish2,ji1966sandru,boundaryvar, boundarystay}, for $i,j = 1,2\,,$
\begin{align}
	\bfa{u}(\bfa{x})= \bfa{\hat{u}}(\bfa{x})\,, \quad 
	&\bfa{\phi}(\bfa{x})= \bfa{\hat{\phi}}(\bfa{x}) = (0,0,\hat{\Phi}(\bfa{x})) \quad  & \text{ on } \Gamma_{D}\,,\label{bc-dis-ro}\\
	%
	\tau_{ij}(\bfa{x})\, n_i(\bfa{x}) = \hat{\tau}_j(\bfa{x})\,, \quad 
	&  
	 m_{i3}(\bfa{x})\, n_i(\bfa{x}) = \hat{m}_3(\bfa{x}) \quad & \text{ on } \Gamma_N\,.\label{bc-tractions}
\end{align}
%
In \cite{existence1965}, the solution of a linearized  version of the above mixed-boundary value problem is proven to be unique.

Note that if one would like to investigate dynamics, then the right-hand side of \eqref{bof1}, \eqref{bof2}, and \eqref{bom12} will respectively be replaced by $\rho (u_1)_{tt}\,, \rho (u_2)_{tt}\,,$ and $\rho J_3 (\Phi)_{tt}\,,$ where $\rho$ denote the body mass density, and $\rho J_3$ is the microinertia \cite{ji1967deformEringen, energymodes, polish1}. 

\subsection{Variational problem of plane nonlinear strain-limiting Cosserat elasticity}\label{vareqs}
In order to establish the variational form of our planar nonlinear strain-limiting Cosserat elasticity problem \eqref{bof}--\eqref{bom}, we consider the following special boundary conditions of \eqref{bc-dis-ro}--\eqref{bc-tractions} (for $i, j = 1,2$):
\begin{align}\label{bcsall}
\begin{split}
&u_1 = 0, \quad u_2 = 0, \quad \Phi = 0, \quad \bfa{x} \in \Gamma_D\,,\\
&\tau_{ij} n_i = 0, \quad m_{i3} n_i = 0, \qquad  \bfa{x} \in \Gamma_N\,,
\end{split}
\end{align} 
where $\partial \Omega = \Gamma_D \cup \Gamma_N\,,$ and $\bfa{u} = (u_1,u_2,0)\,, \bfa{\phi} = (0,0, \Phi)$ from \eqref{planar}. 

Note that the Neumann (natural or traction) boundary conditions will be self-contained in the linear functional of the variational problem below. In contrast, the Dirichlet (essential or displacement) boundary conditions are fulfilled by the choice of the solution and test function space
\begin{equation}\label{testsp}
V = H^1_{\Gamma_D}(\Omega) : = \{v \in H^1(\Omega) \ | \ \mbox{tr}(v) = 0 \text{ on } \Gamma_D\}\,,
\end{equation}
where $\mbox{tr}(v) \in L^2(\Gamma_D)$ represents the boundary ``values'' of $v$ on $\Gamma_D$ and is called the trace of $v\,.$ Moreover, let $\bfa{V} = V^2$ and $\bfa{\mathcal{V}} = V^3$.

To obtain the variational form of the balance equations \eqref{linearmoment2}--\eqref{angularmoment2}, we multiply them by test functions $\bfa{v} \in \bfa{V}$ and  $\Psi \in V\,,$ respectively, then employ integration by parts to reach the coordinate-free (in tensor notations) problem: find $\bfa{u} \in \bfa{V}$ and  $\Phi \in V$ such that
\begin{align}\label{vbofmip}
\begin{split}
  - \int_{\Omega} (\nabla \cdot \bfa{\tau} ) \cdot \bfa{v} \dx  &= \int_{\Omega} \bfa{f} \cdot \bfa{v} \dx\,,\\
  - \int_{\Omega} (\bfa{\e} : \bfa{\tau}) \Psi \dx - \int_{\Omega} (\nabla \cdot \bfa{m}) \Psi \dx &= \int_{\Omega} g_3 \Psi \dx\,.
\end{split}
\end{align}
Again, using integration by parts for \eqref{vbofmip}, we get
\begin{align}\label{vbofmt}
\begin{split}
  \int_{\Omega} \bfa{\tau}^{\tu{T}} \cdot \bfa{Dv} \dx - \int_{ \Omega} \Div(\bfa{\tau} \bfa{v}) \dx \ &= \int_{\Omega} \bfa{f} \cdot \bfa{v} \dx\,,\\
  - \int_{\Omega} (\bfa{\e} : \bfa{\tau}) \Psi \dx + \int_{\Omega} \bfa{m} \cdot \nabla \Psi \dx - \int_{\Omega} \Div(\Psi \bfa{m})\dx &= \int_{\Omega} g_3 \Psi \dx\,.
\end{split}
\end{align}
Then, the divergence theorem is utilized (with $(\bfa{\tau} \bfa{v}) \cdot \bfa{n} = \bfa{v} \cdot (\bfa{\tau}^{\tu{T}} \bfa{n}$)) to obtain
\begin{align}\label{vbofmtb}
\begin{split}
  \int_{\Omega} \bfa{\tau}^{\tu{T}} \cdot \bfa{Dv} \dx - \int_{\partial \Omega} \bfa{v} \cdot (\bfa{\tau}^{\tu{T}} \bfa{n}) \da \ &= \int_{\Omega} \bfa{f} \cdot \bfa{v} \dx\,,\\
  - \int_{\Omega} (\bfa{\e} : \bfa{\tau}) \Psi \dx + \int_{\Omega} \bfa{m} \cdot \nabla \Psi \dx - \int_{\partial \Omega} \Psi \bfa{m} \cdot \bfa{n} \da &= \int_{\Omega} g_3 \Psi \dx\,.
\end{split}
\end{align}

Given the boundary conditions \eqref{bcsall}, note that the Neumann (natural or traction) boundary conditions are imposed on the linear functional employed in the variational problem \eqref{vbofmtb} thanks to 
\begin{align*}
\int_{\partial \Omega} \bfa{v} \cdot (\bfa{\tau}^{\tu{T}} \bfa{n}) \da &= \int_{\Gamma_D} \bfa{v} \cdot (\bfa{\tau}^{\tu{T}} \bfa{n}) \da + \int_{\Gamma_N} \bfa{v} \cdot (\bfa{\tau}^{\tu{T}} \bfa{n}) \da\,, \\
\int_{\partial \Omega} \Psi \bfa{m} \cdot \bfa{n} \da &= \int_{\Gamma_D} \Psi \bfa{m} \cdot \bfa{n} \da + \int_{\Gamma_N} \Psi \bfa{m} \cdot \bfa{n} \da\,.
\end{align*}

\noindent Taking into account this remark, in simulations, benefiting from the boundary conditions \eqref{bcsall}, the variational form of our strain-limiting Cosserat problem \eqref{linearmoment2}--\eqref{angularmoment2} (or \eqref{bof}--\eqref{bom}) is as follows: Find solutions $\bfa{u} \in \bfa{V}\,, \Phi \in V\,,$ with test functions $\bfa{v} \in \bfa{V}\,, \Psi \in V$ for the problem
\begin{align}\label{vbofmtvb}
\begin{split}
  \int_{\Omega} \bfa{\tau}^{\tu{T}} \cdot \bfa{Dv} \dx   &= \int_{\Omega} \bfa{f} \cdot \bfa{v} \dx\,,\\
  -\int_{\Omega} (\bfa{\e} : \bfa{\tau}) \Psi \dx + \int_{\Omega} \bfa{m} \cdot \nabla \Psi \dx  &= \int_{\Omega} g_3 \Psi \dx\,.
\end{split}
\end{align}
\section{Fine-grid discretization and Picard iteration for linearization}\label{ffem}
This section is devoted to a fine-scale finite element approximation of our 2D variational problem \eqref{vbofmtvb} and its linearization by Picard iteration \cite{gne}. 

In order to discretize \eqref{vbofmtvb}, we first let $\mathcal{T}^h$ be a fine grid as a partition of the computational domain $\Omega$ into fine cells $K_j^h\,,$ for $j = 1,..., N_c^h\,,$ where $N_c^h$ denotes the total number of such fine cells, and $h>0$ is called the fine-grid size.  The grid $\mathcal{T}^h$ is assumed to be fine enough that it resolves all variations of the high-contrast multiscale coefficients $\xi, \alpha$, and $\beta\,.$ 
Next, with regard to the fine grid $\mathcal{T}^h$ (which is triangular in our paper), let $V_h$ be the first-order Galerkin (standard) finite element basis space:
\begin{equation}\label{vh}
V_h:= \{ v \in V: v|_{K^h} \in \mathcal{P}_1(K^h) \; \forall K^h \in \mathcal{T}^{h}\}\,,
\end{equation}
where $\mathcal{P}_1(K^h)$ is the space of all linear functions (polynomials of degree  $\leq 1$) over the triangle $K^h\,.$ 
In other words, the finite element space $V_h$ consists of conforming piecewise linear functions $v$   
defined on the mesh $\mathcal{T}^h\,.$  
Let $\bfa{V}_h = V_h^2$
and denote the $[L^2(\Omega)]^2$ projection operator onto $\bfa{V}_h$ by $P_h\,.$  For later use, we also define the function space $\bfa{\mathcal{V}}_h = V_h^3\,.$
%



\textbf{Nonlinear Solve:} Inspired by one of the authors' previous works \cite{gne, cemnlporo,rtt21,ttr22,sdt22}, we use Picard linearization to deal with nonlinearities. At the beginning of this Picard iterative process, we guess $\bfa{u}^0 =(u_1^0,u_2^0) \in \bfa{V}_h\,, \Phi^0 \in V_h\,.$  In this section, the subscript $h$ for solutions can be removed from the fine grid to simplify the notation.  With $n=0,1,2, \cdots\,,$ given $\bfa{u}^n =(u_1^n,u_2^n) \in \bfa{V}_h\,, \Phi^n \in V_h\,,$ we define the following linear and bilinear forms (for all $\bfa{u}, \bfa{v} \in \bfa{V}_h\,, \Phi, \Psi \in V_h$):
\begin{align}\label{bllforms}
\begin{split}
a_{\bfa{u} \bfa{u}}^{n} (\bfa{u},\bfa{v}) &= \int_{\Omega} \frac{\xi^2}{1 - \sqrt{Q(\bfa{u}^{n},\Phi^{n})}} \frac{\partial u_j}{\partial x_i} \frac{\partial v_j}{\partial x_i} \dx \,,\\
a_{\bfa{u} \Phi}^{n} (\Phi,\bfa{v}) &=\int_{\Omega} \frac{ \bfa{\e}_{ji}\xi^2 }{1 - \sqrt{Q(\bfa{u}^{n},\Phi^{n})}} \, \Phi \, \frac{\partial v_j}{\partial x_i} \dx\,,\\
a_{\Phi \bfa{u}}^{n} (\bfa{u}, \Psi) &=  \int_{\Omega} \frac{\bfa{\e}_{ji} \xi^2 }{1 - \sqrt{Q(\bfa{u}^{n},\Phi^{n})}} \frac{\partial u_j}{\partial x_i} \Psi \dx = - \int_{\Omega} \frac{\xi^2}{1 - \sqrt{Q(\bfa{u}^{n},\Phi^{n})}} \left(\frac{\partial u_2}{\partial x_1}-\frac{\partial u_1}{\partial x_2}\right)\Psi \dx\,,\\
\tilde{a}_{\Phi \Phi}^{n} (\Phi, \Psi) &= \int_{\Omega} \frac{2 \xi^2 \Phi}{1 - \sqrt{Q(\bfa{u}^{n},\Phi^{n})}} \Psi \dx\,,\\
a_{\Phi \Phi}^{n} (\Phi,\Psi) &= \int_{\Omega} \frac{\alpha^2}{1 - \sqrt{Q(\bfa{u}^{n},\Phi^{n})}} \frac{\partial \Phi}{\partial x_i}\frac{\partial \Psi}{\partial x_i} \dx\,,
\end{split}
\end{align}
and
\begin{align}\label{bllforms2}
\begin{split}
L_{\bfa{u}}(\bfa{v}) &= \int_{\Omega} f_j v_j \dx\,,\\
L_{\Phi}(\Psi) &= \int_{\Omega} g_3 \Psi \dx\,,
\end{split}
\end{align}
where $Q(\bfa{\chi},\bfa{\gamma})$ is computed by \eqref{qu}.

Employing these forms \eqref{bllforms}--\eqref{bllforms2}, the index notation in \eqref{rel2a}--\eqref{rel2b} and \eqref{kinematic}, we obtain the following variational problem from \eqref{vbofmtvb}: find $\bfa{u}_h^{n+1} \in \bfa{V}_h\,, \Phi_h^{n+1} \in V_h\,,$ such that
\begin{align}\label{vbofm}
\begin{split}
a_{\bfa{u} \bfa{u}}^{n} (\bfa{u}^{n + 1},\bfa{v}) + a_{\bfa{u} \Phi}^{n} (\Phi^{n + 1},\bfa{v}) & = L_{\bfa{u}}(\bfa{v})\,, \quad \mbox{ for all } \bfa{v} \in \bfa{V}_h\,, j=1,2\,,\\
a_{\Phi \bfa{u}}^{n} (\bfa{u}^{n + 1}, \Psi) + \tilde{a}_{\Phi \Phi}^{n} (\Phi^{n + 1}, \Psi) + a_{\Phi \Phi}^{n} (\Phi^{n + 1},\Psi) &= L_{\Phi}(\Psi)\,, \quad \mbox{ for all } \Psi \in V_h\,.
\end{split}
\end{align}
The variational problem \eqref{vbofm} can be represented in a coupled form. To that end, for $\bfa{\Upsilon} = (\bfa{u}, \Phi) \in \bfa{\mathcal{V}}_h$ and $\bfa{\nu} = (\bfa{v}, \Psi) \in \bfa{\mathcal{V}}_h\,,$ let us define the following operators:
\begin{align}\label{bin}
 \begin{split}
b_n(\bfa{\Upsilon},\bfa{\nu}) &= a_{\bfa{uu}}^n (\bfa{u}, \bfa{v}) + a_{\bfa{u}\Phi}^n (\Phi, \bfa{v})\,,\\
c_n(\bfa{\Upsilon},\bfa{\nu}) & = a_{\Phi\bfa{u}}^n (\bfa{u}, \Psi) + \tilde{a}_{\Phi\Phi}^n (\Phi, \Psi) + 
a_{\Phi\Phi}^n (\Phi, \Psi)\,,\\
a_n (\bfa{\Upsilon},\bfa{\nu}) &= b_n (\bfa{\Upsilon},\bfa{\nu}) + c_n (\bfa{\Upsilon},\bfa{\nu})\,,\\
F(\bfa{\nu}) &= L_{\bfa{u}}(\bfa{v}) + L_{\Phi}(\Psi)\,.
 \end{split}
 \end{align}
At the $n$th Picard iterative step, we equip the space $\bfa{\mathcal{V}}_h$ with the norm
\begin{equation}\label{annorm} \norm{\cdot}_{\bfa{\mathcal{V}}_h} = \sqrt{a_n (\cdot, \cdot)}\,.
\end{equation}

Then, we reach the following variational problem: given $\bfa{\Upsilon}^{n} \in \bfa{\mathcal{V}}_h\,,$ find $\bfa{\Upsilon}^{n + 1} \in \bfa{\mathcal{V}}_h$ such that for all $\bfa{\nu} \in \bfa{\mathcal{V}}_h\,,$ 
\begin{equation}\label{absvar}
a_n (\bfa{\Upsilon}^{n + 1}, \bfa{\nu}) = F(\bfa{\nu})\,.
\end{equation}
%
In this study, we propose a halting indicator based on the relative successive difference, which indicates that provided a user-defined tolerance $\delta_0 > 0\,,$ if \begin{equation}\label{pct} \dfrac{\|\bfa{\Upsilon}^{n+1} - \bfa{\Upsilon}^{n} \|_{\bfa{L}^2(\Omega)}}{\| \bfa{\Upsilon}^{n} \|_{\bfa{L}^2(\Omega)}} \leq \delta_0\,, \end{equation}
then the iteration process is completed.  The selection of $\delta_0 = 10^{-5}$ is made in Section \ref{num_results}.
Note that in the fine grid, $\bfa{\Upsilon}^{n+1}$ represents the fine-scale solution $\bfa{\Upsilon}_h^{n+1}\,.$

\section{Coarse-grid approximation using the GMsFEM}\label{approx}
\subsection{Overview}\label{overview}

We aim to construct the offline and online multiscale spaces using the generalized multiscale finite element method (GMsFEM \cite{G1}) for the coupled nonlinear system \eqref{vbofmtvb}.  Inspired by \cite{gne, gnone, tfcmm}, respecting the nonlinearity, we may build those multiscale spaces based on the linearized formulation \eqref{absvar}, where the nonlinearity can be considered constant at each Picard iteration step.

First, we go over the coarse-grid notation. Let $\mathcal{T}^H$ be a coarse grid consisting of coarse cells (coarse blocks) $K_j^H$, where $j = 1 ,..., N_c^H\,,$ and $N_c^H$ is the number of all coarse blocks (see Fig.~\ref{fig:gmsfem_grid}). Let $N^{H}_v$ denote the number of all vertices (including the ones on the boundary) of $\mathcal{T}^H\,.$  
We assume that 
the coarse-grid size $H$ (where $H \gg h >0$) 
allows for efficient computations, but such resolution does not take into account the variation of the high-contrast multiscale coefficients  $\xi\,, \alpha\,,$ and $\beta\,.$ This is a typical situation in the application of various multiscale finite element methods. We note that all computations of multiscale basis functions take place in local domains (or coarse neighborhoods) $\omega_i$ of the coarse node $\bfa{x}_i^H$ defined as the union of all coarse blocks $K_j^H \in \mathcal{T}^H$ containing $\bfa{x}_i^H$ (see also Fig.~\ref{fig:gmsfem_grid}): 
\begin{equation}\label{cgrid}
\omega_i = \bigcup \left\{K_j^H \in \mathcal{T}^H: \bfa{x}_i^H \in \overline{{K}_j^H}\right\}\,, \quad (i = 1, ..., N_v^H)\,.
\end{equation}

We refer the readers to \cite{G1, G2, chungres1, chungres, chung2016adaptive} for the GMsFEM's specifics and to \cite{mcl,gne, mcontinua17} for a summary. 
 Our main objective is to use the GMsFEM to find a multiscale solution $\bfa{\Upsilon}_{\tu{ms}}\,,$ which is a good approximation of the fine-scale solution $\bfa{\Upsilon}_{h}$ from \eqref{absvar}. To that end, at each Picard iteration, we employ the GMsFEM on a coarse grid to solve local problems (to be described then) in each coarse neighborhood. This allows us to systematically generate multiscale basis functions (degrees of freedom for the solution) that retain fine-scale characteristics.  
It can be seen that the creation of local multiscale basis functions is a crucial part of the GMsFEM. Similar to \cite{gne}, first, only the so-called offline multiscale basis functions 
(which are computed during the offline stage) will be produced and used.  Second, based on such offline basis functions and a few local residuals, we will build additional online multiscale basis functions that are problem-dependent and computed locally and adaptively 
in order to increase the multiscale approximation's accuracy. 
When offline and online basis functions are combined, as with \cite{chungres,gne}, our findings demonstrate that the multiscale solution $\bfa{\Upsilon}_{\tu{ms}}$ will rapidly converge to the fine-scale solution $\bfa{\Upsilon}_{h}\,.$

\begin{figure}[hbt!]
\centering
\includegraphics[width=0.6\textwidth]{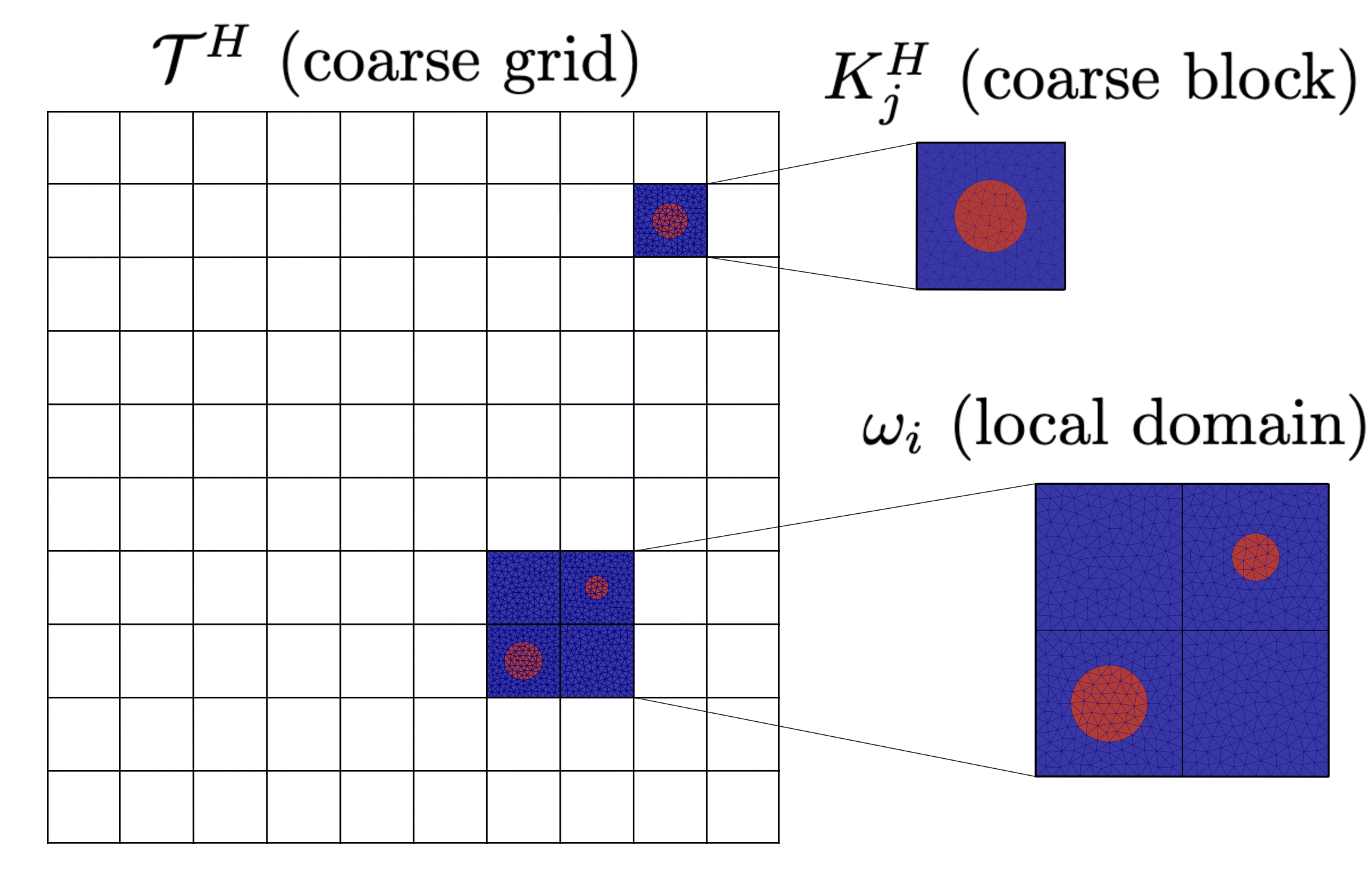}
\caption{Coarse grid $\mathcal{T}^H$, coarse block $K_j^H$, and local domain (coarse neighborhood) $\omega_i\,.$}
\label{fig:gmsfem_grid}
\end{figure}

\subsection{Offline GMsFEM}\label{offbuild}

 We investigate the offline GMsFEM for the 2D strain-limiting Cosserat elasticity model \eqref{bof}--\eqref{bom}. That is, in the next Sections  \ref{snaps}, \ref{mssp}, given $\bfa{\Upsilon}^{n}_{\tu{off}}=(\bfa{u}_{\tu{off}}^n, \Phi_{\tu{off}}^n)$ at the $n$th Picard iteration, we will construct $\bfa{\mathcal{V}}_{\tu{off}}=\bfa{\mathcal{V}}^n_{\tu{off}}$ at the offline stage.    In practice (see Section \ref{offgms}), provided an initial guess $\bfa{\Upsilon}^{0}_{\tu{off}}=(\bfa{u}_{\tu{off}}^0, \Phi_{\tu{off}}^0)\,,$ we will only have to build one $\bfa{\mathcal{V}}_{\tu{off}}=\bfa{\mathcal{V}}^0_{\tu{off}}\,.$ 
 Since this model is coupled, the multiscale basis functions will be established in a coupled way. The resulting coupled basis functions, on average, can provide better accuracy (in comparison with the uncoupled multiscale basis functions) due to considering the mutual influence of the solution fields \cite{mcl, rtt21, cosserat2022, ammosov2024computational}. The algorithm to compute the GMsFEM solution is presented in Section \ref{offgms}.

\subsubsection{Snapshot spaces}\label{snaps} 
The first step of the GMsFEM is constructing snapshot spaces. These spaces are built by solving local problems having different boundary conditions, which allows us to consider models possessing a broad class of boundary conditions and right-hand sides. 
 The details are as follows.  

We denote by $J_h(\omega_i)$ the set of all fine-grid nodes on $\partial \omega_i$ and by $N_v^{\partial \omega_i}$ the count of such nodes. Note that, in $\omega_i\,,$ the integration domain of the operators \eqref{bllforms}--\eqref{bllforms2} and \eqref{bin} will also be $\omega_i\,.$ Our fine-grid approximation (FEM) space is $\bfa{\mathcal{V}}_h(w_i)\,,$ as the conforming space $\bfa{\mathcal{V}}_h$ restricted to the local domain $w_i\,.$  Let $\delta_{k}$ be a function defined as
\begin{equation}\label{del1}
\delta_{k}(\bfa{x}^h_m) =
 \begin{cases}
  1 \quad m = k\,,\\
  0 \quad m \ne k\,,
 \end{cases}
\end{equation}
for all $\bfa{x}^h_m$ in $J_h(\omega_i)\,,$ $1 \leq k \leq N_v^{\partial w_i}\,.$

Toward establishing the snapshot spaces, we need to solve the following local problems on each coarse neighborhood $w_i\,,$ using the balance equations \eqref{bof}--\eqref{bom}.
In particular, at the Picard iterative step $n$th, for $i=1,...,N_v^H$, $k = 1,...,N_v^{\partial \omega_i}$, and $l = 1, 2, 3$, find the snapshot function ${\boldsymbol{\eta}}_k^{l, \omega_i} = (\boldsymbol{\varrho}_k^{l, \omega_i}, \theta_k^{l, \omega_i}) \in \bfa{\mathcal{V}}_h(w_i)$ such that

\begin{equation}\label{snapshotsys_strong}
\begin{split}
\frac{\partial}{\partial x_s}\left(\frac{\xi^2 [(\varrho_k^{l, \omega_i})_{j,s}+ \bfa{\e}_{js} \theta_k^{l, \omega_i}]}{1 - \sqrt{Q(\bfa{u}^{n}_{\tu{off}},\Phi^{n}_{\tu{off}})}} \right) &= 0 \quad \text{in } \omega_i\,,\\
\frac{\bfa{\e}_{sj} \xi^2 [ (\varrho_k^{l, \omega_i})_{j,s} + \bfa{\e}_{js} \theta_k^{l, \omega_i}]}{1 - \sqrt{Q(\bfa{u}^{n}_{\tu{off}},\Phi^{n}_{\tu{off}})}} + \frac{\partial}{\partial x_s}\left(\frac{\alpha^2 (\theta_k^{l, \omega_i})_{,s}}{1 - \sqrt{Q(\bfa{u}^{n}_{\tu{off}},\Phi^{n}_{\tu{off}})}} \right) &= 0 \quad \text{in } \omega_i\,,\\
{\boldsymbol{\eta}}_k^{l, \omega_i} &= {\boldsymbol{\delta}}_k^l  \quad \text{on } \partial \omega_i\,.
\end{split}
\end{equation}
Here, we specify $\bfa{\delta}^l_k$ as
\begin{equation}\label{del2}
\bfa{\delta}_k^l (\bfa{x}^h_m) = \delta_{k}(\bfa{x}^h_m) {\bfa{e}}_l\,, 
\end{equation}
where $\bfa{x}_m^h \in J_h(\omega_i)\,,$ the function $\delta_k$ was defined in \eqref{del1}, and $\{\bfa{e}_l \, | \, l=1,2,3\}$ is a standard basis in $\mathbb{R}^3$ (that is, ${\boldsymbol{e}}_l$ is the $l$th column of the identity matrix ${\boldsymbol{I}}_{3}$). In this way, we need to solve $3  N_v^{\partial \omega_i}$ local problems. Using the obtained solutions from \eqref{snapshotsys_strong}, for $i = 1,..., N_v^H$ and $l = 1,2,3\,,$ we get the local snapshot spaces
\begin{equation}\label{eq:coupled_snapshot_space}
\bfa{\mathcal{V}}_{\tu{snap}}^{l, \omega_i} = \mbox{span}\{{\boldsymbol{\eta}}_k^{l, \omega_i}\, | \, k = 1, ..., N_v^{\partial \omega_i}\}\,.
\end{equation}

\subsubsection{Multiscale space}\label{mssp} To construct multiscale basis functions, we need to perform a local spectral decomposition of the snapshot spaces by solving the following problems:
for $i=1,...,N_v^H$ and $l = 1, 2, 3$, find $\lambda_k^{l, \omega_i} \in \mathbb{R}$ and $\tilde{{\bfa{\eta}}}_k^{l, \omega_i} \in \bfa{\mathcal{V}}_{\tu{snap}}^{l, \omega_i}$ such that
\begin{equation}
a_n (\tilde{{\bfa{\eta}}}_k^{l, \omega_i}, \bfa{\nu}) = \lambda_k^{l, \omega_i} s_n (\tilde{{\bfa{\eta}}}_k^{l, \omega_i}, \bfa{\nu})\,, \quad \text{for all } \bfa{\nu} = (\bfa{v}, \Psi) \in \bfa{\mathcal{V}}_{\tu{snap}}^{l, \omega_i}\,, 
\end{equation}
where the operator $a_n$ was defined in \eqref{bin}, and for all $\bfa{\eta} = (\bfa{\varrho}, \theta), \, \bfa{\nu} = (\bfa{v}, \Psi) \in \bfa{\mathcal{V}}_{\tu{snap}}^{l, \omega_i}\,,$
\begin{equation*}
s_n(\bfa{\eta}, \bfa{\nu}) = \int_{\omega_i} \frac{\xi^2}{1 - \sqrt{Q(\bfa{u}^{n}_{\tu{off}},\Phi^{n}_{\tu{off}})}} \varrho_{j} v_{j} \dx + \int_{\omega_i} \frac{\alpha^2}{1 - \sqrt{Q(\bfa{u}^{n}_{\tu{off}},\Phi^{n}_{\tu{off}})}} \theta \Psi \dx\,.
\end{equation*}

Next, we sort the eigenvalues in ascending order, then, select the eigenfunctions corresponding to the first $N_{\tu{b}}$ eigenvalues and name these eigenfunctions as $\bar{{\bfa{\eta}}}_k^{l, \omega_i}\,,$ for $1\leq k \leq N_{\tu{b}}\,.$ These eigenfunctions are taken to be multiscale basis functions. For this reason, they are also called spectral basis functions. However, we need the multiscale partition of unity \cite{gne,ammosov2024computational} to ensure their conformality.

Since the considered model is coupled, this partition of unity functions is created in a coupled way \cite{ammosov2024computational}. For this purpose, we solve the following local problems on coarse-grid cells. With $i=1,...,N_v^H$, $l = 1, 2, 3$, and $\forall K \in \omega_i$, find $\bfa{\chi}_i^l = (\bfa{\zeta}_i^l, \mu_i^l)$ such that
\begin{equation}\label{mpou_strong}
\begin{split}
\frac{\partial}{\partial x_s}\left(\frac{\xi^2 [(\zeta_i^{l})_{j,s}+ \bfa{\e}_{js} \mu_i^{l}]}{1 - \sqrt{Q(\bfa{u}^{n}_{\tu{off}},\Phi^{n}_{\tu{off}})}} \right) &= 0 \quad \text{in } K\,,\\
\frac{\bfa{\e}_{sj} \xi^2 [ (\zeta_i^{l})_{j,s} + \bfa{\e}_{js} \mu_i^{l}]}{1 - \sqrt{Q(\bfa{u}^{n}_{\tu{off}},\Phi^{n}_{\tu{off}})}} + \frac{\partial}{\partial x_s}\left(\frac{\alpha^2 (\mu_i^{l})_{,s}}{1 - \sqrt{Q(\bfa{u}^{n}_{\tu{off}},\Phi^{n}_{\tu{off}})}} \right) &= 0 \quad \text{in } K\,,\\
\boldsymbol{\chi}_i^l &= \boldsymbol{g}_i^l \quad \text{on } \partial K\,,
\end{split}
\end{equation}
where $\boldsymbol{g}_i^l = \bar{\chi}_i \boldsymbol{e}_l$, $\bar{\chi}_i$ is a standard partition of unity function, which is continuous and linear in $\omega_i\,,$ equal to $1$ at $\bfa{x}_i^H$ and $0$ at all the other coarse-grid nodes of $w_i\,.$ Then, we assemble the multiscale partition of unity $\bfa{\chi}_i$ as follows
\begin{equation*}
\boldsymbol{\chi}_i = (\zeta_{i1}^{1}, \zeta_{i2}^{2}, \mu_{i}^{3}), \quad
i = 1,...,N_v^H\,.
\end{equation*}

Next, we multiply the obtained eigenfunctions by these multiscale partition of unity functions to obtain the local multiscale basis functions
\begin{equation}\label{offfunc}
{\boldsymbol{\eta}}_{ik}^{l,\tu{off}} = \boldsymbol{\chi}_i \bar{{\boldsymbol{\eta}}}_k^{l, \omega_i}, \quad i = 1,...,N_v^H, \quad k = 1, ..., N_{\tu{b}}, \quad l = 1, 2, 3\,.
\end{equation}
Finally, we establish the global multiscale space 
\begin{equation}\label{offsp}
\bfa{\mathcal{V}}_{\tu{off}} = \mbox{span}\{{\boldsymbol{\eta}}_{ik}^{l,\tu{off}}: i = 1, ..., N_v^H, \quad k = 1, ..., N_{\tu{b}}, \quad l=1, 2, 3\}\,.
\end{equation}
Since constructing these basis functions and function space takes place before solving our main problem (online stage), we will call \eqref{offfunc} the offline multiscale basis functions and \eqref{offsp} the offline multiscale function space, respectively.

\subsubsection{Offline GMsFEM for strain-limiting Cosserat elasticity}\label{offgms}
Let us consider an algorithm for solving the 2D strain-limiting Cosserat problem in the offline multiscale space $\bfa{\mathcal{V}}_{\tu{off}} = \bfa{\mathcal{V}}^0_{\tu{off}}$ (computed from Section \ref{mssp}).  Choosing $\delta_0 \in \mathbb{R}_{+}$ and guessing an initial $\bfa{\Upsilon}^{0}_{\tu{off}}=(\bfa{u}_{\tu{off}}^0, \Phi_{\tu{off}}^0) \in \bfa{\mathcal{V}}_{\tu{off}}\,,$ we iterate as follows:

\textbf{Step 1:} Provided $\bfa{\Upsilon}^{n}_{\tu{off}} \in \bfa{\mathcal{V}}_{\tu{off}}\,,$ solve for $\bfa{\Upsilon}^{n + 1}_{\tu{off}} \in \bfa{\mathcal{V}}_{\tu{off}}$ from a similar formulation to \eqref{absvar}:
\begin{equation}\label{solmsoff}
a_n (\bfa{\Upsilon}^{n + 1}_{\tu{off}}, \bfa{\nu}) = F(\bfa{\nu}), \quad \text{for all } \bfa{\nu} \in \bfa{\mathcal{V}}_{\tu{off}}\,.
\end{equation}

\textbf{Step 2:} If $\dfrac{\norm{\bfa{\Upsilon}_{\tu{off}}^{n + 1} - \bfa{\Upsilon}_{\tu{off}}^{n}}_{\bfa{\mathcal{V}}_h}}{\norm{\bfa{\Upsilon}_{\tu{off}}^{n}}_{\bfa{\mathcal{V}}_h}} > \delta_0$ (where $\norm{\cdot}_{\bfa{\mathcal{V}}_h} = \sqrt{a_n (\cdot, \cdot)}$ defined in \eqref{annorm}), then set $n \gets n+1$ and go to \textbf{Step 1}.

\bigskip

In our numerical implementation, we use the first-order finite elements on the fine grid $\mathcal{T}^h$ to compute the multiscale basis functions $\bfa{\eta}_{ik}^{l,\tu{off}}$ \eqref{offfunc} of $\bfa{\mathcal{V}}_{\tu{off}}$ \eqref{offsp}. Hence, we can define
\begin{equation*}
\scriptstyle
\bfa{R}_{\tu{off}} = \left [\bfa{\eta}_{11}^{1,\tu{off}}, \ldots , \bfa{\eta}_{1 N_{\tu{b}}}^{1, \tu{off}},
\bfa{\eta}_{2 1}^{1, \tu{off}}, \ldots, \bfa{\eta}_{2 N_{\tu{b}}}^{1, \tu{off}},
\ldots,
\bfa{\eta}_{N_v^H 1}^{1, \tu{off}}, \ldots, \bfa{\eta}_{N_v^H N_{\tu{b}}}^{1, \tu{off}},
\ldots,
\bfa{\eta}_{11}^{3,\tu{off}}, \ldots , \bfa{\eta}_{1 N_{\tu{b}}}^{3, \tu{off}},
\bfa{\eta}_{2 1}^{3, \tu{off}}, \ldots, \bfa{\eta}_{2 N_{\tu{b}}}^{3, \tu{off}},
\ldots,
\bfa{\eta}_{N_v^H 1}^{3, \tu{off}}, \ldots, \bfa{\eta}_{N_v^H N_{\tu{b}}}^{3, \tu{off}}\right]\,,
\end{equation*}
where each multiscale basis function $\bfa{\eta}_{ik}^{l,\tu{off}}$ \eqref{offfunc} is in the form of fine-scale coordinate representation arranged in column vector. Thus, for the solution of the form $\bfa{\Upsilon}_{\tu{off}}=(u_1, u_2, \Phi)\,,$ we first put the basis functions for $u_1$ on the first local domain, then do the same for other local domains, and after the last local domain, we move to the basis functions for $u_2\,,$ then repeat the process until we include all the basis functions for $\Phi\,.$  Using $\bfa{R}_{\tu{off}}\,,$ the coarse-scale solution $\bfa{\Upsilon}_{\tu{off}}^{c, n+1}$ can be found as $\bfa{\Upsilon}_{\tu{off}}^{c, n+1} = (\bfa{R}_{\tu{off}}^{\text{T}} \bfa{A}_h^n \bfa{R}_{\tu{off}})^{-1}(\bfa{R}_{\tu{off}}^{\text{T}} \bfa{b}_h ) \in \bfa{\mathcal{V}}_{\tu{off}}\,.$ Here, $a_n (\bfa{\Upsilon}, \bfa{\nu})$ and $F(\bfa{\nu})$ (defined in \eqref{bin}) respectively give the fine-grid matrix and right-hand side vector  $\bfa{A}_h^n$ and $\bfa{b}_h\,,$ obtained by standard FEM piecewise linear basis functions. Note that we can project $\bfa{\Upsilon}_{\tu{off}}^{c,n + 1}$ onto $\bfa{\mathcal{V}}_h$ using $\bfa{R}_{\tu{off}}$ as follows $\bfa{\Upsilon}_{\tu{off}}^{f, n+1}$ = $\bfa{R}_{\tu{off}} \bfa{\Upsilon}_{\tu{off}}^{c, n+1}$.
\section{Residual-based online adaptive basis enrichment for GMsFEM}\label{online}

As we have indicated in the previous Sections and as in \cite{gne}, to create a coarse representation of the fine-grid solution and to enable fast convergence of the adaptive enrichment technique, certain online basis functions are necessary.  Such online adaptivity is suggested and examined mathematically in \cite{chungres}.
 To be more precise, at the current $n$th Picard iterative step, if the local residual associated with some coarse neighborhood $w_i$ is big (refer to Section \ref{onalg}), a new basis function $\bfa{\eta}_i^{\tu{on}} \in \bfa{\mathcal{V}}_i = V_i^3 =(H_0^1(w_i) \cap V_h)^3$ (defined in Section \ref{dualway2}) can be constructed at the online stage (using the equipped norm $\|\cdot \|_{\bfa{\mathcal{V}}_i}$ defined in \eqref{anorm}), and added to the multiscale basis functions space.  It is further demonstrated that the online basis construction yields an efficient approximation of the fine-scale solution $\bfa{u}_h$ if the offline space has offline basis functions that contain sufficient information.
 
We utilize the index $m\text{ } (\geq 0)$ to indicate the adaptive enrichment level within the $n$th Picard iteration. As a result, $\bfa{\mathcal{V}}^m_{\tu{ms}}$ stands for the corresponding GMsFEM space, and $\bfa{\Upsilon}^m_{\tu{ms}}$ denotes the corresponding solution found in (\ref{adapeq}).  In our numerical results, the sequence $\{\bfa{\Upsilon}^m_{\tu{ms}}\}_{m\geq 0}$ will eventually converge to the fine-scale solution $\bfa{\Upsilon}_h\,.$  We note in this section that offline and online basis functions can be contained in the space $\bfa{\mathcal{V}}^m_{\tu{ms}}$.  A method for deriving the space $\bfa{\mathcal{V}}^{m+1}_{\tu{ms}}$ from $\bfa{\mathcal{V}}^m_{\tu{ms}}$ will be established.

Now, following \cite{gne}, we develop online multiscale approaches based on the residual-driven online generalized multiscale finite element method \cite{chungres}.  The main idea of this method is to build only online basis functions during the adaptively iterative procedure, as opposed to offline basis functions, which are created ahead of the iterative procedure.  Here, adaptivity means adding online basis functions to a select few regions, and ``adaptive'' can stand for either ``uniform'' or ``adaptive''.  With the current multiscale solution $\bfa{\Upsilon}^m_{\tu{ms}}$ from \eqref{adapeq}, some local residuals are used to compute the online basis functions. Thus, we understand that certain offline basis functions play a vital role in the online basis function constructions.  In addition, we will acquire the necessary quantity of offline basis functions to reach a quickly and robustly convergent sequence of solutions.  

More specifically, at the current Picard iterative step $n$th, we are provided with $\bfa{\Upsilon}^n_{\tu{ms}}=(\bfa{u}^n_{\tu{ms}},\Phi^n_{\tu{ms}})$ (see Section \ref{onpi}), a local domain $w_i\,,$ and an inner adaptive iteration $m$th using the approximation space $\bfa{\mathcal{V}}^m_{\tu{ms}}\,.$  One can obtain the GMsFEM solution $\bfa{\Upsilon}^m_{\tu{ms}} = (\bfa{u}^m_{\tu{ms}}, \Phi^m_{\tu{ms}})\in \bfa{\mathcal{V}}^m_{\tu{ms}}$ from \eqref{adapeq}:
\begin{equation}
a_n (\bfa{\Upsilon}^m_{\tu{ms}}, \bfa{\nu}) = F(\bfa{\nu}), \quad \text{for all } \bfa{\nu} \in \bfa{\mathcal{V}}_{\tu{ms}}^m\,.
\end{equation}
Assume that a basis function $\bfa{\eta}_i^{\tu{on}} \in \bfa{\mathcal{V}}_i$ needs to be added to the $i$th local domain $w_i\,.$  To start, we set $\bfa{\mathcal{V}}^0_{\tu{ms}} = \bfa{\mathcal{V}}_{\tu{off}}\,.$ Let $\bfa{\mathcal{V}}^{m+1}_{\tu{ms}} = \bfa{\mathcal{V}}^m_{\tu{ms}} \oplus \tu{span}\{\bfa{\eta}_i^{\tu{on}}\}$ be the new approximation space, and the GMsFEM solution from \eqref{adapeq} be $\bfa{\Upsilon}^{m+1}_{\tu{ms}} \in \bfa{\mathcal{V}}^{m+1}_{\tu{ms}}\,.$  Consider the local residual
\begin{equation}
R^i(\bfa{\nu}) = F^i(\bfa{\nu}) - a_n^i(\bfa{\Upsilon}^m_{\tu{ms}}, \bfa{\nu})\,, \quad \forall \bfa{\nu} \in \bfa{\mathcal{V}}_i\,,
\end{equation}
where $a_n (\bfa{\Upsilon}, \bfa{\nu})$ and $F(\bfa{\nu})$ were defined in \eqref{bin}.
According to \cite{chungres}, the new online basis function $\bfa{\eta}_i^{\tu{on}} \in \bfa{\mathcal{V}}_i$ is the solution to
\begin{align}
a_n^i(\bfa{\eta}_i^{\tu{on}},\bfa{\nu})= R^i(\bfa{\nu})\,, \quad \forall \bfa{\nu} \in \bfa{\mathcal{V}}_i\,,
\end{align}
and the residual norm $\| R^i\|_{\bfa{\mathcal{V}}_i^*}= \|\bfa{\eta}_i^{\tu{on}}\|_{\bfa{\mathcal{V}}_i}\,,$ which provides a measurement of the amount of energy error reduction. In the present Picard iteration, the convergence of this method is discussed in \cite{chungres}. 

The adaptive approach's key point is identifying the coarse neighborhoods $w_i$ with the largest residuals $r_i=\| R^i\|_{\bfa{\mathcal{V}}_i^*}\,.$ Hence, we need to define the dual norm $\| \cdot \|_{\bfa{\mathcal{V}}_i^*}$ as follows.
\subsection{Approaches to compute the dual norm}\label{dualnorm}
We consider two ways for computing the dual norm $\| \cdot \|_{\bfa{\mathcal{V}}_i^*}\,.$  All the calculations are at the current $n$th Picard iteration, provided $\bfa{\Upsilon}^n_{\tu{ms}}=(\bfa{u}^n_{\tu{ms}},\Phi^n_{\tu{ms}})$ (refer to Section \ref{onpi}).

\subsubsection{First approach to compute the dual norm}\label{dualway1}
Over the coarse neighborhood $\omega_i\,,$ let $\bfa{V}_i = V_i^2 =(H_0^1(w_i) \cap V_h)^2\,,$ where $V_h$ was defined in \eqref{vh}.  Also in $w_i\,,$ with $\bfa{\Upsilon}^m_{\tu{ms}} = (\bfa{u}^m_{\tu{ms}}, \Phi^m_{\tu{ms}})\in \bfa{\mathcal{V}}^m_{\tu{ms}}$ from \eqref{adapeq}, and for all $\bfa{v} \in \bfa{V}_i\,,$ using the operators \eqref{bllforms}--\eqref{bllforms2}, we obtain the residuals as follows: 
\begin{align}\label{res1}
\begin{split}
(R^i_{\bfa{u}}(\bfa{v}))_1 &=  (L^i_{\bfa{u}}(\bfa{v}) -  a^{i,n}_{\bfa{u} \bfa{u}} (\bfa{u}^m_{\tu{ms}},\bfa{v}) - a^{i,n}_{\bfa{u} \Phi} (\Phi^m_{\tu{ms}},\bfa{v}))_1\\
&= \int_{w_i} \left( f_1 v_1  -  \frac{\xi^2}{1 - \sqrt{Q(\bfa{u}^n_{\tu{ms}},\Phi^n_{\tu{ms}})}} \frac{\partial u^m_{1,\tu{ms}}}{\partial x_k} \frac{\partial v_1}{\partial x_k}   - \frac{\xi^2 \bfa{\e}_{1k}\Phi^m_{\tu{ms}}}{1 - \sqrt{Q(\bfa{u}^n_{\tu{ms}},\Phi^n_{\tu{ms}})}} \, \frac{\partial v_1}{\partial x_k} \right)\dx\,, 
\end{split}
\end{align}
\begin{align}\label{res2}
\begin{split}
(R^i_{\bfa{u}}(\bfa{v}))_2 &=  (L^i_{\bfa{u}}(\bfa{v}) -  a^{i,n}_{\bfa{u} \bfa{u}} (\bfa{u}^m_{\tu{ms}},\bfa{v}) - a^{i,n}_{\bfa{u} \Phi} (\Phi^m_{\tu{ms}},\bfa{v}))_2\\
&= \int_{w_i} \left( f_2 v_2  -  \frac{\xi^2}{1 - \sqrt{Q(\bfa{u}^n_{\tu{ms}},\Phi^n_{\tu{ms}})}} \frac{\partial u^m_{2,\tu{ms}}}{\partial x_k} \frac{\partial v_2}{\partial x_k}   - \frac{\xi^2 \bfa{\e}_{2k}\Phi^m_{\tu{ms}}}{1 - \sqrt{Q(\bfa{u}^n_{\tu{ms}},\Phi^n_{\tu{ms}})}} \, \frac{\partial v_2}{\partial x_k} \right)\dx\,, 
\end{split}
\end{align}
and
\begin{align}\label{resphi}
\begin{split}
  R^i_{\Phi}(\Psi)&= L^i_{\Phi}(\Psi) - a^{i,n}_{\Phi \bfa{u}}(\bfa{u}^m_{\tu{ms}}, \Psi) - \tilde{a}^{i,n}_{\Phi \Phi} (\Phi^m_{\tu{ms}}, \Psi) - a^{i,n}_{\Phi \Phi}(\Phi^m_{\tu{ms}},\Psi)\\
  & = \int_{w_i} \left(g_3 \Psi - \frac{\bfa{\e}_{js} \xi^2 }{1 - \sqrt{Q(\bfa{u}^n_{\tu{ms}},\Phi^n_{\tu{ms}})}} \frac{\partial u^m_{j,\tu{ms}}}{\partial x_s} \Psi - \frac{2 \xi^2 \Phi^m_{\tu{ms}}}{1 - \sqrt{Q(\bfa{u}^n_{\tu{ms}},\Phi^n_{\tu{ms}})}} \Psi \right. \\ &\qquad \left.- \frac{\alpha^2}{1 - \sqrt{Q(\bfa{u}^n_{\tu{ms}},\Phi^n_{\tu{ms}})}} \frac{\partial \Phi^m_\tu{ms}}{\partial x_k}\frac{\partial \Psi}{\partial x_k}\right)\dx\,.
\end{split}
\end{align}
We denote the vector
\begin{equation}\label{resvec}
(R_1^i \quad R_2^i \quad R_3^i)^{\tu{T}} = ((R^i_{\bfa{u}})_1 \quad (R^i_{\bfa{u}})_2 \quad R^i_{\Phi})^{\tu{T}}\,. 
\end{equation}
Now, let $z_l \in V_i$ (for $l=1,2,3$) such that each $z_l$ be the unique weak solution of the Dirichlet Laplacian problem
\begin{align}\label{laplace}
\begin{split}
    -\Delta z_l &= R^i_l \quad \text{in } w_i\,,\\
    z_l &= 0 \quad \text{on } \partial w_i\,.
\end{split}    
\end{align}
Then, we get
\begin{equation}\label{dualnorm3}
r^i_l = \|R_l^i\|_{H^{-1}(w_i)} = \|z_l\|_{H^1_0(w_i)} = \left(\int_{w_i} |\nabla z_l|^2\right )^{1/2}= \left(\int_{w_i}   \sum_{k=1}^2 \Bigg |\frac{\partial z_l}{\partial x_k} \Bigg |^2 \dx\right)^{1/2} \,.
\end{equation}
Equivalently,
\begin{equation}\label{inverselap}
r^i_l = \|R_l^i\|_{H^{-1}(w_i)} = \left(\int_{w_i} ((-\Delta)^{-1} R_l^i) R_l^i \dx\right)^{1/2} \,.
\end{equation}
Finally, we obtain the residual
\begin{equation}\label{ressum}
r^i = \sqrt{(r^i_1)^2 + (r^i_2)^2 + (r_3^i)^2}= \sum_{l=1}^3 \|R_l^i\|_{H^{-1}(w_i)}\,.
\end{equation}

\subsubsection{Second approach to compute the dual norm}\label{dualway2}
Let us recall the function space $\bfa{\mathcal{V}}_i = V_i^3 =(H_0^1(w_i) \cap V_h)^3$ in the given coarse neighborhood $\omega_i\,,$ where $V_h$ was defined in \eqref{vh}. At the Picard iteration $n$th, with $\bfa{\nu} = (\bfa{v}, \Psi) \in \bfa{\mathcal{V}}_i\,,$ using the operator $a_n$ defined in \eqref{bin}, we equip the space $\bfa{\mathcal{V}}_i$ with the energy norm $\|\bfa{\nu}\|_{\bfa{\mathcal{V}}_i}$:
\begin{align}\label{anorm}
\begin{split}
\|\bfa{\nu}\|^2_{\bfa{\mathcal{V}}_i} &= a^i_n(\bfa{\nu},\bfa{\nu})\\
&= \int_{w_i} \left(   \frac{\xi^2}{1 - \sqrt{Q(\bfa{u}^n_{\tu{ms}},\Phi^n_{\tu{ms}})}} \left( \frac{\partial v_j}{\partial x_k} \right)^2  
+ \frac{\xi^2 \bfa{\e}_{jk}}{1 - \sqrt{Q(\bfa{u}^n_{\tu{ms}},\Phi^n_{\tu{ms}})}} \, \Psi \, \frac{\partial v_j}{\partial x_k} \right)\dx\\
& \quad + \int_{w_i} \left( \frac{\bfa{\e}_{js} \xi^2 }{1 - \sqrt{Q(\bfa{u}^n_{\tu{ms}},\Phi^n_{\tu{ms}})}} \, \Psi \, \frac{\partial v_j}{\partial x_s}  
+ \frac{2 \xi^2 }{1 - \sqrt{Q(\bfa{u}^n_{\tu{ms}},\Phi^n_{\tu{ms}})}} (\Psi)^2 \right.\\
& \left. \hspace{70pt} + \frac{\alpha^2}{1 - \sqrt{Q(\bfa{u}^n_{\tu{ms}},\Phi^n_{\tu{ms}})}} \left(\frac{\partial \Psi}{\partial x_k}\right)^2\right)\dx\,.
\end{split}
\end{align}

By Poincar{\'e} inequality, it follows from \eqref{anorm} that for all $\bfa{\nu} \in \bfa{\mathcal{V}}_i\,,$
\begin{align*}
\begin{split}
 \tilde{C}\| \nabla \bfa{\nu}\|^2_{\mathbb{L}^2(w_i)} \leq \|\bfa{\nu}\|^2_{\bfa{\mathcal{V}}_i} &= a^i_n(\bfa{\nu},\bfa{\nu}) \leq C  \| \nabla \bfa{\nu}\|^2_{\mathbb{L}^2(w_i)}\,, 
\end{split}  
\end{align*}
with some positive constants $C, \tilde{C}$ depending only on $w_i\,.$
Therefore, the norm $\|\bfa{\nu}\|_{\bfa{\mathcal{V}}_i}$ in \eqref{anorm} is equivalent to the norm 
\begin{equation}\label{equivnorms}
\| \bfa{\nu}\|_{\bfa{H}^1_0(w_i)} = \| \nabla \bfa{\nu}\|_{\mathbb{L}^2(w_i)} = \left(\int_{w_i} |\nabla \bfa{\nu}|^2\right )^{1/2}= \left(\int_{w_i}   \sum_{l=1}^3 \sum_{k=1}^2 \Bigg |\frac{\partial \nu_l}{\partial x_k} \Bigg |^2 \dx\right)^{1/2} \,,
\end{equation}
as in Section \ref{dualway1}.

Further, the norm $\|\bfa{\nu}\|_{\bfa{\mathcal{V}}_i}$ in \eqref{anorm} induces the following inner product (for all $\bfa{\Upsilon} = (\bfa{u}, \Phi)$ and $\bfa{\nu} = (\bfa{v},\Psi)$ in $\bfa{\mathcal{V}}_i$):
\begin{align}\label{aifor2}
\begin{split}
a_n^i(\bfa{\Upsilon},\bfa{\nu})
&=   (a^{i,n}_{\bfa{u} \bfa{u}} (\bfa{u},\bfa{v}) + a^{i,n}_{\bfa{u} \Phi} (\Phi,\bfa{v})) + (a^{i,n}_{\Phi \bfa{u}}(\bfa{u}, \Psi) + \tilde{a}^{i,n}_{\Phi \Phi} (\Phi, \Psi) + a^{i,n}_{\Phi \Phi}(\Phi,\Psi))\\
&= \int_{w_i} \left(   \frac{\xi^2}{1 - \sqrt{Q(\bfa{u}^n_{\tu{ms}},\Phi^n_{\tu{ms}})}} \frac{\partial u_j}{\partial x_k} \frac{\partial v_j}{\partial x_k}   + \frac{\xi^2 \bfa{\e}_{jk}\Phi}{1 - \sqrt{Q(\bfa{u}^n_{\tu{ms}},\Phi^n_{\tu{ms}})}} \, \frac{\partial v_j}{\partial x_k} \right)\dx\\
& \quad + \int_{w_i} \left( \frac{\bfa{\e}_{js} \xi^2 }{1 - \sqrt{Q(\bfa{u}^n_{\tu{ms}},\Phi^n_{\tu{ms}})}} \frac{\partial u_j}{\partial x_s} \Psi + \frac{2 \xi^2 \Phi}{1 - \sqrt{Q(\bfa{u}^n_{\tu{ms}},\Phi^n_{\tu{ms}})}} \Psi \right.\\
& \left. \hspace{70pt} + \frac{\alpha^2}{1 - \sqrt{Q(\bfa{u}^n_{\tu{ms}},\Phi^n_{\tu{ms}})}} \frac{\partial \Phi}{\partial x_k}\frac{\partial \Psi}{\partial x_k}\right)\dx\,.
\end{split}
\end{align}

Locally, at the coarse neighborhood $w_i\,,$ for all $\bfa{\nu} = (\bfa{v},\Psi)$ in $\bfa{\mathcal{V}}_i\,,$ we obtain the residual $R^i$ as a linear functional on $\bfa{\mathcal{V}}_i$:
\begin{align}\label{resa}
\begin{split}
R^i(\bfa{\nu})& = F^i(\bfa{\nu}) - a_n^i(\bfa{\Upsilon}^m_{\tu{ms}}, \bfa{\nu})\\
&=  (L^i_{\bfa{u}}(\bfa{v}) -  a^{i,n}_{\bfa{u} \bfa{u}} (\bfa{u}^m_{\tu{ms}},\bfa{v}) - a^{i,n}_{\bfa{u} \Phi} (\Phi^m_{\tu{ms}},\bfa{v})) \\
& \quad + (L^i_{\Phi}(\Psi) - a^{i,n}_{\Phi \bfa{u}}(\bfa{u}^m_{\tu{ms}}, \Psi) - \tilde{a}^{i,n}_{\Phi \Phi} (\Phi^m_{\tu{ms}}, \Psi) - a^{i,n}_{\Phi \Phi}(\Phi^m_{\tu{ms}},\Psi))\\
&= \int_{w_i} \left( f_j v_j  -  \frac{\xi^2}{1 - \sqrt{Q(\bfa{u}^n_{\tu{ms}},\Phi^n_{\tu{ms}})}} \frac{\partial u^m_{j,\tu{ms}}}{\partial x_k} \frac{\partial v_j}{\partial x_k}   - \frac{\xi^2 \bfa{\e}_{jk}\Phi^m_{\tu{ms}}}{1 - \sqrt{Q(\bfa{u}^n_{\tu{ms}},\Phi^n_{\tu{ms}})}} \, \frac{\partial v_j}{\partial x_k} \right)\dx\\
& \quad + \int_{w_i} \left(g_3 \Psi - \frac{\bfa{\e}_{js} \xi^2 }{1 - \sqrt{Q(\bfa{u}^n_{\tu{ms}},\Phi^n_{\tu{ms}})}} \frac{\partial u^m_{j,\tu{ms}}}{\partial x_s} \Psi - \frac{2 \xi^2 \Phi^m_{\tu{ms}}}{1 - \sqrt{Q(\bfa{u}^n_{\tu{ms}},\Phi^n_{\tu{ms}})}} \Psi \right.\\
& \left. \hspace{70pt} - \frac{\alpha^2}{1 - \sqrt{Q(\bfa{u}^n_{\tu{ms}},\Phi^n_{\tu{ms}})}} \frac{\partial \Phi^m_\tu{ms}}{\partial x_k}\frac{\partial \Psi}{\partial x_k}\right)\dx\,.
\end{split}
\end{align}

With the function $R^i(\bfa{\nu})$ defined in \eqref{resa}, to compute $\| R^i\|_{\bfa{\mathcal{V}}_i^*}\,,$ by Riesz representation theorem, we can first obtain the unique weak solution $\bfa{\eta}_i^{\tu{on}} \in \bfa{\mathcal{V}}_i$ from the problem 
\begin{align}\label{riesz2}
\begin{split}
a_n^i(\bfa{\eta}_i^{\tu{on}},\bfa{\nu})= R^i(\bfa{\nu})\,,
\end{split}
\end{align}
for all $\bfa{\nu} \in \bfa{\mathcal{V}}_i\,,$ using \eqref{aifor2}.

Then, the norm of $R^i$ is denoted by $r_i$ and calculated by
\begin{equation}\label{enorm2}
r_i = \| R^i\|_{\bfa{\mathcal{V}}_i^*}= \|\bfa{\eta}_i^{\tu{on}}\|_{\bfa{\mathcal{V}}_i}\,,
\end{equation}
which is defined in \eqref{anorm}.

\subsection{Online adaptive algorithm}\label{onalg}

Assume that we are at the Picard step $n$th (for $n=0,1,2, \cdots)$ and given $\bfa{\Upsilon}^n_{\tu{ms}}=(\bfa{u}^n_{\tu{ms}},\Phi^n_{\tu{ms}})$ (see Section \ref{onpi}). 
Set $m = 0$ and pick a parameter $\theta \in (0, 1]$. 
We denote $\bfa{\mathcal{V}}^m_{\tu{ms}} = \bfa{\mathcal{V}}^0_{\tu{ms}} = \bfa{\mathcal{V}}_{\tu{off}}$ (where the space notation $\bfa{\mathcal{V}}$ was defined at the beginning of Sections \ref{approx} and \ref{online}).  
For each $m \in \mathbb{N}\,,$ that is, $m=0,1,2, \cdots\,,$ assume that $\bfa{\mathcal{V}}^m_{\tu{ms}}$ is provided. 
To obtain the new multiscale finite element space $\bfa{\mathcal{V}}^{m+1}_{\tu{ms}}\,,$ we operate the following procedure. 

\textbf{Step 1:} Solve for the GMsFEM solution $\bfa{\Upsilon}^m_{\tu{ms}} = (\bfa{u}^m_{\tu{ms}}, \Phi^m_{\tu{ms}})\in \bfa{\mathcal{V}}^m_{\tu{ms}}$  from a similar equation to \eqref{absvar}:
\begin{equation}\label{adapeq}
a_n (\bfa{\Upsilon}^m_{\tu{ms}}, \bfa{\nu}) = F(\bfa{\nu}), \quad \text{for all } \bfa{\nu} \in \bfa{\mathcal{V}}_{\tu{ms}}^m\,.
\end{equation}

\textbf{Step 2:} For each $i = 1, ..., N_v\,,$ using \eqref{enorm2}, compute the residual $r_i$ on the coarse region $\omega_i\,.$ Assume that we obtain 
\begin{equation*}
r_1 \geq r_2 \geq \cdots \geq r_{N_v}\,,
\end{equation*}

\textbf{Step 3:} Choose the smallest integer $k_p$ such that

\begin{equation*}
\theta \sum_{i = 1}^{N_v} r_i^2 \leq \sum_{i = 1}^{k_p} r_i^2\,.
\end{equation*}

Now, for $i = 1, ..., k_p$, add basis functions $\boldsymbol{\eta}_i^{\tu{on}} =(\psi_{1,i}, \psi_{2,i}, \phi_i) \in \bfa{\mathcal{V}}_i$ (solved from \eqref{riesz2}) to the space $\bfa{\mathcal{V}}^m_{\tu{ms}}\,.$  The new multiscale space is defined by $\bfa{\mathcal{V}}^{m+1}_{\tu{ms}}$ in the following manner:

\begin{equation*}
\begin{split}
\bfa{\mathcal{V}}^{m+1}_{\tu{ms}} = \bfa{\mathcal{V}}^m_{\tu{ms}} \oplus \text{span} \{ \boldsymbol{\eta}_i^{\tu{on}}: 1 \leq i \leq k_p \}\,.
\end{split}
\end{equation*}
Let $N_{\tu{it}} = m+1\,.$

\textbf{Step 4:} If $N_{\tu{it}} \geq N_{\tu{max}}\,,$ then stop. Otherwise, set $m \leftarrow m + 1$ and go back to \textbf{Step 1}.


\begin{remark}\label{rmk_unif}
 The enrichment is considered ``uniform'' if $\theta =1\,.$ The term ``adaptive'' is also used for the situation $\theta=1$ in this work, depending on the context.
\end{remark}

\subsection{Online GMsFEM for strain-limiting Cosserat elasticity}\label{onpi}

We provide the main procedure for applying the GMsFEM to solve the planar strain-limiting Cosserat problem \eqref{bof}--\eqref{bom}: pick a basis update tolerance value $\delta \in \mathbb{R}_+$ and Picard iteration termination tolerance value $\delta_0 \in \mathbb{R}_+$. Also, we guess an initial guess $\bfa{\Upsilon}_{\tu{ms}}^{\tu{old}} = (\boldsymbol{u}_{\tu{ms}}^{\tu{old}}, \Phi_{\tu{ms}}^{\tu{old}})\,,$ compute $\kappa^{\tu{old}}(\boldsymbol{x}) = \dfrac{1}{1 - \sqrt{Q(\boldsymbol{u}_{\tu{ms}}^{\tu{old}}, \Phi_{\tu{ms}}^{\tu{\tu{old}}})}}$ and construct the multiscale space $\bfa{\mathcal{V}}^{\tu{pre}}_{\tu{ms}}$ (by the Online adaptive enrichment algorithm in Section \ref{onalg} starting from $n=0$). Then, we repeat the following steps:

\textbf{Step 1:} Solve for $\bfa{\Upsilon}_{\tu{ms}}^{\tu{new}} \in \bfa{\mathcal{V}}^{\tu{pre}}_{\tu{ms}}$ from the following equation (as \eqref{absvar}):

\begin{equation}\label{pieq}
a_{\tu{old}} (\bfa{\Upsilon}_{\tu{ms}}^{\tu{new}}, \boldsymbol{\nu}) = F(\boldsymbol{\nu})\,, \quad \forall \bfa{\nu} \in \bfa{\mathcal{V}}^{\tu{pre}}_{\tu{ms}}\,, 
\end{equation}
where $\bfa{\Upsilon}_{\tu{ms}}^{\tu{new}} = (\boldsymbol{u}_{\tu{ms}}^{\tu{new}}, \Phi_{\tu{ms}}^{\tu{new}})\,,$ $\boldsymbol{\nu} = (\boldsymbol{v}, \Psi)$, the operators $a_{\tu{old}}$ and $F$ were defined in \eqref{bin}. 

If $\dfrac{\| \boldsymbol{\Upsilon}_{\tu{ms}}^{\tu{new}} - \boldsymbol{\Upsilon}_{\tu{ms}}^{\tu{old}} \|_{\bfa{\mathcal{V}}_h}}{\| \boldsymbol{\Upsilon}_{\tu{ms}}^{\tu{old}} \|_{\bfa{\mathcal{V}}_h}} > \delta_0$ (with the norm $\| \cdot \|_{\bfa{\mathcal{V}}_h}$ defined in \eqref{annorm} and $\bfa{\mathcal{V}}_h$ defined in \eqref{vh}), let $\boldsymbol{\Upsilon}_{\tu{ms}}^{\tu{old}} = \boldsymbol{\Upsilon}_{\tu{ms}}^{\tu{new}}$ and go to \textbf{Step 2}.

\textbf{Step 2:} Calculate $\kappa^{\tu{new}}(\boldsymbol{x}) = \dfrac{1}{1 - \sqrt{Q(\boldsymbol{u}_{\tu{ms}}^{\tu{new}}, \Phi_{\tu{ms}}^{\tu{new}})}}$. If $\dfrac{\| \kappa^{\tu{new}}(\bfa{x}) - \kappa^{\tu{old}}(\bfa{x})\|_{L^2(\Omega)}}{\| \kappa^{\tu{old}}(\bfa{x}) \|_{L^2(\Omega)}} > \delta$, build the new space $\bfa{\mathcal{V}}_{\tu{ms}}^{\tu{af}}$ (by the Online adaptive algorithm in Section \ref{onalg}), let $\bfa{\mathcal{V}}_{\tu{ms}}^{\tu{pre}} = \bfa{\mathcal{V}}_{\tu{ms}}^{\tu{af}}\,.$  Set $\kappa^{\tu{old}}(\bfa{x}) = \kappa^{\tu{new}}(\bfa{x})\,,$ then go to \textbf{Step 1}.

\bigskip

The final multiscale basis function space is denoted by $\bfa{\mathcal{V}}_{\tu{ms}}\,.$


\section{Numerical results}\label{num_results}

With a number of test cases, this section demonstrates the performance of our multiscale approaches  for the planar strain-limiting Cosserat elasticity problems \eqref{bof}--\eqref{bom} in various media types. In particular, we consider perforated, composite, and heterogeneous media. Also, small and big strain-limiting parameters are taken into account. We solve these model problems using the proposed multiscale approaches, that is, offline, online uniform ($\theta=1$), and online adaptive ($\theta=0.8$) multiscale methods. Our computational domain is $\Omega = [0, 2]\times [0, 2]\,,$ having a uniform coarse grid with 100 rectangular cells, 121 vertices, and coarse-mesh size $H=1/5\,.$  The reference solutions are obtained by employing the finite element method (FEM) on fine grids, which will be specified below.

For all model problems, we fix the displacements and microrotations on the left, right, bottom, and top boundaries. In the case of perforated media, we set zero Neumann boundary conditions on the perforations. Thus, the boundary conditions are as follows (for $i,j=2$):
\begin{equation}\label{eq:num_results_bcs}
\begin{split}
&u_1 = 0, \quad u_2 = 0, \quad \Phi = 0, \quad \boldsymbol{x} \in \Gamma_D,\\
&\tau_{ij} n_i = 0, \quad m_{i3} n_i = 0, \quad \boldsymbol{x} \in \Omega \setminus \Gamma_D,
\end{split}
\end{equation}
where $\Gamma_D = \Gamma_L \cup \Gamma_R \cup \Gamma_B \cup \Gamma_T\,,$ $\partial \Omega = \Gamma_D \cup \Gamma_N\,,$ and $\bfa{u} = (u_1,u_2,0)\,, \bfa{\phi} = (0,0, \Phi)$ in \eqref{planar}.

In all types of media and strain-limiting cases, for $i=1,2\,,$  we use the source terms
\begin{equation}\label{eq:rhs}
f_i(\boldsymbol{x}) = f_v \sqrt{x_1^2 + x_2^2 + 1}, \quad
g_3(\boldsymbol{x}) = g_v \sqrt{x_1^2 + x_2^2 + 1}\,,
\end{equation}
where $f_v$ together with $g_v$ are model parameters and can be different in each problem.

As we mentioned in the previous sections, in order to handle the nonlinearity of our strain-limiting Cosserat model, the Picard iteration method is employed. The tolerance in the Picard stopping criterion is set to $\delta_0 = 10^{-5}\,,$ which can guarantee the linearization procedure's convergence. We choose $\boldsymbol{u}^0 = (10^{-6}, 10^{-6})$ and $\Phi^{0} = 10^{-6}$ as initial guesses. These settings of Picard's method are applied to all model problems.

Note that it seems not possible to find analytical solutions for nonlinear strain-limiting Cosserat elasticity problems in heterogeneous media.
%
Therefore, we use FEM solutions on fine grids as reference ones. At the final $(n+1)$th Picard iterative step, to compare the FEM reference $\bfa{\Upsilon}^{\tu{ref}} = (\bfa{u}^{\tu{ref}}, \Phi^{\tu{ref}})$ (as $\bfa{\Upsilon}_h$ from \eqref{absvar}) with the multiscale solutions $\bfa{\Upsilon}^{\tu{ms}} = (\bfa{u}^{\tu{ms}}, \Phi^{\tu{ms}})$ (as $\bfa{\Upsilon}_{\tu{off}}$ in \eqref{solmsoff} or $\bfa{\Upsilon}_{\tu{ms}}$ from \eqref{pieq}), we use the following relative $L^2$ errors
\begin{equation}\label{eq:L2_errors}
e^{u}_{L^2} = \frac{\sqrt{\int_\Omega (\boldsymbol{u}^{\tu{ref}} - \boldsymbol{u}^{\tu{ms}})^2 d\boldsymbol{x}}}{\sqrt{\int_\Omega (\boldsymbol{u}^{\tu{ref}})^2 d\boldsymbol{x}}}\,, \quad 
e^{\Phi}_{L^2} = \frac{\sqrt{\int_\Omega (\Phi^{\tu{ref}} - \Phi^{\tu{ms}})^2 d\boldsymbol{x}}}{\sqrt{\int_\Omega (\Phi^{\tu{ref}})^2 d\boldsymbol{x}}}\,,
\end{equation}
and relative $H^1$ errors
\begin{equation}\label{eq:H1_errors}
e^{u}_{H^1} = \frac{\sqrt{a_{\boldsymbol{uu}}^{n + 1} (\boldsymbol{u}^{\tu{ref}} - \boldsymbol{u}^{\tu{ms}}, \boldsymbol{u}^{\tu{ref}} - \boldsymbol{u}^{\tu{ms}})}}{\sqrt{a_{\boldsymbol{uu}}^{n + 1} (\boldsymbol{u}^{\tu{ref}}, \boldsymbol{u}^{\tu{ref}})}}\,, \quad 
e^{\Phi}_{H^1} = \frac{\sqrt{a_{\Phi\Phi}^{n + 1} (\Phi^{\tu{ref}} - \Phi^{\tu{ms}}, \Phi^{\tu{ref}} - \Phi^{\tu{ms}})}}{\sqrt{a_{\Phi\Phi}^{n + 1} (\Phi^{\tu{ref}}, \Phi^{\tu{ref}})}}\,,
\end{equation}
where the superscripts ``ref'' and ``ms'' denote the reference and multiscale solutions, respectively, and the bilinear forms $a_{\boldsymbol{uu}}^{n + 1}$ and $a_{\Phi \Phi}^{n + 1}$ are defined in \eqref{bllforms} using data from the final $(n+1)$th Picard step.

The numerical implementation of the fine-grid and multiscale methods is based on the FEniCS computational package \cite{logg2012automated}. For grid generation, we use the program Gmsh \cite{Geuzaine2009}. Visualization of numerical results is performed using the ParaView software \cite{ahrens2005paraview}.

\subsection{Perforated media}\label{perform}


\begin{figure}[!htbp]
    \centering
    \begin{subfigure}[b]{\textwidth}
        \centering
        \includegraphics[width=0.7\textwidth]{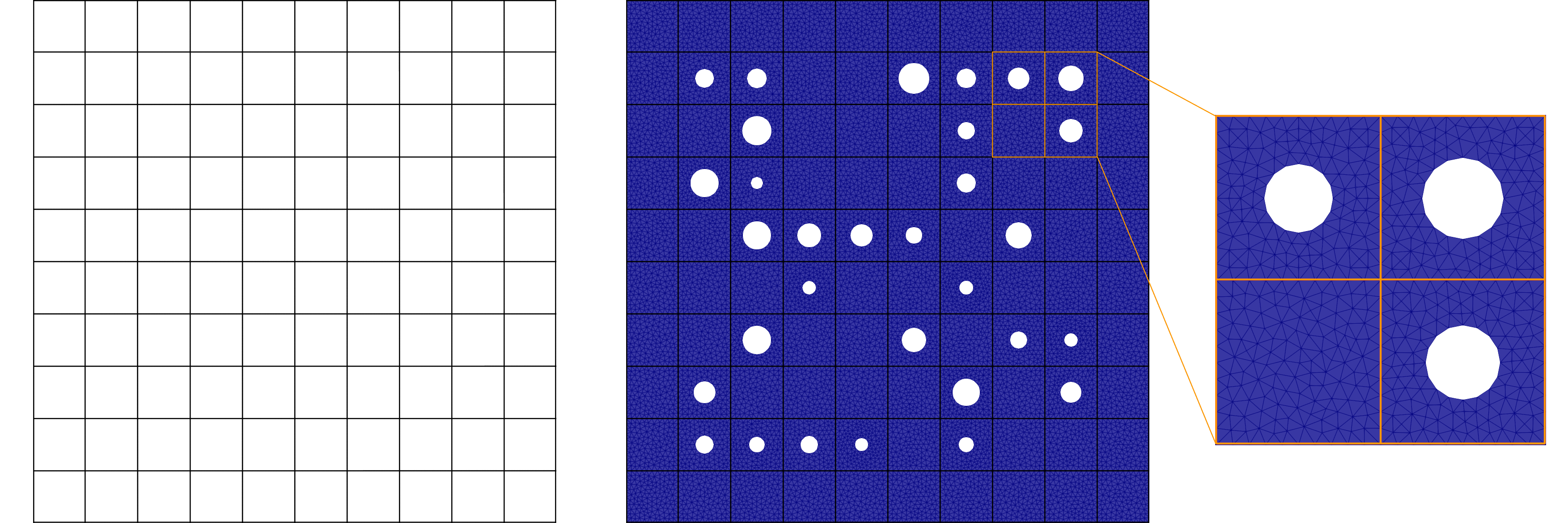}
        \caption{Case 1 (small strain-limiting parameter).}
        \label{fig:p_mesh_1}
    \end{subfigure}
    \begin{subfigure}[b]{\textwidth}
        \centering
        \includegraphics[width=0.7\textwidth]{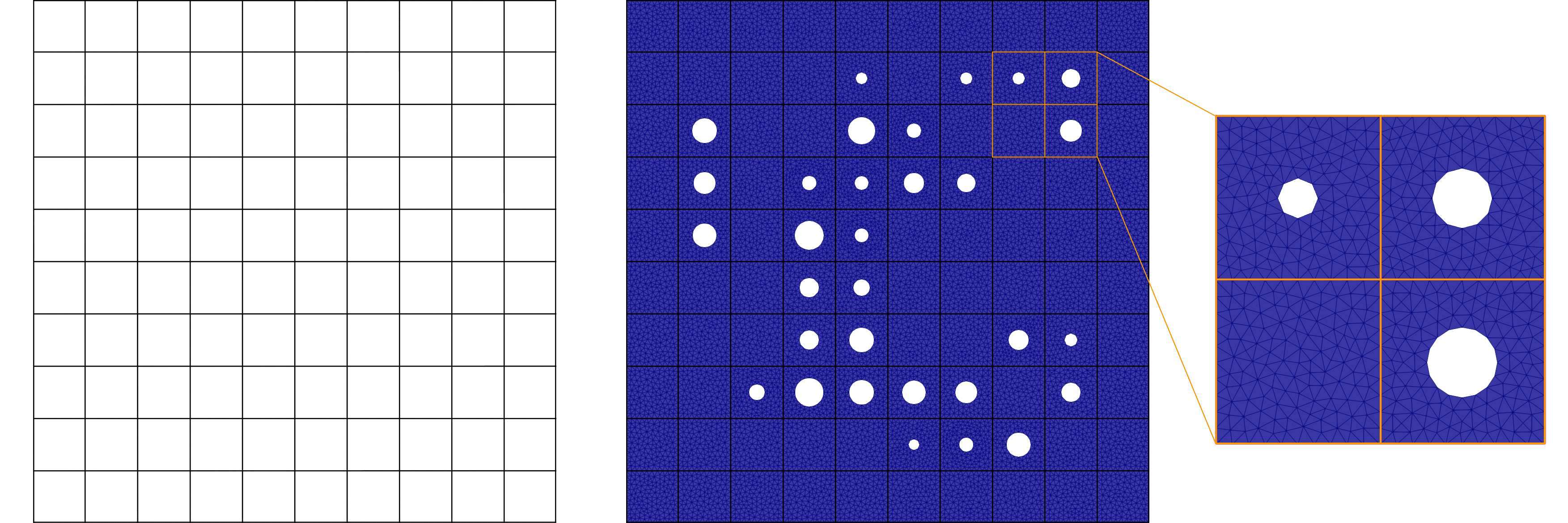}
        \caption{Case 2 (big strain-limiting parameter).}
        \label{fig:p_mesh_2}
    \end{subfigure}
    \caption{Computational grids for the perforated media: left -- the coarse grid (each small square is a coarse block), middle -- the fine grid, right -- a local domain.}
    \label{fig:p_meshes}
\end{figure}

Herein, we consider perforated Cosserat media.  
As mentioned earlier, we consider two cases of strain-limiting parameters: (1) small and (2) big. The following unstructured fine grids are utilized for each strain-limiting case (see Figure \ref{fig:p_meshes}).

\begin{itemize}
    \item Case 1: 25 386 triangular fine cells and 13 089 vertices.
    \item Case 2: 25 572 triangular fine cells and 13 166 vertices.
\end{itemize}

The inhomogeneity of the problems arises from the presence of perforations, while being outside of them, the material properties are constant. On the perforations themselves, we set zero Neumann boundary conditions. The following parameters are used in the model problems.

\begin{itemize}
    \item Case 1: $\xi = 1$, $\alpha = 1.1$, $\beta = 1$, $f_v = 0.29$, and $g_v = 0.3$.
    \item Case 2: $\xi = 1.2$, $\alpha = 1.3$, $\beta = 10^4$, $f_v = 3.5 \cdot 10^{-5}$, and $g_v = 3.6 \cdot 10^{-5}$.
\end{itemize}


\begin{figure}[!htbp]
    \centering
    \begin{subfigure}[b]{\textwidth}
        \centering
        \includegraphics[width=\textwidth]{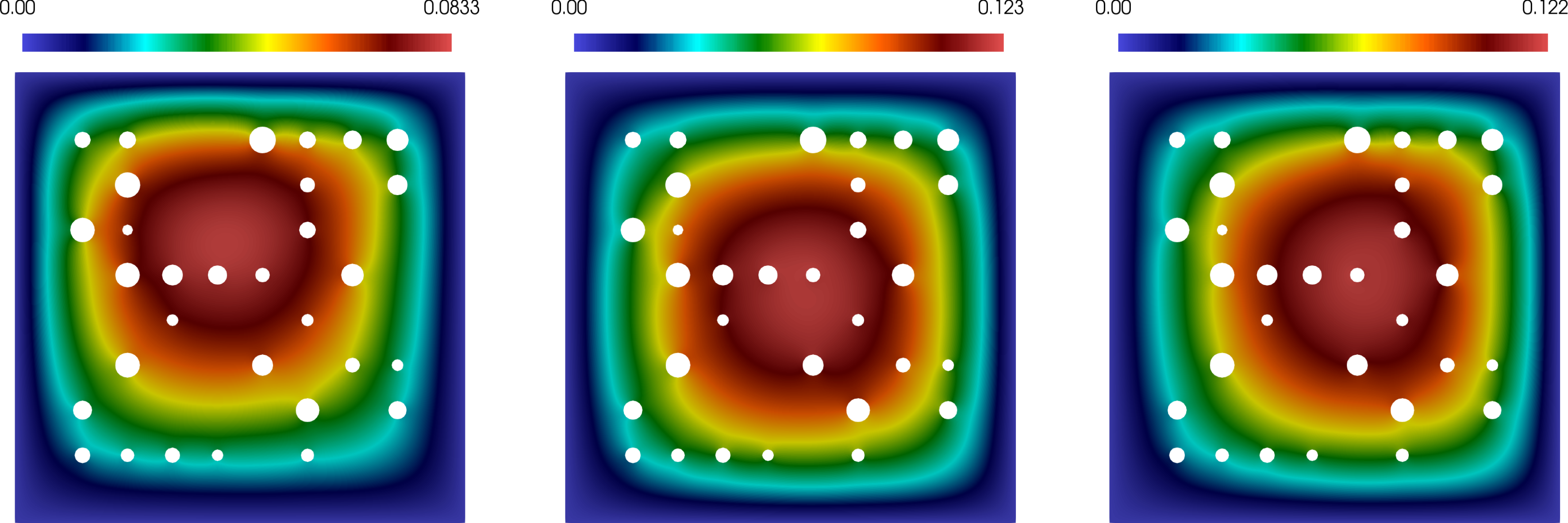}
        \caption{Reference solution.}
        \label{fig:p_results_fine}
    \end{subfigure}
    \begin{subfigure}[b]{\textwidth}
        \centering
        \includegraphics[width=\textwidth]{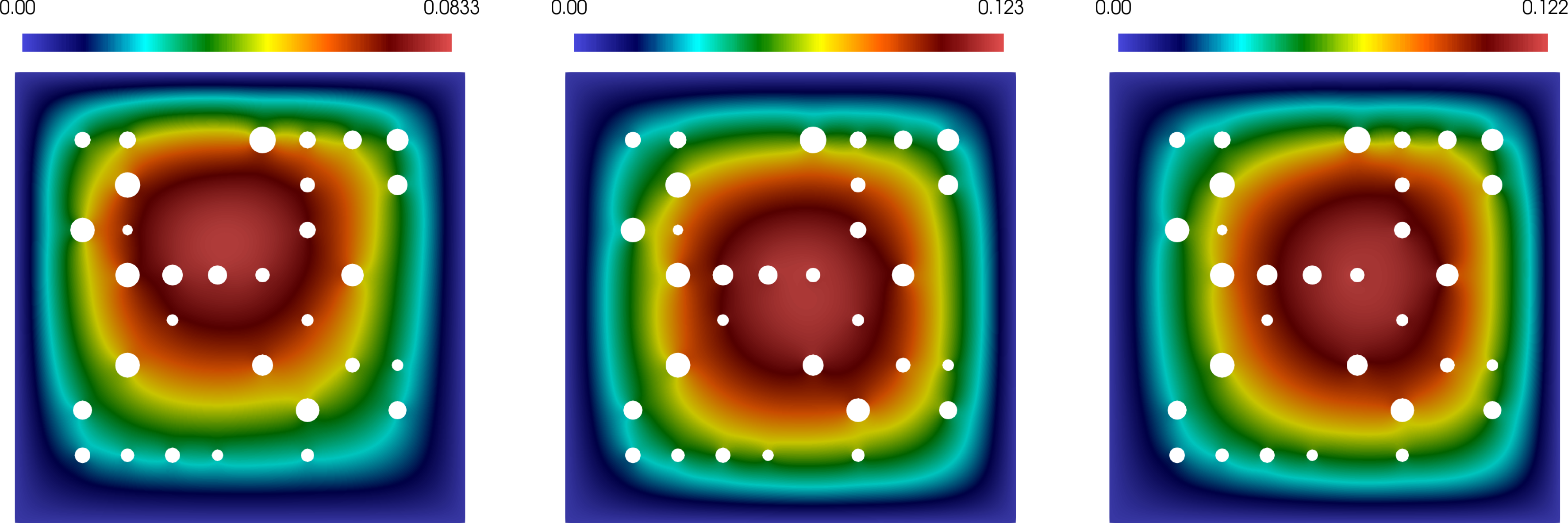}
        \caption{Online adaptive multiscale solution with $N_{\tu{b}} = 4$, $N_{\tu{it}} = 3$, $\theta = 0.8$, and $\delta = 0.1$\,.}
        \label{fig:p_results_online}
    \end{subfigure}
    \begin{subfigure}[b]{\textwidth}
        \centering
        \includegraphics[width=\textwidth]{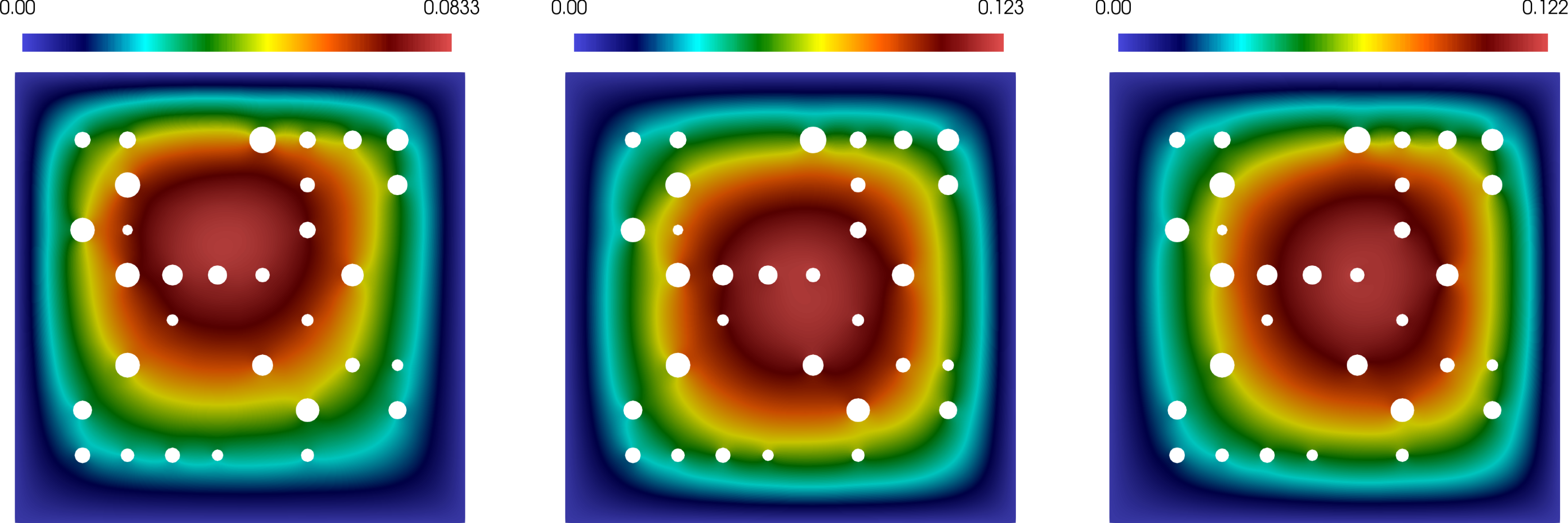}
        \caption{Offline multiscale solution with $N_{\tu{b}} = 4$\,.}
        \label{fig:p_results_offline}
    \end{subfigure}
    \caption{Distributions of the microrotations $\Phi$ and displacements $u_1$ and $u_2$ (from left to right) for Case 1 (small strain-limiting parameter) of the perforated media.}
    \label{fig:p_results}
\end{figure}

Figure \ref{fig:p_results} shows the numerical results of the model problem for Case 1. We present the distributions of microrotations $\Phi$ and displacements $u_1$ and $u_2$ from left to right. Figure \ref{fig:p_results_fine} depicts the reference solution on a fine grid. In Figure \ref{fig:p_results_online}, we show the online adaptive multiscale solution with four offline and three online basis functions per each coarse neighborhood, adaptivity parameter $\theta = 0.8$, and online basis function update criterion $\delta = 0.1$. In Figure \ref{fig:p_results_offline}, we depict the offline multiscale solution with four basis functions per coarse-grid node. One can see that all figures are similar, indicating that the proposed multiscale approaches can approximate the reference solution with high accuracy.

Regarding the modeled process, we see from Figure \ref{fig:p_results} that on the outer boundaries, the microrotations and displacements are fixed according to our boundary conditions. It is also clear that the perforations add inhomogeneity to the distribution of solutions. The most significant values of the fields are observed in the central area. In general, the numerical solutions are consistent with the modeling parameters and conditions.


Note that in all media types, we present the distributions of solutions only for Case 1. At the same time, we analyze the errors for all strain-limiting cases. The errors are computed for different numbers of offline and online basis functions per coarse-grid node. For online approaches, we compare the errors of the uniform method (having $\theta=1$) with the adaptive method (possessing $\theta = 0.8$). Also, different values of $\delta$ are taken into account for updating the online basis functions.


\begin{table}[!htbp]
\caption{Relative errors of the offline method for the case 1 of the perforated media.}
\label{tab:offline_errors_p_case_1}
\begin{center}
\begin{tabular}{ | c | c | c c | c c |}
\hline
\multirowcell{2}{$N_{\tu{b}}$}
& \multirowcell{2}{$\tu{DOF}_{\tu{H}}$}
& \multicolumn{2}{c|}{Displacement}
& \multicolumn{2}{c|}{Microrotation}
\\
\cline{3-6}
&
& $e^{u}_{L^2}$
& $e^{u}_{H^1}$
& $e^{\Phi}_{L^2}$ 
& $e^{\Phi}_{H^1}$ 
\\
\hline
 1   & 363 & 8.790890e-03 & 1.188050e-01 & 1.105500e-02 & 1.333090e-01\\
 2   & 726 & 7.919140e-03 & 1.136200e-01 & 9.840610e-03 & 1.287110e-01\\
 4   & 1452 & 2.327710e-03 & 2.878240e-02 & 2.580910e-03 & 3.180650e-02\\
 8   & 2904 & 9.246960e-04 & 1.889780e-02 & 1.077210e-03 & 2.116530e-02\\
\hline
\end{tabular}
\end{center}
\end{table}

Let us examine the errors with Case 1 (small strain-limiting parameter). Table \ref{tab:offline_errors_p_case_1} presents the errors for the offline method with different numbers of offline basis functions $N_{\tu{b}}$ per coarse-grid node. One can see that the errors are minor. With two offline basis functions, one can achieve $L^2$ errors less than 1\% for all fields. However, relative $H^1$ errors are more considerable. When using eight basis functions, we get $H^1$ errors around 2\% for all fields. This phenomenon is explained by the fact that these errors correspond to gradient errors. Note that all errors decrease as the number of basis functions increases. Therefore, the solution convergence is related to the number of basis functions.

\begin{table}[!htbp]
\caption{Relative errors of the online method with $\theta = 1$ for the case 1 of the perforated media.}
\label{tab:online_errors_p_case_1_theta_1}
\begin{center}
\begin{tabular}{ | c | c | c c | c c | }
\hline
\multirowcell{2}{$N_{\tu{b}} + N_{\tu{it}}$}
& \multirowcell{2}{$\tu{DOF}_{\tu{H}}$}
& \multicolumn{2}{c|}{Displacement}
& \multicolumn{2}{c|}{Microrotation}
\\
\cline{3-6}
&
& $e^{u}_{L^2}$
& $e^{u}_{H^1}$
& $e^{\Phi}_{L^2}$ 
& $e^{\Phi}_{H^1}$ 
\\
\hline
\multicolumn{6}{|c|}{$\delta = \infty$}
\\
\hline
 1 + 3 & 1452 & 2.293390e-03 & 2.252810e-02 & 2.312540e-03 & 2.132750e-02\\
 2 + 3 & 1815 & 1.974990e-03 & 2.165180e-02 & 1.917930e-03 & 2.064000e-02\\
 4 + 3 & 2541 & 1.302500e-03 & 1.794700e-02 & 1.342190e-03 & 1.816850e-02\\
 8 + 3 & 3993 & 7.425300e-04 & 1.365740e-02 & 7.438900e-04 & 1.295070e-02\\
 1 + 4 & 1815 & 2.089160e-03 & 2.206200e-02 & 2.002210e-03 & 2.068080e-02\\
\hline
\multicolumn{6}{|c|}{$\delta = 0.4$}
\\
\hline
 1 + 3 & 1452 & 3.887440e-03 & 2.946610e-02 & 5.033600e-03 & 3.333880e-02\\
 2 + 3 & 1815 & 3.614680e-03 & 2.724770e-02 & 3.903850e-03 & 2.791010e-02\\
 4 + 3 & 2541 & 3.750810e-04 & 7.430590e-03 & 3.956140e-04 & 7.451690e-03\\
 8 + 3 & 3993 & 1.718010e-04 & 4.694270e-03 & 1.919810e-04 & 5.223890e-03\\
 1 + 4 & 1815 & 2.579030e-03 & 2.416660e-02 & 3.950470e-03 & 3.025680e-02\\
\hline
\multicolumn{6}{|c|}{$\delta = 0.2$}
\\
\hline
 1 + 3 & 1452 & 1.563150e-03 & 1.224960e-02 & 1.854350e-03 & 1.205780e-02\\
 2 + 3 & 1815 & 1.246530e-03 & 1.167840e-02 & 1.314910e-03 & 1.142790e-02\\
 4 + 3 & 2541 & 7.634300e-04 & 9.951380e-03 & 7.917770e-04 & 1.025430e-02\\
 8 + 3 & 3993 & 3.381710e-04 & 7.389850e-03 & 3.401310e-04 & 7.258190e-03\\
 1 + 4 & 1815 & 1.333490e-03 & 1.190830e-02 & 1.450690e-03 & 1.142320e-02\\
\hline
\multicolumn{6}{|c|}{$\delta = 0.1$}
\\
\hline
 1 + 3 & 1452 & 2.909710e-03 & 1.200190e-02 & 2.601300e-03 & 1.097140e-02\\
 2 + 3 & 1815 & 2.261950e-03 & 1.021030e-02 & 2.719220e-03 & 1.221250e-02\\
 4 + 3 & 2541 & 2.073050e-04 & 3.791500e-03 & 2.437850e-04 & 4.117560e-03\\
 8 + 3 & 3993 & 1.164450e-04 & 3.185930e-03 & 1.256850e-04 & 3.282210e-03\\
 1 + 4 & 1815 & 1.855970e-03 & 9.952830e-03 & 1.729650e-03 & 9.196910e-03\\
\hline
\end{tabular}
\end{center}
\end{table}

Table \ref{tab:online_errors_p_case_1_theta_1} presents the errors of the online uniform method with $\theta = 1$ for the Case 1 of the perforated media. Here, $N_{\tu{it}}$ denotes the number of inner online iterations, which is equal to the number of online basis functions in the uniform case. We consider different values of $\delta$ for updating the online basis functions. Note that $\delta=\infty$ corresponds to one computation of the online basis functions without further updates. As expected, using online basis functions can significantly improve the accuracy of multiscale modeling. For example, when using four offline and three online basis functions with $\delta=0.4$, the $L^2$ together with $H^1$ errors are less than 1\%\,, and $L^2$ errors are less than 0.1\%.

As desired, we observe a decrease in the error as $\delta$ decreases. Smaller values of $\delta$ correspond to more frequent updates of the online basis functions, leading to a more detailed account of changes in the properties of the medium. One notable observation is the comparison of errors of the solutions using $N_{\tu{b}} = 2$ with $N_{\tu{it}} = 3$ and $N_{\tu{b}} = 1$ with $N_{\tu{it}} = 4$. These solutions have the same number of degrees of freedom $\tu{DOF}_{\tu{H}} = 1815$. In the case of $\delta = \infty$, we get a more accurate solution when using $N_{\tu{b}} = 2$ with $N_{\tu{it}} = 3$ than when using $N_{\tu{b}} = 1$ with $N_{\tu{it}} = 4\,.$ This phenomenon is due to the property of the online method, which may require more than one offline basis function for better accuracy. However, as $\delta$ decreases, the difference between the solutions becomes less noticeable. This is because more frequent updating of the online basis functions reduces the error caused by nonlinear changes in the medium properties.

\begin{table}[!htbp]
\caption{Relative errors of the online method with $\theta = 0.8$ for the case 1 of the perforated media.}
\label{tab:online_errors_p_case_1_theta_0.8}
\begin{center}
\begin{tabular}{ | c | c | c c | c c | }
\hline
\multirowcell{2}{$N_{\tu{b}} + N_{\tu{it}}$}
& \multirowcell{2}{$\tu{DOF}_{\tu{H}}$}
& \multicolumn{2}{c|}{Displacement}
& \multicolumn{2}{c|}{Microrotation}
\\
\cline{3-6}
&
& $e^{u}_{L^2}$
& $e^{u}_{H^1}$
& $e^{\Phi}_{L^2}$ 
& $e^{\Phi}_{H^1}$ 
\\
\hline
\multicolumn{6}{|c|}{$\delta = \infty$}
\\
\hline
 1 + 3 & 918 & 2.740720e-03 & 2.499600e-02 & 2.730260e-03 & 2.372540e-02\\
 2 + 3 & 1332 & 2.355740e-03 & 2.337180e-02 & 2.294790e-03 & 2.257730e-02\\
 4 + 3 & 2046 & 1.574790e-03 & 1.865390e-02 & 1.510320e-03 & 1.872400e-02\\
 8 + 3 & 3531 & 8.028730e-04 & 1.417650e-02 & 7.956290e-04 & 1.335410e-02\\
 1 + 4 & 1074 & 2.557320e-03 & 2.323330e-02 & 2.614950e-03 & 2.208280e-02\\
\hline
\multicolumn{6}{|c|}{$\delta = 0.4$}
\\
\hline
 1 + 3 & 918 & 8.074180e-03 & 4.347730e-02 & 1.048930e-02 & 4.866860e-02\\
 2 + 3 & 1332 & 5.527070e-03 & 3.559860e-02 & 6.928700e-03 & 3.968600e-02\\
 4 + 3 & 2046 & 6.749050e-04 & 9.323000e-03 & 7.635170e-04 & 1.027090e-02\\
 8 + 3 & 3531 & 2.153540e-04 & 5.400450e-03 & 2.606410e-04 & 6.214200e-03\\
 1 + 4 & 1074 & 7.072420e-03 & 4.118610e-02 & 9.615810e-03 & 4.664350e-02\\
\hline
\multicolumn{6}{|c|}{$\delta = 0.2$}
\\
\hline
 1 + 3 & 921 & 2.001750e-03 & 1.484240e-02 & 2.214470e-03 & 1.453870e-02\\
 2 + 3 & 1332 & 1.843170e-03 & 1.503500e-02 & 1.896930e-03 & 1.474460e-02\\
 4 + 3 & 2046 & 1.247520e-03 & 1.092100e-02 & 1.293610e-03 & 1.131100e-02\\
 8 + 3 & 3537 & 3.967820e-04 & 7.744780e-03 & 3.853940e-04 & 7.493360e-03\\
 1 + 4 & 1095 & 1.950220e-03 & 1.322370e-02 & 2.199220e-03 & 1.317400e-02\\
\hline
\multicolumn{6}{|c|}{$\delta = 0.1$}
\\
\hline
 1 + 3 & 921 & 4.707950e-03 & 1.803840e-02 & 4.842570e-03 & 1.910870e-02\\
 2 + 3 & 1359 & 3.158860e-03 & 1.336130e-02 & 3.610560e-03 & 1.592600e-02\\
 4 + 3 & 2046 & 4.927190e-04 & 5.223880e-03 & 5.724240e-04 & 5.706150e-03\\
 8 + 3 & 3537 & 1.383780e-04 & 3.568890e-03 & 1.477890e-04 & 3.777700e-03\\
 1 + 4 & 1095 & 4.330130e-03 & 1.550790e-02 & 4.237330e-03 & 1.609030e-02\\
\hline
\end{tabular}
\end{center}
\end{table}

We present the errors of the online adaptive multiscale method with $\theta=0.8$ in Table \ref{tab:online_errors_p_case_1_theta_0.8} for the Case 1 of perforated media. In this situation, $N_{\tu{it}}$ is not equal to the number of online basis functions per coarse-grid node, and $\tu{DOF}_{\tu{H}}$ corresponds to the maximum number of degrees of freedom since they can be different at each online basis update. One can see that the errors of the adaptive method are at about the same level as those of the uniform method. For example, when $N_{\tu{b}} = 4$ with $N_{\tu{it}} = 3$ at $\delta = 0.1$, we obtain $L^2$ errors less than 0.1\% and $H^1$ errors less than 1\% for both online methods. The uniform method has 2541 degrees of freedom, and the adaptive method has only 2046 degrees of freedom. Thus, by solving a discrete problem of smaller size, one can reach comparable accuracy.


\begin{table}[!htbp]
\caption{Relative errors of the offline method for the case 2 of the perforated media.}
\label{tab:offline_errors_p_case_2}
\begin{center}
\begin{tabular}{ | c | c | c c | c c |}
\hline
\multirowcell{2}{$N_{\tu{b}}$}
& \multirowcell{2}{$\tu{DOF}_{\tu{H}}$}
& \multicolumn{2}{c|}{Displacement}
& \multicolumn{2}{c|}{Microrotation}
\\
\cline{3-6}
&
& $e^{u}_{L^2}$
& $e^{u}_{H^1}$
& $e^{\Phi}_{L^2}$ 
& $e^{\Phi}_{H^1}$ 
\\
\hline
 1   & 363 & 1.370340e-02 & 1.149450e-01 & 1.699130e-02 & 1.298840e-01\\
 2   & 726 & 1.136590e-02 & 1.117520e-01 & 1.310930e-02 & 1.266990e-01\\
 4   & 1452 & 2.085370e-03 & 2.609290e-02 & 2.369690e-03 & 2.847360e-02\\
 8   & 2904 & 8.463220e-04 & 1.798340e-02 & 1.026200e-03 & 2.083840e-02\\
\hline
\end{tabular}
\end{center}
\end{table}

Let us briefly consider Case 2, which corresponds to a big strain-limiting parameter. Table \ref{tab:offline_errors_p_case_2} presents the errors for the offline multiscale method using different numbers of basis functions per coarse-grid node. We see that the errors are mainly comparable to Case 1 (Table \ref{tab:offline_errors_p_case_1}). One can also observe convergence in terms of the number of basis functions. Generally, the errors are minor.

\begin{table}[!htbp]
\caption{Relative errors of the online method with $\theta = 1$ for the case 2 of the perforated media.}
\label{tab:online_errors_p_case_2_theta_1}
\begin{center}
\begin{tabular}{ | c | c | c c | c c | }
\hline
\multirowcell{2}{$N_{\tu{b}} + N_{\tu{it}}$}
& \multirowcell{2}{$\tu{DOF}_{\tu{H}}$}
& \multicolumn{2}{c|}{Displacement}
& \multicolumn{2}{c|}{Microrotation}
\\
\cline{3-6}
&
& $e^{u}_{L^2}$
& $e^{u}_{H^1}$
& $e^{\Phi}_{L^2}$ 
& $e^{\Phi}_{H^1}$ 
\\
\hline
\multicolumn{6}{|c|}{$\delta = \infty$}
\\
\hline
 1 + 3 & 1452 & 2.435850e-03 & 2.170450e-02 & 2.476330e-03 & 1.967200e-02\\
 2 + 3 & 1815 & 2.033500e-03 & 2.116150e-02 & 1.936440e-03 & 1.942930e-02\\
 4 + 3 & 2541 & 1.110540e-03 & 1.629070e-02 & 1.164350e-03 & 1.549650e-02\\
 8 + 3 & 3993 & 7.451800e-04 & 1.318320e-02 & 7.644670e-04 & 1.254380e-02\\
 1 + 4 & 1815 & 2.196970e-03 & 2.133130e-02 & 2.058320e-03 & 1.918770e-02\\
\hline
\multicolumn{6}{|c|}{$\delta = 0.4$}
\\
\hline
 1 + 3 & 1452 & 4.987360e-03 & 2.872880e-02 & 6.462580e-03 & 3.388980e-02\\
 2 + 3 & 1815 & 4.167940e-03 & 2.595640e-02 & 5.253880e-03 & 2.809690e-02\\
 4 + 3 & 2541 & 3.906550e-04 & 6.585540e-03 & 4.810260e-04 & 7.630010e-03\\
 8 + 3 & 3993 & 1.791280e-04 & 4.494120e-03 & 2.201490e-04 & 5.358520e-03\\
 1 + 4 & 1815 & 3.068900e-03 & 2.296110e-02 & 3.740420e-03 & 2.849240e-02\\
\hline
\multicolumn{6}{|c|}{$\delta = 0.2$}
\\
\hline
 1 + 3 & 1452 & 1.985970e-03 & 1.132960e-02 & 2.201000e-03 & 1.053690e-02\\
 2 + 3 & 1815 & 1.344490e-03 & 1.090480e-02 & 1.398930e-03 & 1.002070e-02\\
 4 + 3 & 2541 & 7.758180e-04 & 8.888690e-03 & 9.287580e-04 & 8.832760e-03\\
 8 + 3 & 3993 & 3.431700e-04 & 6.869240e-03 & 3.341570e-04 & 6.398860e-03\\
 1 + 4 & 1815 & 1.540020e-03 & 1.097400e-02 & 1.695390e-03 & 9.976230e-03\\
\hline
\multicolumn{6}{|c|}{$\delta = 0.1$}
\\
\hline
 1 + 3 & 1452 & 3.814970e-03 & 1.217410e-02 & 3.875700e-03 & 1.189900e-02\\
 2 + 3 & 1815 & 2.827150e-03 & 1.069590e-02 & 3.370360e-03 & 1.213590e-02\\
 4 + 3 & 2541 & 2.864940e-04 & 3.782950e-03 & 3.362190e-04 & 4.274660e-03\\
 8 + 3 & 3993 & 1.444410e-04 & 3.211770e-03 & 1.825890e-04 & 3.529940e-03\\
 1 + 4 & 1815 & 2.400610e-03 & 1.012620e-02 & 1.995780e-03 & 9.051530e-03\\
\hline
\end{tabular}
\end{center}
\end{table}

Table \ref{tab:online_errors_p_case_2_theta_1} shows the errors of the online uniform multiscale method with $\theta = 1$ for the Case 2 of perforated media. We also consider different values of $\delta$. One can realize that the accuracy of the online uniform method for Case 2 is about the same as for Case 1. Regarding $\delta=\infty$, we again see that the solution using $N_{\tu{b}} = 2$ with $N_{\tu{it}} = 3$ gives smaller errors than the one using $N_{\tu{b}} = 1$ with $N_{\tu{it}} = 4\,.$ However, for the remaining cases of $\delta$, this advantage is not noticeable. Moreover, the multiscale solution using $N_{\tu{b}} = 1$ and $N_{\tu{it}} = 4$ often gives better accuracy. One can explain this by the larger strain-limiting parameter of Case 2.

\begin{table}[!htbp]
\caption{Relative errors of the online method with $\theta = 0.8$ for the case 2 of the perforated media.}
\label{tab:online_errors_p_case_2_theta_0.8}
\begin{center}
\begin{tabular}{ | c | c | c c | c c | }
\hline
\multirowcell{2}{$N_{\tu{b}} + N_{\tu{it}}$}
& \multirowcell{2}{$\tu{DOF}_{\tu{H}}$}
& \multicolumn{2}{c|}{Displacement}
& \multicolumn{2}{c|}{Microrotation}
\\
\cline{3-6}
&
& $e^{u}_{L^2}$
& $e^{u}_{H^1}$
& $e^{\Phi}_{L^2}$ 
& $e^{\Phi}_{H^1}$ 
\\
\hline
\multicolumn{6}{|c|}{$\delta = \infty$}
\\
\hline
 1 + 3 & 909 & 2.891710e-03 & 2.413830e-02 & 3.136580e-03 & 2.241100e-02\\
 2 + 3 & 1302 & 2.555780e-03 & 2.283970e-02 & 2.638580e-03 & 2.143710e-02\\
 4 + 3 & 2064 & 1.458480e-03 & 1.712380e-02 & 1.483370e-03 & 1.638800e-02\\
 8 + 3 & 3558 & 8.242540e-04 & 1.378050e-02 & 8.203950e-04 & 1.310640e-02\\
 1 + 4 & 1071 & 2.844800e-03 & 2.255290e-02 & 3.143100e-03 & 2.070790e-02\\
\hline
\multicolumn{6}{|c|}{$\delta = 0.4$}
\\
\hline
 1 + 3 & 912 & 9.028990e-03 & 4.054170e-02 & 1.295380e-02 & 4.849750e-02\\
 2 + 3 & 1332 & 7.354100e-03 & 3.505800e-02 & 1.054620e-02 & 4.153830e-02\\
 4 + 3 & 2064 & 6.639720e-04 & 8.346750e-03 & 8.410670e-04 & 9.949140e-03\\
 8 + 3 & 3558 & 2.508880e-04 & 5.235840e-03 & 2.820890e-04 & 6.044700e-03\\
 1 + 4 & 1077 & 7.995480e-03 & 3.853170e-02 & 1.276200e-02 & 4.740510e-02\\
\hline
\multicolumn{6}{|c|}{$\delta = 0.2$}
\\
\hline
 1 + 3 & 912 & 2.495600e-03 & 1.467600e-02 & 2.914860e-03 & 1.402710e-02\\
 2 + 3 & 1332 & 2.120530e-03 & 1.426070e-02 & 2.326210e-03 & 1.444220e-02\\
 4 + 3 & 2070 & 1.338720e-03 & 1.001800e-02 & 1.515790e-03 & 1.018390e-02\\
 8 + 3 & 3558 & 3.960510e-04 & 7.142770e-03 & 3.882370e-04 & 6.863680e-03\\
 1 + 4 & 1077 & 2.513640e-03 & 1.257770e-02 & 2.991310e-03 & 1.210970e-02\\
\hline
\multicolumn{6}{|c|}{$\delta = 0.1$}
\\
\hline
 1 + 3 & 912 & 7.231120e-03 & 1.997160e-02 & 7.389370e-03 & 2.010930e-02\\
 2 + 3 & 1347 & 5.104610e-03 & 1.550660e-02 & 5.186220e-03 & 1.737100e-02\\
 4 + 3 & 2070 & 5.897670e-04 & 5.283380e-03 & 6.642410e-04 & 5.975320e-03\\
 8 + 3 & 3558 & 2.182390e-04 & 3.586320e-03 & 2.569860e-04 & 3.958370e-03\\
 1 + 4 & 1077 & 6.709820e-03 & 1.779880e-02 & 6.864230e-03 & 1.812390e-02\\
\hline
\end{tabular}
\end{center}
\end{table}

Finally, Table \ref{tab:online_errors_p_case_2_theta_0.8} demonstrates the errors of the online adaptive multiscale method with $\theta = 0.8$ for the Case 2 of perforated media. As in the previous case, one can see that the accuracy of the adaptive method is comparable to the uniform one. For example, using $N_{\tu{b}} = 2$ with $N_{\tu{it}} = 3$ at $\delta=0.4$, the uniform method gives $L^2$ error less than 0.1\% and $H^1$ error less than 1\%. The adaptive method has similar accuracy to the uniform one. At the same time, the uniform method has 2541 degrees of freedom, and the adaptive method has only 2064 degrees of freedom. Thus, our proposed multiscale approaches allow us to consider perforated media with small and large strain-limiting parameters.

\subsection{Composite media}\label{combom}


\begin{figure}[!htbp]
    \centering
    \begin{subfigure}[b]{\textwidth}
        \centering
        \includegraphics[width=0.7\textwidth]{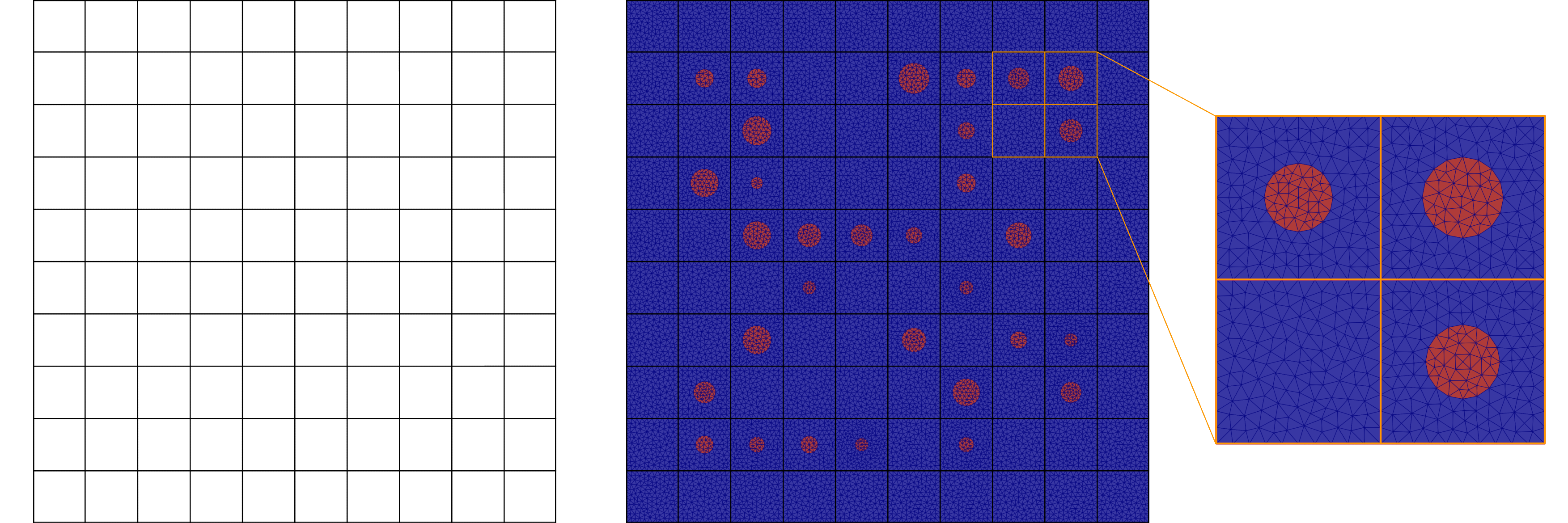}
        \caption{Case 1 (small strain-limiting parameter).}
        \label{fig:c_mesh_1}
    \end{subfigure}
    \begin{subfigure}[b]{\textwidth}
        \centering
        \includegraphics[width=0.7\textwidth]{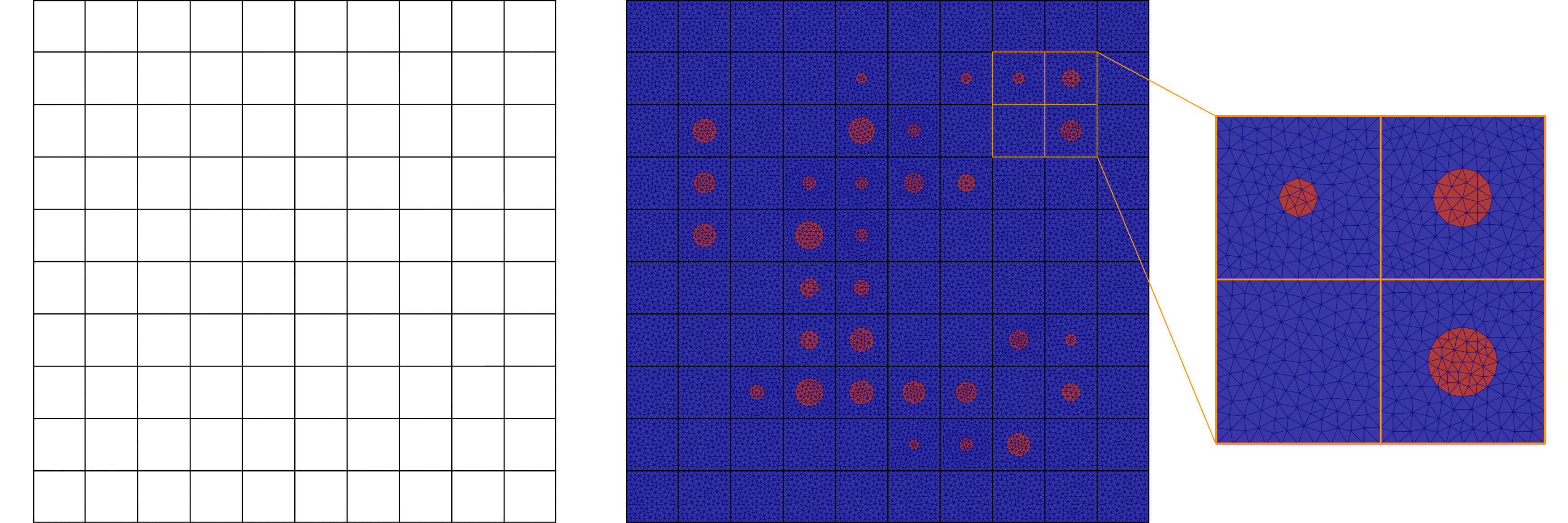}
        \caption{Case 2 (big strain-limiting parameter).}
        \label{fig:c_mesh_2}
    \end{subfigure}
    \caption{Computational grids for the composite media: left -- the coarse grid (each small square is a coarse block), middle -- the fine grid ($\Omega_1$ is blue, $\Omega_2$ is red), right -- a local domain.}
    \label{fig:c_meshes}
\end{figure}

Now, we investigate composite nonlinear Cosserat media with small and big strain-limiting parameters. The model problems are defined in the computational domain $\Omega = \Omega_1 \cup \Omega_2 = [0, 2]\times [0, 2]$ (see Figure \ref{fig:c_meshes}). Here, $\Omega_2$ is the domain of inclusions, and $\Omega_1$ is the rest. As in the previous Section \ref{perform}, we use a uniform coarse grid with 100 rectangular cells and 121 vertices. The following unstructured fine grids are applied.

\begin{itemize}
    \item Case 1: 26 922 triangular fine cells and 13 662 vertices.
    \item Case 2: 26 940 triangular fine cells and 13 671 vertices.
\end{itemize}

In this type of media, the inhomogeneity is caused by two domains whose properties can be very different. We suppose that the materials of the composite are rigidly connected, that is, the interface conditions follow from differential equations by continuity. In our model problems, {the following parameters are chosen}.

\begin{itemize}
    \item Case 1: $f_v = 0.3$, $g_v = 0.31$, and
    \begin{equation*}
        \xi = 
        \begin{cases}
            1, & x \in \Omega_1,\\
            1000, & x \in \Omega_2
        \end{cases}, \quad
        \alpha = 
        \begin{cases}
            1.1, & x \in \Omega_1,\\
            2000, & x \in \Omega_2
        \end{cases}, \quad
        \beta = 
        \begin{cases}
            1, & x \in \Omega_1,\\
            10^{-4}, & x \in \Omega_2
        \end{cases}.
    \end{equation*}
    \item Case 2: $f_v = 3.6 \cdot 10^{-5}$, $g_v = 3.45 \cdot 10^{-5}$, and
    \begin{equation*}
        \xi = 
        \begin{cases}
            1.2, & x \in \Omega_1,\\
            3000, & x \in \Omega_2
        \end{cases}, \quad
        \alpha = 
        \begin{cases}
            1.3, & x \in \Omega_1,\\
            4000, & x \in \Omega_2
        \end{cases}, \quad
        \beta = 
        \begin{cases}
            10^{4}, & x \in \Omega_1,\\
            1, & x \in \Omega_2
        \end{cases}.
    \end{equation*}
\end{itemize}


\begin{figure}[!htbp]
    \centering
    \begin{subfigure}[b]{\textwidth}
        \centering
        \includegraphics[width=\textwidth]{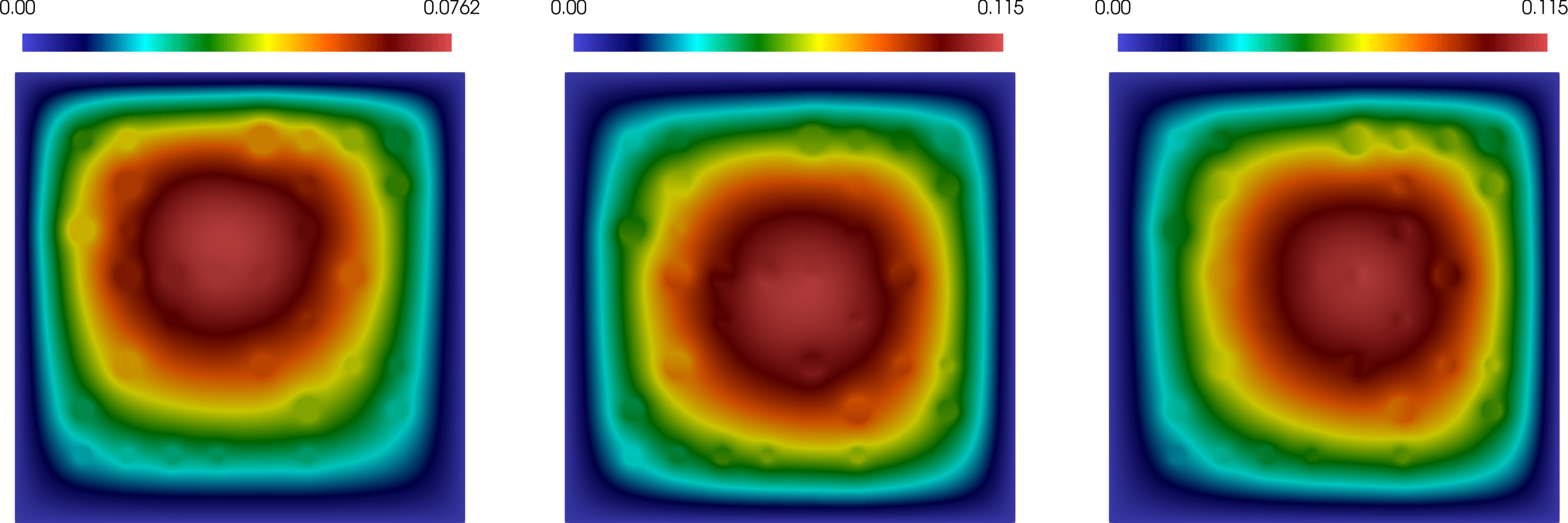}
        \caption{Reference solution.}
        \label{fig:c_results_fine}
    \end{subfigure}
    \begin{subfigure}[b]{\textwidth}
        \centering
        \includegraphics[width=\textwidth]{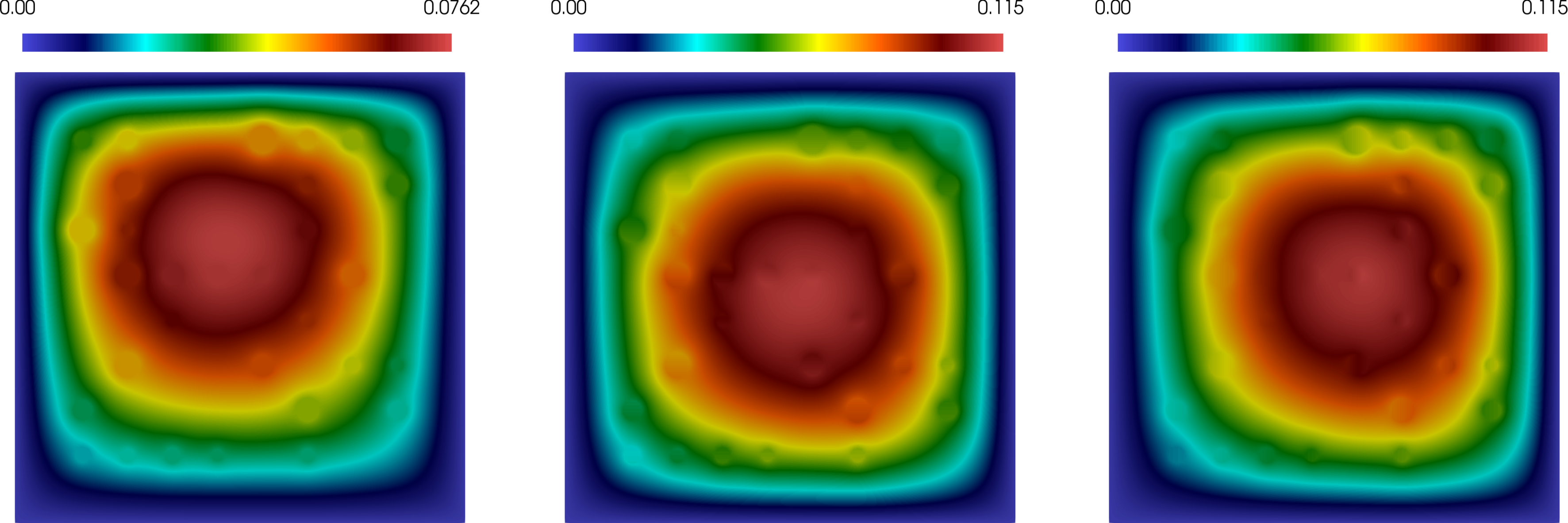}
        \caption{Online adaptive multiscale solution with $N_{\tu{b}} = 4$, $N_{\tu{it}} = 3$, $\theta = 0.8$, and $\delta = 0.1$\,.}
        \label{fig:c_results_online}
    \end{subfigure}
    \begin{subfigure}[b]{\textwidth}
        \centering
        \includegraphics[width=\textwidth]{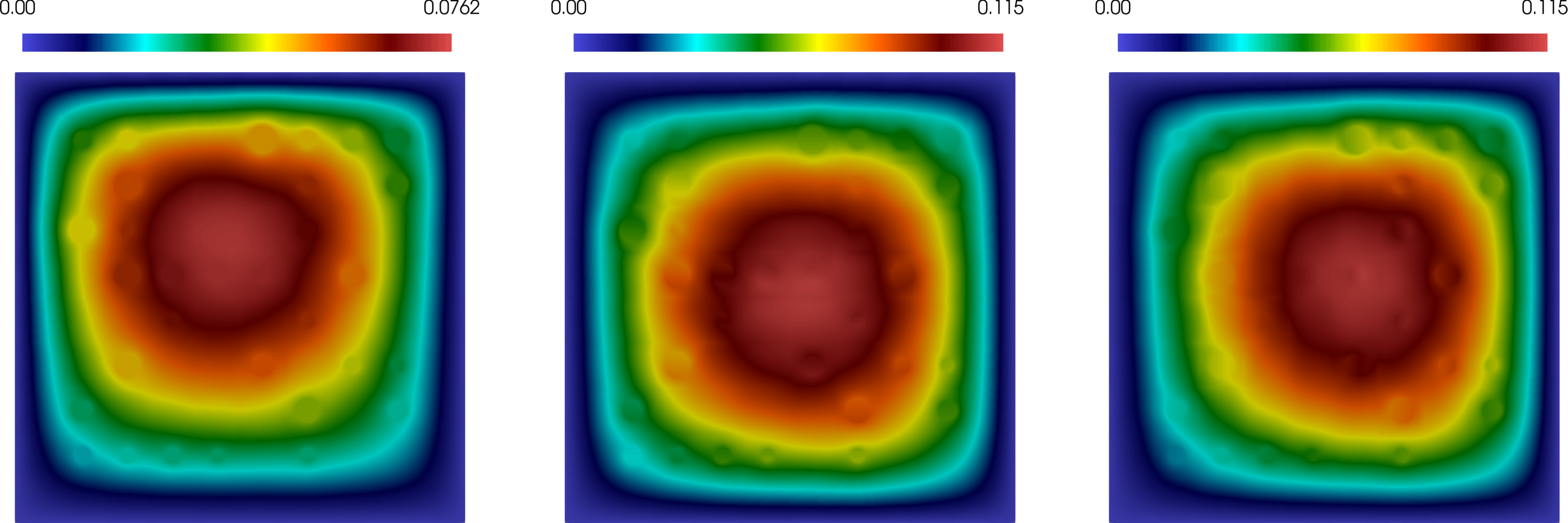}
        \caption{Offline multiscale solution with $N_{\tu{b}} = 4$\,.}
        \label{fig:c_results_offline}
    \end{subfigure}
    \caption{Distributions of the microrotations $\Phi$ and displacements $u_1$ and $u_2$ (from left to right) for Case 1 (small strain-limiting parameter) of the composite media.}
    \label{fig:c_results}
\end{figure}

Figure \ref{fig:c_results} presents the numerical solution of the model problem for Case 1. From left to right, we depict the microrotations $\Phi$ and displacements $u_1$ and $u_2$. In Figure \ref{fig:c_results_fine}, the reference solution is shown on a fine grid. Figure \ref{fig:c_results_online} presents the online adaptive multiscale solution ($\theta = 0.8$) using four offline and three online basis functions with parameter $\delta = 0.1$. Figure \ref{fig:c_results_offline} shows the offline multiscale solution with four offline basis functions. All figures are very similar, indicating the high accuracy of the multiscale approaches. However, one can recognize certain differences between the offline solution and the reference solution.

It is also clear from Figure \ref{fig:c_results} that the distributions of the solution fields noticeably reflect the medium's inhomogeneity, which is expressed by the presence of inclusions. In these inclusions, the strain-limiting parameter $\beta$ is smaller, but the material properties $\xi$ and $\alpha$ are stiffer than in the rest domain. As in the previous media type, the microrotations and displacements are fixed on the boundaries. The highest field magnitudes are in the central region. In general, the numerical solutions correspond to the modeled process.



\begin{table}[!htbp]
\caption{Relative errors of the offline method for the case 1 of the composite media.}
\label{tab:offline_errors_c_case_1}
\begin{center}
\begin{tabular}{ | c | c | c c | c c |}
\hline
\multirowcell{2}{$N_{\tu{b}}$}
& \multirowcell{2}{$\tu{DOF}_{\tu{H}}$}
& \multicolumn{2}{c|}{Displacement}
& \multicolumn{2}{c|}{Microrotation}
\\
\cline{3-6}
&
& $e^{u}_{L^2}$
& $e^{u}_{H^1}$
& $e^{\Phi}_{L^2}$ 
& $e^{\Phi}_{H^1}$ 
\\
\hline
 1   & 363 & 8.821480e-02 & 3.189490e-01 & 2.988490e-01 & 3.715490e-01\\
 2   & 726 & 3.917690e-02 & 1.324200e-01 & 1.222990e-01 & 2.208070e-01\\
 4   & 1452 & 6.962580e-03 & 2.019020e-02 & 1.850640e-02 & 5.644450e-02\\
 8   & 2904 & 1.988760e-03 & 4.348660e-03 & 3.817930e-03 & 2.722010e-02\\
\hline
\end{tabular}
\end{center}
\end{table}

Let us consider the errors of the multiscale solutions for the Case 1. In Table \ref{tab:offline_errors_c_case_1}, we present the errors of the offline multiscale method. One can see that the errors are minor. For example, using eight offline basis functions per coarse-grid node, all errors (except the $H^1$ error for the microrotation) are less than 1\%. As for the previous media type, the $H^1$ errors are significantly bigger than the $L^2$ ones because they reflect errors in computing the gradients of the solution. We also observe the convergence based on the number of basis functions.

\begin{table}[!htbp]
\caption{Relative errors of the online method with $\theta = 1$ for the case 1 of the composite media.}
\label{tab:online_errors_c_case_1_theta_1}
\begin{center}
\begin{tabular}{ | c | c | c c | c c | }
\hline
\multirowcell{2}{$N_{\tu{b}} + N_{\tu{it}}$}
& \multirowcell{2}{$\tu{DOF}_{\tu{H}}$}
& \multicolumn{2}{c|}{Displacement}
& \multicolumn{2}{c|}{Microrotation}
\\
\cline{3-6}
&
& $e^{u}_{L^2}$
& $e^{u}_{H^1}$
& $e^{\Phi}_{L^2}$ 
& $e^{\Phi}_{H^1}$ 
\\
\hline
\multicolumn{6}{|c|}{$\delta = \infty$}
\\
\hline
 1 + 3 & 1452 & 1.540550e-03 & 1.432950e-03 & 1.445170e-03 & 1.408890e-02\\
 2 + 3 & 1815 & 1.322930e-03 & 1.102990e-03 & 1.238920e-03 & 1.324480e-02\\
 4 + 3 & 2541 & 8.390830e-04 & 5.006020e-04 & 7.973220e-04 & 1.126740e-02\\
 8 + 3 & 3993 & 4.287550e-04 & 3.243690e-04 & 4.135700e-04 & 9.052670e-03\\
 1 + 4 & 1815 & 1.200350e-03 & 1.185380e-03 & 1.187120e-03 & 1.318330e-02\\
\hline
\multicolumn{6}{|c|}{$\delta = 0.4$}
\\
\hline
 1 + 3 & 1452 & 1.916410e-03 & 2.617040e-03 & 2.660490e-03 & 2.630780e-02\\
 2 + 3 & 1815 & 3.753340e-03 & 2.950720e-03 & 3.638780e-03 & 2.857920e-02\\
 4 + 3 & 2541 & 4.856150e-04 & 4.393730e-04 & 4.785500e-04 & 1.025310e-02\\
 8 + 3 & 3993 & 2.579940e-04 & 3.511540e-04 & 2.560860e-04 & 7.048410e-03\\
 1 + 4 & 1815 & 1.434210e-03 & 1.855850e-03 & 1.794810e-03 & 2.088730e-02\\
\hline
\multicolumn{6}{|c|}{$\delta = 0.2$}
\\
\hline
 1 + 3 & 1452 & 1.481960e-03 & 1.672200e-03 & 2.059940e-03 & 9.501830e-03\\
 2 + 3 & 1815 & 1.259160e-03 & 1.099940e-03 & 1.344230e-03 & 8.796490e-03\\
 4 + 3 & 2541 & 7.458470e-04 & 6.358040e-04 & 1.004590e-03 & 8.353710e-03\\
 8 + 3 & 3993 & 2.415230e-04 & 2.209090e-04 & 3.042640e-04 & 6.169510e-03\\
 1 + 4 & 1815 & 1.113840e-03 & 1.269720e-03 & 1.488810e-03 & 8.402260e-03\\
\hline
\multicolumn{6}{|c|}{$\delta = 0.1$}
\\
\hline
 1 + 3 & 1452 & 2.002130e-03 & 1.399120e-03 & 1.522510e-03 & 1.001270e-02\\
 2 + 3 & 1815 & 2.899690e-03 & 2.111450e-03 & 2.655210e-03 & 1.358140e-02\\
 4 + 3 & 2541 & 3.208680e-04 & 1.236260e-04 & 2.478960e-04 & 4.955970e-03\\
 8 + 3 & 3993 & 1.432610e-04 & 1.283130e-04 & 1.360620e-04 & 4.441870e-03\\
 1 + 4 & 1815 & 1.228000e-03 & 8.582460e-04 & 9.780720e-04 & 7.906290e-03\\
\hline
\end{tabular}
\end{center}
\end{table}

Table \ref{tab:online_errors_c_case_1_theta_1} presents the errors of the online uniform multiscale method ($\theta=1$) for the Case 1 of composite media. One can see that using online basis functions can significantly improve the solution's accuracy. For example, already, with the use of one offline basis function and three online basis functions at $\delta=0.2$, we can achieve $L^2$ and $H^1$ errors less than 1\%. We also observe convergence in relation to the number of offline basis functions. The accuracy also improves on average as $\delta$ decreases. However, it is worth noting that we no longer attain better accuracy for the solution with $N_{\tu{b}} = 2$ and $N_{\tu{it}} = 3$ when comparing to $N_{\tu{b}} = 1$ and $N_{\tu{it}} = 4$. This can be explained by the big strain-limiting parameter and heterogeneity of the problem. It is a conjecture that in the case of computing offline basis functions at each update of the online basis functions, we could see this advantage. However, such a procedure is too costly.

\begin{table}[!htbp]
\caption{Relative errors of the online method with $\theta = 0.8$ for the case 1 of the composite media.}
\label{tab:online_errors_c_case_1_theta_0.8}
\begin{center}
\begin{tabular}{ | c | c | c c | c c | }
\hline
\multirowcell{2}{$N_{\tu{b}} + N_{\tu{it}}$}
& \multirowcell{2}{$\tu{DOF}_{\tu{H}}$}
& \multicolumn{2}{c|}{Displacement}
& \multicolumn{2}{c|}{Microrotation}
\\
\cline{3-6}
&
& $e^{u}_{L^2}$
& $e^{u}_{H^1}$
& $e^{\Phi}_{L^2}$ 
& $e^{\Phi}_{H^1}$ 
\\
\hline
\multicolumn{6}{|c|}{$\delta = \infty$}
\\
\hline
 1 + 3 & 792 & 3.504310e-03 & 4.055800e-03 & 3.587720e-03 & 3.043990e-02\\
 2 + 3 & 1245 & 2.383030e-03 & 2.161330e-03 & 2.288430e-03 & 1.983400e-02\\
 4 + 3 & 2007 & 1.269100e-03 & 8.455580e-04 & 1.223640e-03 & 1.281290e-02\\
 8 + 3 & 3498 & 4.782010e-04 & 3.702360e-04 & 5.093350e-04 & 9.745110e-03\\
 1 + 4 & 975 & 2.998860e-03 & 2.383980e-03 & 2.839910e-03 & 1.950780e-02\\
\hline
\multicolumn{6}{|c|}{$\delta = 0.4$}
\\
\hline
 1 + 3 & 792 & 1.179040e-02 & 1.398400e-02 & 1.279110e-02 & 5.886700e-02\\
 2 + 3 & 1245 & 8.768270e-03 & 9.720320e-03 & 9.785990e-03 & 4.981840e-02\\
 4 + 3 & 2067 & 7.916820e-04 & 3.607100e-04 & 7.917660e-04 & 1.206180e-02\\
 8 + 3 & 3498 & 3.338000e-04 & 3.095550e-04 & 3.305040e-04 & 7.970240e-03\\
 1 + 4 & 975 & 9.338440e-03 & 1.312060e-02 & 1.164840e-02 & 5.417330e-02\\
\hline
\multicolumn{6}{|c|}{$\delta = 0.2$}
\\
\hline
 1 + 3 & 792 & 3.500540e-03 & 3.843040e-03 & 3.569960e-03 & 2.580140e-02\\
 2 + 3 & 1245 & 2.697600e-03 & 2.555530e-03 & 2.638660e-03 & 1.661120e-02\\
 4 + 3 & 2067 & 1.476230e-03 & 1.250290e-03 & 1.706940e-03 & 1.018510e-02\\
 8 + 3 & 3507 & 3.100230e-04 & 2.756500e-04 & 4.049160e-04 & 6.613470e-03\\
 1 + 4 & 975 & 3.278940e-03 & 2.758100e-03 & 2.986990e-03 & 1.419990e-02\\
\hline
\multicolumn{6}{|c|}{$\delta = 0.1$}
\\
\hline
 1 + 3 & 792 & 7.240430e-03 & 7.948310e-03 & 7.767290e-03 & 3.600690e-02\\
 2 + 3 & 1245 & 5.531660e-03 & 5.699530e-03 & 5.785120e-03 & 2.618690e-02\\
 4 + 3 & 2067 & 6.175090e-04 & 3.395710e-04 & 6.051020e-04 & 6.697970e-03\\
 8 + 3 & 3507 & 2.646880e-04 & 1.324650e-04 & 2.354530e-04 & 5.056960e-03\\
 1 + 4 & 975 & 6.899050e-03 & 7.360830e-03 & 7.277210e-03 & 2.612820e-02\\
\hline
\end{tabular}
\end{center}
\end{table}

Table \ref{tab:online_errors_c_case_1_theta_0.8} presents the errors of the online adaptive multiscale method ($\theta = 0.8$) for the Case 1 of composite media. As before, the uniform and adaptive multiscale methods have comparable errors. For example, using two offline and three online basis functions with $\delta=0.4$, the uniform method yields errors of less than 1\% (except the $H^1$ error for microrotation). The online adaptive multiscale method has similar accuracy to the uniform one. However, the uniform method has 1815 degrees of freedom, while the adaptive method has 1245, which is fewer than the uniform one. 


\begin{table}[!htbp]
\caption{Relative errors of the offline method for the case 2 of the composite media.}
\label{tab:offline_errors_c_case_2}
\begin{center}
\begin{tabular}{ | c | c | c c | c c |}
\hline
\multirowcell{2}{$N_{\tu{b}}$}
& \multirowcell{2}{$\tu{DOF}_{\tu{H}}$}
& \multicolumn{2}{c|}{Displacement}
& \multicolumn{2}{c|}{Microrotation}
\\
\cline{3-6}
&
& $e^{u}_{L^2}$
& $e^{u}_{H^1}$
& $e^{\Phi}_{L^2}$ 
& $e^{\Phi}_{H^1}$ 
\\
\hline
 1   & 363 & 1.640680e-01 & 5.914740e-01 & 5.319100e-01 & 6.056560e-01\\
 2   & 726 & 6.918150e-02 & 2.251570e-01 & 1.978250e-01 & 2.936410e-01\\
 4   & 1452 & 7.349560e-03 & 2.117800e-02 & 1.794980e-02 & 5.218640e-02\\
 8   & 2904 & 1.842920e-03 & 4.914930e-03 & 3.976310e-03 & 2.749620e-02\\
\hline
\end{tabular}
\end{center}
\end{table}

Let us consider the Case 2 of composite Cosserat media. 
Table \ref{tab:offline_errors_c_case_2} presents the errors of the offline multiscale method. One can notice that the errors are comparable to Case 1 (Table \ref{tab:offline_errors_c_case_1}), although slightly larger than Case 1. We also see that the errors decrease as the number of offline basis functions per coarse-grid node increases. Thus, one obtains convergence in regards to the number of basis functions.

\begin{table}[!htbp]
\caption{Relative errors of the online method with $\theta = 1$ for the case 2 of the composite media.}
\label{tab:online_errors_c_case_2_theta_1}
\begin{center}
\begin{tabular}{ | c | c | c c | c c | }
\hline
\multirowcell{2}{$N_{\tu{b}} + N_{\tu{it}}$}
& \multirowcell{2}{$\tu{DOF}_{\tu{H}}$}
& \multicolumn{2}{c|}{Displacement}
& \multicolumn{2}{c|}{Microrotation}
\\
\cline{3-6}
&
& $e^{u}_{L^2}$
& $e^{u}_{H^1}$
& $e^{\Phi}_{L^2}$ 
& $e^{\Phi}_{H^1}$ 
\\
\hline
\multicolumn{6}{|c|}{$\delta = \infty$}
\\
\hline
 1 + 3 & 1452 & 1.564350e-03 & 1.233900e-03 & 1.491590e-03 & 1.245180e-02\\
 2 + 3 & 1815 & 1.252380e-03 & 1.132610e-03 & 1.353440e-03 & 1.248730e-02\\
 4 + 3 & 2541 & 7.840430e-04 & 4.222470e-04 & 7.640060e-04 & 1.017490e-02\\
 8 + 3 & 3993 & 3.862730e-04 & 2.408230e-04 & 3.889250e-04 & 8.622800e-03\\
 1 + 4 & 1815 & 1.180850e-03 & 9.115450e-04 & 1.060520e-03 & 1.132220e-02\\
\hline
\multicolumn{6}{|c|}{$\delta = 0.4$}
\\
\hline
 1 + 3 & 1452 & 2.678970e-03 & 2.759920e-03 & 3.309940e-03 & 2.192180e-02\\
 2 + 3 & 1815 & 4.134810e-03 & 2.921550e-03 & 4.138980e-03 & 2.276860e-02\\
 4 + 3 & 2541 & 4.919840e-04 & 2.736570e-04 & 4.189340e-04 & 8.417580e-03\\
 8 + 3 & 3993 & 2.176740e-04 & 2.722260e-04 & 2.199610e-04 & 5.964190e-03\\
 1 + 4 & 1815 & 1.530630e-03 & 1.257940e-03 & 1.334140e-03 & 1.430190e-02\\
\hline
\multicolumn{6}{|c|}{$\delta = 0.2$}
\\
\hline
 1 + 3 & 1452 & 1.504100e-03 & 1.612560e-03 & 2.292810e-03 & 8.535800e-03\\
 2 + 3 & 1815 & 1.280870e-03 & 1.039480e-03 & 1.510380e-03 & 8.601380e-03\\
 4 + 3 & 2541 & 9.026200e-04 & 6.005110e-04 & 1.046730e-03 & 7.874670e-03\\
 8 + 3 & 3993 & 1.825100e-04 & 1.500560e-04 & 2.494650e-04 & 5.350810e-03\\
 1 + 4 & 1815 & 1.100780e-03 & 1.004660e-03 & 1.480570e-03 & 6.667670e-03\\
\hline
\multicolumn{6}{|c|}{$\delta = 0.1$}
\\
\hline
 1 + 3 & 1452 & 2.328760e-03 & 1.247360e-03 & 1.736490e-03 & 8.320120e-03\\
 2 + 3 & 1815 & 3.608740e-03 & 1.819820e-03 & 2.835160e-03 & 1.142450e-02\\
 4 + 3 & 2541 & 3.385930e-04 & 1.101790e-04 & 2.455330e-04 & 4.463650e-03\\
 8 + 3 & 3993 & 1.387230e-04 & 9.797640e-05 & 1.181420e-04 & 4.153570e-03\\
 1 + 4 & 1815 & 1.321970e-03 & 7.348080e-04 & 1.023170e-03 & 6.212770e-03\\
\hline
\end{tabular}
\end{center}
\end{table}

Table \ref{tab:online_errors_c_case_2_theta_1} presents the errors of the online uniform multiscale method ($\theta=1$) for the Case 2 of composite media. One can see that the use of online basis functions allows for facilitating convergence significantly. For example, adding three online basis functions to one offline basis function per coarse-grid node at $\delta=0.2$ gives errors less than 1\%. We observe convergence relating to the number of the offline basis functions. One can also realize that on average, the accuracy improves as $\delta$ decreases.

\begin{table}[!htbp]
\caption{Relative errors of the online method with $\theta = 0.8$ for the case 2 of the composite media.}
\label{tab:online_errors_c_case_2_theta_0.8}
\begin{center}
\begin{tabular}{ | c | c | c c | c c | }
\hline
\multirowcell{2}{$N_{\tu{b}} + N_{\tu{it}}$}
& \multirowcell{2}{$\tu{DOF}_{\tu{H}}$}
& \multicolumn{2}{c|}{Displacement}
& \multicolumn{2}{c|}{Microrotation}
\\
\cline{3-6}
&
& $e^{u}_{L^2}$
& $e^{u}_{H^1}$
& $e^{\Phi}_{L^2}$ 
& $e^{\Phi}_{H^1}$ 
\\
\hline
\multicolumn{6}{|c|}{$\delta = \infty$}
\\
\hline
 1 + 3 & 831 & 3.086740e-03 & 2.572870e-03 & 3.189070e-03 & 2.674780e-02\\
 2 + 3 & 1242 & 2.266410e-03 & 1.430050e-03 & 2.205310e-03 & 2.034500e-02\\
 4 + 3 & 2004 & 1.277540e-03 & 6.800980e-04 & 1.212320e-03 & 1.195470e-02\\
 8 + 3 & 3504 & 4.411360e-04 & 3.006820e-04 & 5.189180e-04 & 9.357110e-03\\
 1 + 4 & 1032 & 2.934250e-03 & 2.023030e-03 & 3.025310e-03 & 1.674490e-02\\
\hline
\multicolumn{6}{|c|}{$\delta = 0.4$}
\\
\hline
 1 + 3 & 831 & 1.412200e-02 & 1.273810e-02 & 1.468340e-02 & 5.247610e-02\\
 2 + 3 & 1242 & 9.504390e-03 & 1.013690e-02 & 1.234710e-02 & 4.625370e-02\\
 4 + 3 & 2034 & 8.137250e-04 & 2.965500e-04 & 8.708100e-04 & 1.030060e-02\\
 8 + 3 & 3504 & 3.300950e-04 & 1.655290e-04 & 3.523260e-04 & 6.878870e-03\\
 1 + 4 & 1032 & 1.125880e-02 & 9.721680e-03 & 1.148180e-02 & 4.241010e-02\\
\hline
\multicolumn{6}{|c|}{$\delta = 0.2$}
\\
\hline
 1 + 3 & 831 & 3.252640e-03 & 2.731240e-03 & 3.645410e-03 & 2.755540e-02\\
 2 + 3 & 1242 & 2.332410e-03 & 1.676960e-03 & 2.609550e-03 & 1.812550e-02\\
 4 + 3 & 2034 & 1.621970e-03 & 1.105830e-03 & 1.839640e-03 & 9.489620e-03\\
 8 + 3 & 3504 & 2.940410e-04 & 2.175700e-04 & 4.219670e-04 & 6.108530e-03\\
 1 + 4 & 1032 & 3.254500e-03 & 2.371670e-03 & 3.594400e-03 & 1.372610e-02\\
\hline
\multicolumn{6}{|c|}{$\delta = 0.1$}
\\
\hline
 1 + 3 & 831 & 8.198060e-03 & 6.405510e-03 & 8.391160e-03 & 3.665250e-02\\
 2 + 3 & 1245 & 6.459500e-03 & 5.427280e-03 & 7.218040e-03 & 2.414800e-02\\
 4 + 3 & 2034 & 6.939110e-04 & 3.268860e-04 & 6.366130e-04 & 6.357700e-03\\
 8 + 3 & 3504 & 2.529180e-04 & 1.253960e-04 & 2.760710e-04 & 4.846080e-03\\
 1 + 4 & 1032 & 6.988820e-03 & 5.844490e-03 & 7.534010e-03 & 2.245580e-02\\
\hline
\end{tabular}
\end{center}
\end{table}

Last, in Table \ref{tab:online_errors_c_case_2_theta_0.8}, we present the errors of the online adaptive multiscale method ($\theta = 0.8$) for the Case 2 of composite Cosserat media. One can see that the accuracy of the adaptive method is mainly comparable to that of the uniform method. For example, using two offline with three online basis functions at $\delta=\infty$, the uniform method yields errors less than 1\% (except for $H^1$ microrotation error). The adaptive method achieves similar accuracy but uses fewer degrees of freedom than the uniform one. As the considered example, the uniform method has 1815 degrees of freedom, and the adaptive method has only 1242. Thus, the proposed multiscale approaches achieve quite accurate approximation of the reference solution for Cosserat composite media with both small and large strain-limiting parameters.

\subsection{Heterogeneous media}\label{heterom}


\begin{figure}[!htbp]
    \centering
    \centering
    \includegraphics[width=0.7\textwidth]{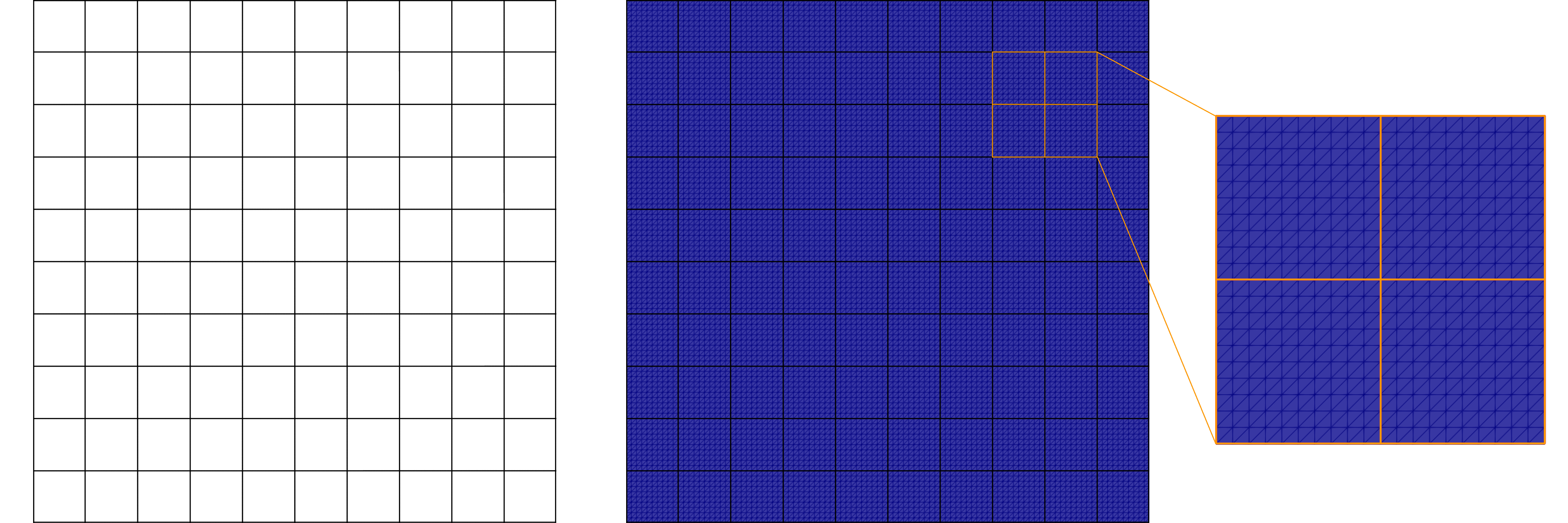}
    \caption{Computational grids for the heterogeneous media: left -- the coarse grid (each small square is a coarse block), middle -- the fine grid, right -- a local domain.}
    \label{fig:h_meshes}
\end{figure}

Finally, we consider heterogeneous Cosserat media with small and large strain-limiting parameters. This is a more complicated media type compared to perforated and composite ones. 
A uniform fine grid is used with 20,000 triangular cells and 10,201 vertices for all problems (see Figure \ref{fig:h_meshes}).


\begin{figure}[!htbp]
    \centering
    \begin{subfigure}[b]{\textwidth}
        \centering
        \includegraphics[width=0.9\textwidth]{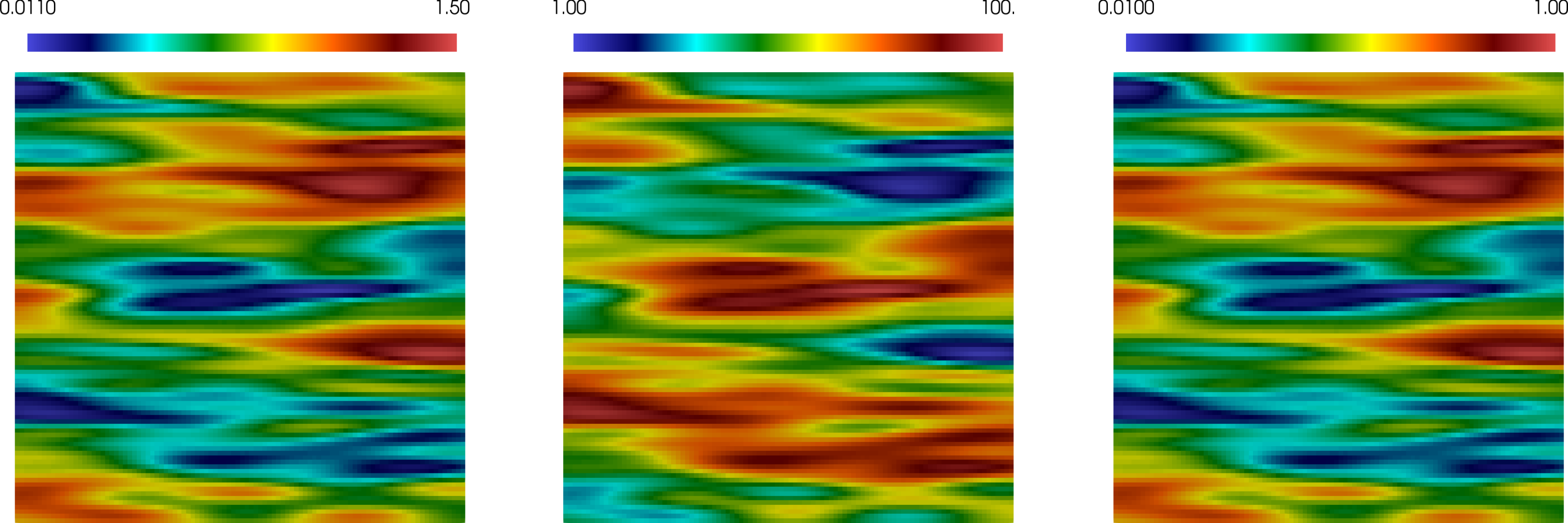}
        \caption{Case 1 (small strain-limiting parameter).}
        \label{fig:h_coefficients_1}
    \end{subfigure}
    \begin{subfigure}[b]{\textwidth}
        \centering
        \includegraphics[width=0.9\textwidth]{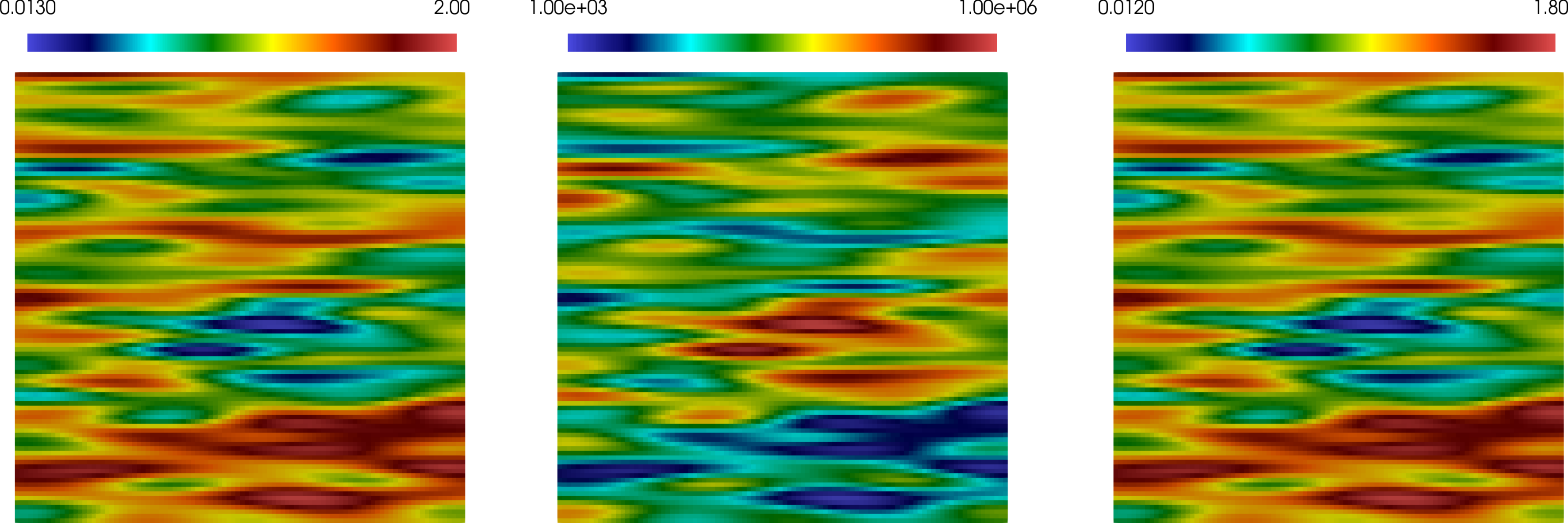}
        \caption{Case 2 (big strain-limiting parameter).}
        \label{fig:h_coefficients_2}
    \end{subfigure}
    \caption{Distributions of $\alpha$, $\beta$, and $\xi$ (from left to right).}
    \label{fig:h_coefficients}
\end{figure}

Figure \ref{fig:h_coefficients} presents the distributions of the heterogeneous coefficients $\alpha$, $\beta$, and $\xi$ (from left to right). In Figure \ref{fig:h_coefficients_1}, we plot the coefficients for Case 1 (small strain-limiting parameter). Meanwhile, Figure \ref{fig:h_coefficients_2} presents the coefficients for Case 2 (large strain-limiting parameter). Regarding the right-hand sides, we have the following settings. For Case 1, set $f_v = 2.2 \cdot 10^{-3}$ and $g_v = 2.3 \cdot 10^{-3}$. Case 2 has $f_v = 5.35 \cdot 10^{-7}$ and $g_v = 5.4 \cdot 10^{-7}$.


\begin{figure}[!htbp]
    \centering
    \begin{subfigure}[b]{\textwidth}
        \centering
        \includegraphics[width=\textwidth]{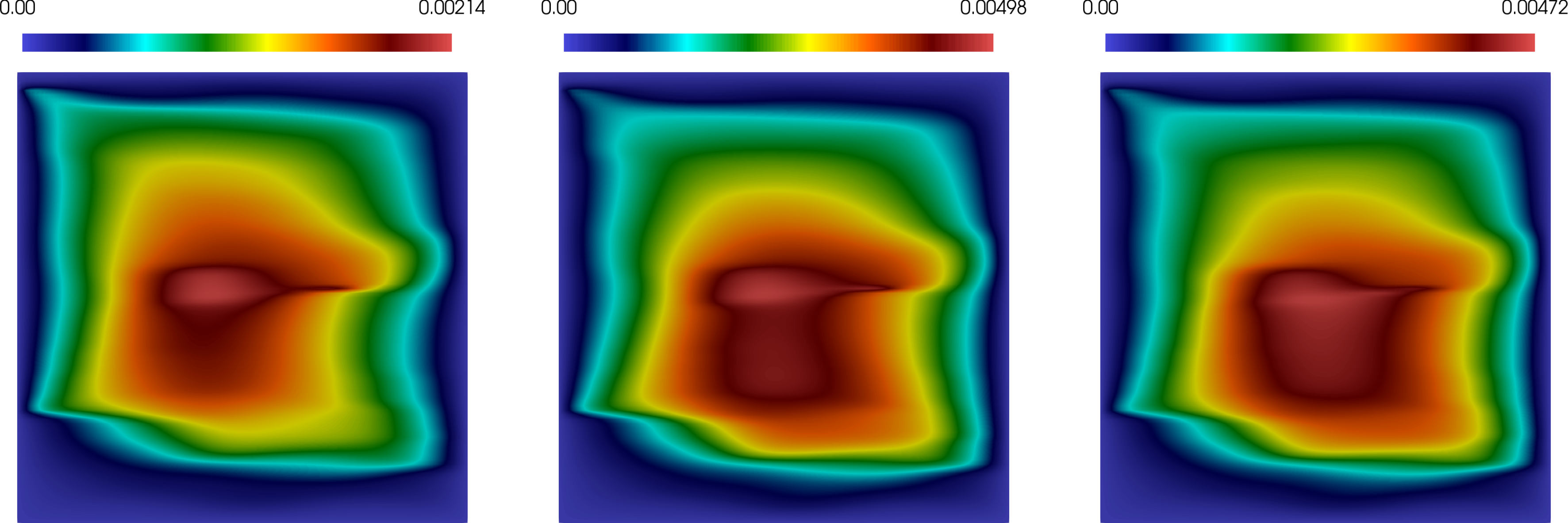}
        \caption{Reference solution.}
        \label{fig:h_results_fine}
    \end{subfigure}
    \begin{subfigure}[b]{\textwidth}
        \centering
        \includegraphics[width=\textwidth]{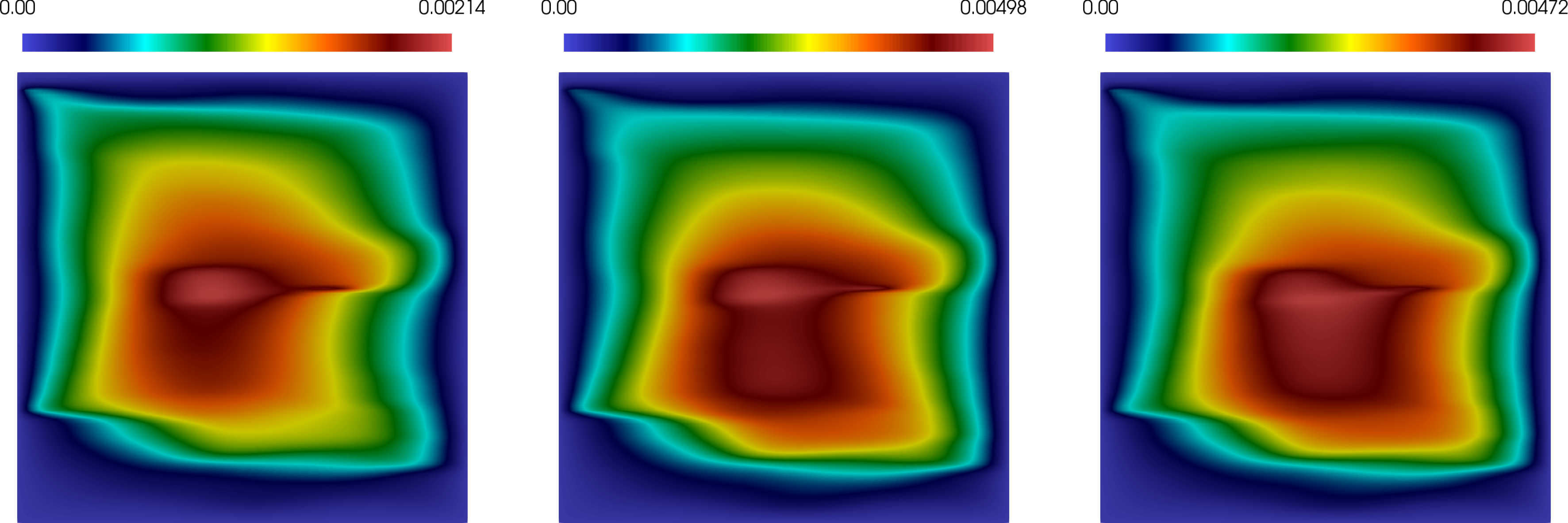}
        \caption{Online adaptive multiscale solution with $N_{\tu{b}} = 4$, $N_{\tu{it}} = 3$, $\theta = 0.8$, and $\delta = 0.1$\,.}
        \label{fig:h_results_online}
    \end{subfigure}
    \begin{subfigure}[b]{\textwidth}
        \centering
        \includegraphics[width=\textwidth]{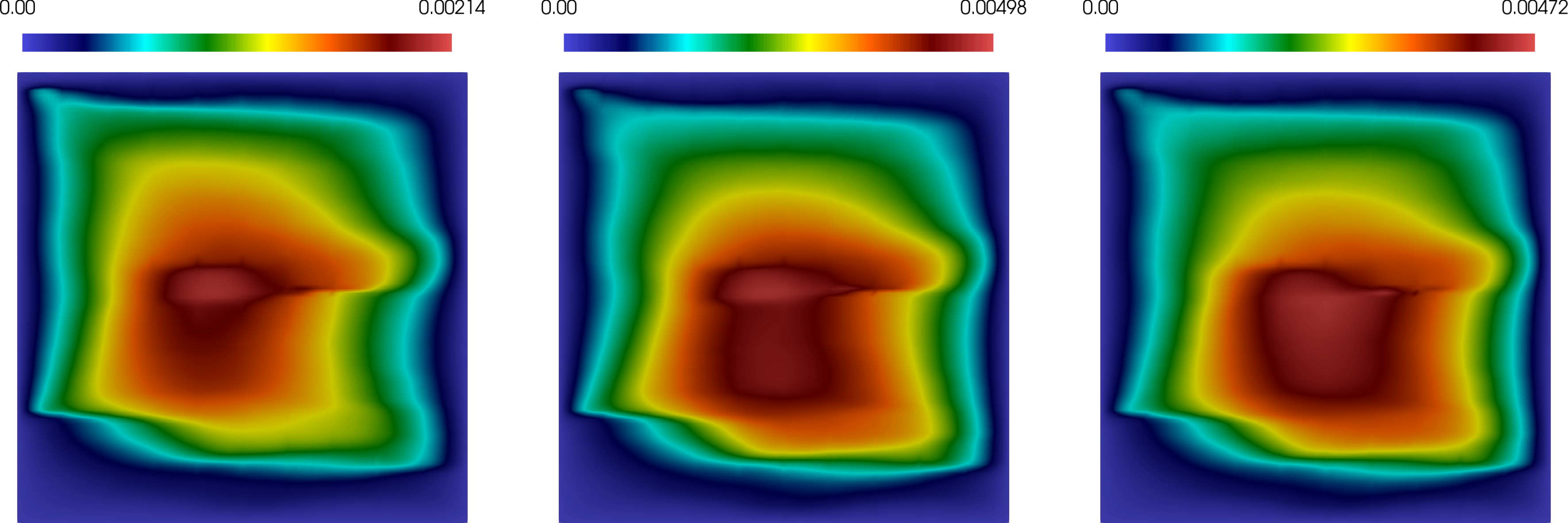}
        \caption{Offline multiscale solution with $N_{\tu{b}} = 4$\,.}
        \label{fig:h_results_offline}
    \end{subfigure}
    \caption{Distributions of the microrotations $\Phi$ and displacements $u_1$ and $u_2$ (from left to right) for Case 1 (small strain-limiting parameter) of the heterogeneous media.}
    \label{fig:h_results}
\end{figure}

In Figure \ref{fig:h_results}, we show the numerical solution to the model problem for the Case 1 of heterogeneous Cosserat media. The distributions of the microrotation $\Phi$ and displacements $u_1$ and $u_2$ are depicted from left to right. Figure \ref{fig:h_results_fine} presents the reference solution on a fine grid. In Figure \ref{fig:h_results_online}, we show the online adaptive solution with $N_{\tu{b}} = 4$, $N_{\tu{it}} = 3$, $\theta = 0.8$, and $\delta = 0.1$. Figure \ref{fig:h_results_offline} demonstrates the offline multiscale solution with $N_{\tu{b}} = 4$. One can see that all figures are very similar, indicating that our proposed multiscale approaches can approximate the reference solution with high accuracy. However, one can observe slight differences between the reference solution and the offline solution.

Also, it is apparent from Figure \ref{fig:h_results} that the distributions of the solution fields reflect heterogeneities in the coefficients $\xi$, $\alpha$, and $\beta$. One can see that the microrotation and displacements are fixed at the boundaries of the domain following our boundary conditions. The most significant values of the solution fields appear in the central part of the computational domain. In general, the obtained numerical solutions correspond to the modeled procedure.



\begin{table}[!htbp]
\caption{Relative errors of the offline method for the case 1 of the heterogeneous media.}
\label{tab:offline_errors_h_case_1}
\begin{center}
\begin{tabular}{ | c | c | c c | c c |}
\hline
\multirowcell{2}{$N_{\tu{b}}$}
& \multirowcell{2}{$\tu{DOF}_{\tu{H}}$}
& \multicolumn{2}{c|}{Displacement}
& \multicolumn{2}{c|}{Microrotation}
\\
\cline{3-6}
&
& $e^{u}_{L^2}$
& $e^{u}_{H^1}$
& $e^{\Phi}_{L^2}$ 
& $e^{\Phi}_{H^1}$ 
\\
\hline
 1   & 363 & 5.419690e-02 & 2.173160e-01 & 6.611790e-02 & 2.273770e-01\\
 2   & 726 & 3.264460e-02 & 1.747270e-01 & 3.368040e-02 & 1.798890e-01\\
 4   & 1452 & 1.116400e-02 & 6.791530e-02 & 1.334260e-02 & 7.371510e-02\\
 8   & 2904 & 5.168200e-03 & 3.711260e-02 & 6.323110e-03 & 4.054100e-02\\
\hline
\end{tabular}
\end{center}
\end{table}

Let us consider the relative errors of the proposed multiscale approaches for the Case 1 of heterogeneous media. Table \ref{tab:offline_errors_h_case_1} presents the errors for the offline multiscale method. The errors are evidently minor. For example, using eight offline basis functions, we can achieve $L^2$ errors less than 1\% and $H^1$ errors less than 5\% for all solution fields. It is clear that as the number of offline basis functions per coarse-grid node increases, the errors decrease. This phenomenon indicates convergence with respect to the number of basis functions.

\begin{table}[!htbp]
\caption{Relative errors of the online method with $\theta = 1$ for the case 1 of the heterogeneous media.}
\label{tab:online_errors_h_case_1_theta_1}
\begin{center}
\begin{tabular}{ | c | c | c c | c c | }
\hline
\multirowcell{2}{$N_{\tu{b}} + N_{\tu{it}}$}
& \multirowcell{2}{$\tu{DOF}_{\tu{H}}$}
& \multicolumn{2}{c|}{Displacement}
& \multicolumn{2}{c|}{Microrotation}
\\
\cline{3-6}
&
& $e^{u}_{L^2}$
& $e^{u}_{H^1}$
& $e^{\Phi}_{L^2}$ 
& $e^{\Phi}_{H^1}$ 
\\
\hline
\multicolumn{6}{|c|}{$\delta = \infty$}
\\
\hline
 1 + 3 & 1452 & 4.383800e-03 & 3.522840e-02 & 4.866030e-03 & 3.699580e-02\\
 2 + 3 & 1815 & 3.891300e-03 & 3.288900e-02 & 4.243790e-03 & 3.565550e-02\\
 4 + 3 & 2541 & 2.940850e-03 & 2.847530e-02 & 3.311570e-03 & 3.198390e-02\\
 8 + 3 & 3993 & 1.836940e-03 & 2.252670e-02 & 1.916790e-03 & 2.366840e-02\\
 1 + 4 & 1815 & 4.120080e-03 & 3.424830e-02 & 4.385240e-03 & 3.557660e-02\\
\hline
\multicolumn{6}{|c|}{$\delta = 0.4$}
\\
\hline
 1 + 3 & 1452 & 4.440450e-03 & 3.931560e-02 & 3.818960e-03 & 3.693580e-02\\
 2 + 3 & 1815 & 2.309330e-03 & 3.109220e-02 & 2.257410e-03 & 3.126000e-02\\
 4 + 3 & 2541 & 9.869050e-04 & 1.491160e-02 & 1.039230e-03 & 1.588080e-02\\
 8 + 3 & 3993 & 4.411120e-04 & 7.751330e-03 & 4.776120e-04 & 8.655500e-03\\
 1 + 4 & 1815 & 3.213520e-03 & 3.375880e-02 & 2.782390e-03 & 3.152320e-02\\
\hline
\multicolumn{6}{|c|}{$\delta = 0.2$}
\\
\hline
 1 + 3 & 1452 & 3.645930e-03 & 3.110550e-02 & 3.748050e-03 & 3.144440e-02\\
 2 + 3 & 1815 & 2.785150e-03 & 2.761440e-02 & 2.789900e-03 & 3.012630e-02\\
 4 + 3 & 2541 & 2.243970e-03 & 2.459300e-02 & 2.446930e-03 & 2.786640e-02\\
 8 + 3 & 3993 & 1.313930e-03 & 1.859830e-02 & 1.320570e-03 & 2.003710e-02\\
 1 + 4 & 1815 & 3.435690e-03 & 3.030630e-02 & 3.467350e-03 & 3.079680e-02\\
\hline
\multicolumn{6}{|c|}{$\delta = 0.1$}
\\
\hline
 1 + 3 & 1452 & 2.025480e-03 & 2.076980e-02 & 1.554510e-03 & 1.943370e-02\\
 2 + 3 & 1815 & 1.043430e-03 & 1.619310e-02 & 1.032300e-03 & 1.684330e-02\\
 4 + 3 & 2541 & 5.063580e-04 & 9.255770e-03 & 5.355140e-04 & 1.039420e-02\\
 8 + 3 & 3993 & 2.705430e-04 & 5.974960e-03 & 2.946030e-04 & 6.856720e-03\\
 1 + 4 & 1815 & 1.517970e-03 & 1.790020e-02 & 1.286640e-03 & 1.715280e-02\\
\hline
\end{tabular}
\end{center}
\end{table}

Next, let us consider the relative errors of the online uniform method ($\theta=1$) for the Case 1 of heterogeneous media, listed in Table \ref{tab:online_errors_h_case_1_theta_1}. The application of online basis functions has significantly improved the accuracy of the multiscale solution. For example, we can achieve $L^2$ errors of less than 1\% and $H^1$ errors of 4\% for all fields using one offline and three online basis functions with $\delta=\infty$. A more significant reduction in $H^1$ error requires more offline basis functions and frequent updates. One can observe that, on average, the errors decrease as the offline basis functions number increases and $\delta$ decreases.

\begin{table}[!htbp]
\caption{Relative errors of the online method with $\theta = 0.8$ for the case 1 of the heterogeneous media.}
\label{tab:online_errors_h_case_1_theta_0.8}
\begin{center}
\begin{tabular}{ | c | c | c c | c c | }
\hline
\multirowcell{2}{$N_{\tu{b}} + N_{\tu{it}}$}
& \multirowcell{2}{$\tu{DOF}_{\tu{H}}$}
& \multicolumn{2}{c|}{Displacement}
& \multicolumn{2}{c|}{Microrotation}
\\
\cline{3-6}
&
& $e^{u}_{L^2}$
& $e^{u}_{H^1}$
& $e^{\Phi}_{L^2}$ 
& $e^{\Phi}_{H^1}$ 
\\
\hline
\multicolumn{6}{|c|}{$\delta = \infty$}
\\
\hline
 1 + 3 & 924 & 5.673020e-03 & 3.962000e-02 & 6.434470e-03 & 4.145830e-02\\
 2 + 3 & 1305 & 4.364190e-03 & 3.569630e-02 & 4.860200e-03 & 3.840520e-02\\
 4 + 3 & 1995 & 3.578610e-03 & 3.184730e-02 & 3.936960e-03 & 3.451640e-02\\
 8 + 3 & 3477 & 2.006330e-03 & 2.378650e-02 & 2.107570e-03 & 2.546960e-02\\
 1 + 4 & 1092 & 4.731540e-03 & 3.708350e-02 & 5.170640e-03 & 3.866460e-02\\
\hline
\multicolumn{6}{|c|}{$\delta = 0.4$}
\\
\hline
 1 + 3 & 924 & 7.419340e-03 & 5.283190e-02 & 6.778140e-03 & 5.395420e-02\\
 2 + 3 & 1305 & 3.662770e-03 & 4.008400e-02 & 3.360560e-03 & 4.038750e-02\\
 4 + 3 & 1995 & 1.341400e-03 & 1.808080e-02 & 1.357920e-03 & 1.887000e-02\\
 8 + 3 & 3477 & 4.921400e-04 & 8.743900e-03 & 5.444320e-04 & 9.740430e-03\\
 1 + 4 & 1092 & 6.384480e-03 & 4.942240e-02 & 5.640400e-03 & 5.021520e-02\\
\hline
\multicolumn{6}{|c|}{$\delta = 0.2$}
\\
\hline
 1 + 3 & 924 & 4.101040e-03 & 3.419680e-02 & 4.245310e-03 & 3.423250e-02\\
 2 + 3 & 1305 & 3.077750e-03 & 2.967090e-02 & 3.195300e-03 & 3.229150e-02\\
 4 + 3 & 2001 & 3.029550e-03 & 2.886550e-02 & 3.183790e-03 & 3.084220e-02\\
 8 + 3 & 3477 & 1.438160e-03 & 1.941150e-02 & 1.459980e-03 & 2.110530e-02\\
 1 + 4 & 1092 & 3.863460e-03 & 3.261860e-02 & 3.825900e-03 & 3.242170e-02\\
\hline
\multicolumn{6}{|c|}{$\delta = 0.1$}
\\
\hline
 1 + 3 & 933 & 2.758910e-03 & 2.463040e-02 & 2.513950e-03 & 2.478490e-02\\
 2 + 3 & 1305 & 1.557010e-03 & 1.973010e-02 & 1.533950e-03 & 1.984750e-02\\
 4 + 3 & 2001 & 6.361230e-04 & 1.073310e-02 & 6.772690e-04 & 1.216150e-02\\
 8 + 3 & 3477 & 3.327400e-04 & 6.678640e-03 & 3.408500e-04 & 7.354840e-03\\
 1 + 4 & 1095 & 2.612870e-03 & 2.286920e-02 & 2.320680e-03 & 2.220310e-02\\
\hline
\end{tabular}
\end{center}
\end{table}

Table \ref{tab:online_errors_h_case_1_theta_0.8} describes the errors of the online adaptive multiscale method ($\theta=0.8$) for the Case 1 of heterogeneous media. We see that the adaptive method achieves comparable accuracy to the uniform method, for example, with $N_{\tu{b}}=1$, $N_{\tu{it}}=3$, and $\delta=\infty\,.$ 
However, the adaptive method requires fewer degrees of freedom ($\tu{DOF}_{\tu{H}} = 924$) than the uniform method ($\tu{DOF}_{\tu{H}} = 1452$).


\begin{table}[!htbp]
\caption{Relative errors of the offline method for the case 2 of the heterogeneous media.}
\label{tab:offline_errors_h_case_2}
\begin{center}
\begin{tabular}{ | c | c | c c | c c |}
\hline
\multirowcell{2}{$N_{\tu{b}}$}
& \multirowcell{2}{$\tu{DOF}_{\tu{H}}$}
& \multicolumn{2}{c|}{Displacement}
& \multicolumn{2}{c|}{Microrotation}
\\
\cline{3-6}
&
& $e^{u}_{L^2}$
& $e^{u}_{H^1}$
& $e^{\Phi}_{L^2}$ 
& $e^{\Phi}_{H^1}$ 
\\
\hline
 1   & 363 & 6.919310e-02 & 2.303840e-01 & 7.408850e-02 & 2.337440e-01\\
 2   & 726 & 4.170070e-02 & 1.803740e-01 & 4.924700e-02 & 1.865870e-01\\
 4   & 1452 & 3.462540e-02 & 1.057580e-01 & 4.076340e-02 & 1.074750e-01\\
 8   & 2904 & 2.512890e-02 & 6.930720e-02 & 3.071440e-02 & 7.207860e-02\\
\hline
\end{tabular}
\end{center}
\end{table}

We briefly consider the Case 2 of heterogeneous media. Table \ref{tab:offline_errors_h_case_2} shows the relative errors for the offline multiscale method. In general, the errors are minor, but they are somehow larger than for Case 1 (Table \ref{tab:offline_errors_h_case_1}). One can explain this phenomenon by the more significant nonlinearity of the problem due to the big strain-limiting parameter.

\begin{table}[!htbp]
\caption{Relative errors of the online method with $\theta = 1$ for the case 2 of the heterogeneous media.}
\label{tab:online_errors_h_case_2_theta_1}
\begin{center}
\begin{tabular}{ | c | c | c c | c c | }
\hline
\multirowcell{2}{$N_{\tu{b}} + N_{\tu{it}}$}
& \multirowcell{2}{$\tu{DOF}_{\tu{H}}$}
& \multicolumn{2}{c|}{Displacement}
& \multicolumn{2}{c|}{Microrotation}
\\
\cline{3-6}
&
& $e^{u}_{L^2}$
& $e^{u}_{H^1}$
& $e^{\Phi}_{L^2}$ 
& $e^{\Phi}_{H^1}$ 
\\
\hline
\multicolumn{6}{|c|}{$\delta = \infty$}
\\
\hline
 1 + 3 & 1452 & 1.304680e-02 & 8.310370e-02 & 1.356710e-02 & 7.892670e-02\\
 2 + 3 & 1815 & 1.270610e-02 & 8.345240e-02 & 1.293110e-02 & 7.801780e-02\\
 4 + 3 & 2541 & 8.771970e-03 & 6.713950e-02 & 9.300600e-03 & 6.501380e-02\\
 8 + 3 & 3993 & 6.115750e-03 & 5.141410e-02 & 6.915440e-03 & 5.079360e-02\\
 1 + 4 & 1815 & 1.199220e-02 & 8.040890e-02 & 1.326950e-02 & 7.540530e-02\\
\hline
\multicolumn{6}{|c|}{$\delta = 0.4$}
\\
\hline
 1 + 3 & 1452 & 1.240920e-02 & 2.147460e-02 & 1.538920e-02 & 2.537460e-02\\
 2 + 3 & 1815 & 1.232700e-02 & 2.107880e-02 & 1.505190e-02 & 2.488050e-02\\
 4 + 3 & 2541 & 1.062390e-02 & 1.944770e-02 & 1.273210e-02 & 2.330430e-02\\
 8 + 3 & 3993 & 7.977310e-03 & 1.544880e-02 & 1.045970e-02 & 1.868190e-02\\
 1 + 4 & 1815 & 1.179330e-02 & 2.045310e-02 & 1.472660e-02 & 2.416850e-02\\
\hline
\multicolumn{6}{|c|}{$\delta = 0.2$}
\\
\hline
 1 + 3 & 1452 & 3.542250e-03 & 1.492340e-02 & 4.053120e-03 & 1.812700e-02\\
 2 + 3 & 1815 & 3.303210e-03 & 1.604530e-02 & 4.135830e-03 & 1.983060e-02\\
 4 + 3 & 2541 & 1.617330e-03 & 1.260760e-02 & 2.078650e-03 & 1.520880e-02\\
 8 + 3 & 3993 & 9.925260e-04 & 8.658330e-03 & 1.299360e-03 & 1.106480e-02\\
 1 + 4 & 1815 & 3.385520e-03 & 1.388360e-02 & 4.076770e-03 & 1.715350e-02\\
\hline
\multicolumn{6}{|c|}{$\delta = 0.1$}
\\
\hline
 1 + 3 & 1452 & 5.204710e-03 & 1.134430e-02 & 6.150940e-03 & 1.400520e-02\\
 2 + 3 & 1815 & 5.158760e-03 & 1.162310e-02 & 6.059750e-03 & 1.430720e-02\\
 4 + 3 & 2541 & 3.127480e-03 & 1.061640e-02 & 3.463800e-03 & 1.257570e-02\\
 8 + 3 & 3993 & 2.033680e-03 & 6.791850e-03 & 2.563690e-03 & 8.353520e-03\\
 1 + 4 & 1815 & 5.022030e-03 & 1.055400e-02 & 6.146010e-03 & 1.332760e-02\\
\hline
\end{tabular}
\end{center}
\end{table}

Table \ref{tab:online_errors_h_case_2_theta_1} presents the errors of the online uniform multiscale method ($\theta = 1$) for the Case 2 of heterogeneous media. Obviously, the online enrichment of the multiscale space significantly reduces the errors of the numerical solution, in comparison to the offline method. For example, using one offline and three online basis functions with $\delta = 0.1$, one can achieve $L^2$ errors less than 1\% and $H^1$ errors less than 2\% for all solution fields. We also observe a decrease in the errors when the number of offline basis functions increases and $\delta$ decreases.

\begin{table}[!htbp]
\caption{Relative errors of the online method with $\theta = 0.8$ for the case 2 of the heterogeneous media.}
\label{tab:online_errors_h_case_2_theta_0.8}
\begin{center}
\begin{tabular}{ | c | c | c c | c c | }
\hline
\multirowcell{2}{$N_{\tu{b}} + N_{\tu{it}}$}
& \multirowcell{2}{$\tu{DOF}_{\tu{H}}$}
& \multicolumn{2}{c|}{Displacement}
& \multicolumn{2}{c|}{Microrotation}
\\
\cline{3-6}
&
& $e^{u}_{L^2}$
& $e^{u}_{H^1}$
& $e^{\Phi}_{L^2}$ 
& $e^{\Phi}_{H^1}$ 
\\
\hline
\multicolumn{6}{|c|}{$\delta = \infty$}
\\
\hline
 1 + 3 & 903 & 1.678830e-02 & 9.209500e-02 & 1.753610e-02 & 8.716010e-02\\
 2 + 3 & 1311 & 1.421190e-02 & 8.963580e-02 & 1.454010e-02 & 8.488830e-02\\
 4 + 3 & 1875 & 1.087310e-02 & 7.408720e-02 & 1.159350e-02 & 7.138570e-02\\
 8 + 3 & 3246 & 7.131570e-03 & 5.584860e-02 & 7.758890e-03 & 5.521970e-02\\
 1 + 4 & 1068 & 1.489030e-02 & 8.842950e-02 & 1.584340e-02 & 8.389720e-02\\
\hline
\multicolumn{6}{|c|}{$\delta = 0.4$}
\\
\hline
 1 + 3 & 906 & 1.288540e-02 & 2.498310e-02 & 1.596110e-02 & 2.918940e-02\\
 2 + 3 & 1311 & 1.273660e-02 & 2.411180e-02 & 1.548070e-02 & 2.790280e-02\\
 4 + 3 & 2013 & 1.055560e-02 & 2.056260e-02 & 1.308720e-02 & 2.473470e-02\\
 8 + 3 & 3459 & 8.095620e-03 & 1.639740e-02 & 1.052910e-02 & 2.020620e-02\\
 1 + 4 & 1068 & 1.270700e-02 & 2.324870e-02 & 1.579960e-02 & 2.757000e-02\\
\hline
\multicolumn{6}{|c|}{$\delta = 0.2$}
\\
\hline
 1 + 3 & 906 & 4.121940e-03 & 2.042730e-02 & 4.498640e-03 & 2.375570e-02\\
 2 + 3 & 1311 & 3.664080e-03 & 2.017230e-02 & 4.228240e-03 & 2.282930e-02\\
 4 + 3 & 2016 & 1.821150e-03 & 1.400640e-02 & 2.311220e-03 & 1.713640e-02\\
 8 + 3 & 3459 & 1.061690e-03 & 9.518380e-03 & 1.364670e-03 & 1.192980e-02\\
 1 + 4 & 1068 & 3.697350e-03 & 1.725500e-02 & 4.298040e-03 & 2.090630e-02\\
\hline
\multicolumn{6}{|c|}{$\delta = 0.1$}
\\
\hline
 1 + 3 & 906 & 5.639990e-03 & 1.678090e-02 & 6.509760e-03 & 1.897330e-02\\
 2 + 3 & 1311 & 5.214250e-03 & 1.536110e-02 & 6.051070e-03 & 1.823880e-02\\
 4 + 3 & 2016 & 2.968060e-03 & 1.212170e-02 & 3.551270e-03 & 1.457640e-02\\
 8 + 3 & 3459 & 2.036360e-03 & 7.914200e-03 & 2.580780e-03 & 1.030360e-02\\
 1 + 4 & 1068 & 5.217510e-03 & 1.262830e-02 & 6.209310e-03 & 1.583760e-02\\
\hline
\end{tabular}
\end{center}
\end{table}

Finally, in Table \ref{tab:online_errors_h_case_2_theta_0.8}, we present the errors of the online adaptive multiscale method ($\theta = 0.8$) for the Case 2 of heterogeneous Cosserat media. It is observable that the adaptive approach yields errors comparable to the uniform one. However, the adaptive approach has fewer degrees of freedom. For example, with the settings discussed above ($N_{\tu{b}} = 1$, $N_{\tu{it}} = 3$, and $\delta = 0.1$), the uniform method has $\tu{DOF}_{\tu{H}} = 1452$, while the adaptive method has $\tu{DOF}_{\tu{H}} = 906$. Thus, the proposed multiscale approaches can be successfully applied to complex heterogeneous Cosserat media with small and large strain-limiting parameters.

\begin{remark}\label{rmk_constraints}
In the presented results, the source terms and boundary conditions were numerically chosen to satisfy 
$Q(\bfa{\chi},\bfa{\gamma}) < 1$ \eqref{less1}. In general, it would be interesting to solve the optimization problem with the constraint $Q(\bfa{\chi},\bfa{\gamma}) < 1\,.$ We plan to do this in the future. However, in this work, the primary goals are to develop a strain-limiting Cosserat elasticity model and multiscale approaches based on the GMsFEM for efficient simulation.
\end{remark}

\begin{remark}\label{rmk_dualnorm}
The computation of the dual norm plays an important role in the online adaptive algorithm. In Section \ref{dualnorm}, we described two approaches for computing this norm. The numerical results were obtained using the second approach because it showed better results in preliminary experiments. This is because such a dual norm can better account for various features of the nonlinear Cosserat model.
\end{remark}

\section{Conclusion}\label{conclusion}

In this paper, we have explored some multiscale methods to tackle a nonlinear Cosserat elasticity model in media with high contrast and heterogeneities. More specifically, we investigate the application of generalized multiscale finite element method (GMsFEM) to the solving of an isotropic Cosserat problem that has strain-limiting property (which guarantees confined linearized strains even in the presence of extremely large stresses).  
Mathematically, the model forms a system of nonlinear partial differential equations for the Cosserat's microrotations and displacements. The system's nonlinearity is handled by Picard iteration.  At each Picard iterative step, we employ the offline or residual-based online (adaptive or uniform) GMsFEM.

Numerically, we have presented finite element approximation on a fine computational grid. The numerical tests are carried out through a variety of two-dimensional cases (perforated, composite, and stochastically heterogeneous media, with small and large strain-limiting parameters).  Our numerical results demonstrate the robustness, convergence, and efficiency of the techniques that can approximate the reference solution with high accuracy. 
  Furthermore, the online GMsFEM provides more accurate solutions than the offline one, and the online adaptive strategy has similar accuracy to the uniform one but with fewer degrees of freedom. 



The proposed nonlinear Cosserat model and multiscale approaches can be used for efficiently modeling heterogeneous materials with rotational degrees of freedom. Moreover, the strain-limiting theory allows us to consider intense loads.


\newpage

\noindent \textbf{Acknowledgements.} 

The research of Dmitry Ammosov was funded by the Russian Science Foundation Grant No. 23-71-30013 (\href{https://rscf.ru/en/project/23-71-30013/}{https://rscf.ru/en/project/23-71-30013/}) and the Grant of the Head of the Republic of Sakha (Yakutia), Russia, agreement from 21.04.2023 No. 02/928.

Tina Mai expresses her gratitude to all the valuable support from Duy Tan University, who is going to celebrate its 30th anniversary of establishment (Nov. 11, 1994 -- Nov. 11, 2024) towards ``Integral, Sustainable and Stable Development''. 

Juan Galvis thanks the MATHDATA - AUIP Network ({Red Iberoamericana de Investigaci\'on en Matem\'aticas Aplicadas a Datos} \href{https://www.mathdata.science/}{https://www.mathdata.science/}), the Center for Excellence in Scientific Computing (Centro de Excelencia en Computaci\'on Cient\'ifica) of the Universidad Nacional de Colombia.

We express sincere gratitude to Professors Kumbakonam Rajagopal and Jay Walton for their invaluable guidance on the strain-limiting Cosserat model and their encouraging comments.

\bibliographystyle{plain} 
\bibliography{r0,r1,r2}

\end{document}